\documentclass[12pt,letterpaper]{amsart}

\usepackage[english]{babel}
\usepackage[autostyle]{csquotes}  

\usepackage{amsmath}      
\usepackage{amssymb}      
\usepackage{amsthm}       

\usepackage{graphicx}     
\usepackage{xcolor}       
\usepackage{tikz}         

\usepackage{varioref}     
\usepackage[hypertexnames=false]{hyperref}  

\usepackage{float}        
\usepackage{enumitem}     
\usepackage{comment}      
\usepackage{microtype}    

\hypersetup{
    colorlinks=true,
    citecolor=green,
    linkcolor=blue,
    urlcolor=cyan,
    pdftitle={Your Title},
    pdfauthor={Your Name}
}

\makeatletter
\@namedef{subjclassname@2020}{\textup{2020} Mathematics Subject Classification}
\makeatother

\newtheorem{theo}{Theorem}
\newtheorem{defi}{Definition}

\newtheorem{prop}{Proposition}
\newtheorem{coro}{Corollary}
\newtheorem{lemm}{Lemma}

\newtheorem{rema}{Remark}

\newtheorem*{theo*}{Theorem}
\newtheorem*{defi*}{Definition}
\newtheorem*{conv*}{Convention}
\newtheorem*{cons*}{Construction}
\newtheorem*{prop*}{Proposition}
\newtheorem*{coro*}{Corollary}
\newtheorem*{lemm*}{Lemma}
\newtheorem*{conj*}{Conjecture}
\newtheorem*{ques*}{Question}
\newtheorem*{prob*}{Problem}
\newtheorem*{exer*}{Exercise}
\newtheorem*{exem*}{Example}
\newtheorem*{rema*}{Remark}
\newtheorem*{claim*}{Claim}

\numberwithin{equation}{section}

  \def\II{\mathbb{I}}
  
 \def\NN{\mathbb{N}} 
  \def\RR{\mathbb{R}}
 \def\TT{\mathbb{T}} 
  
 \def\ZZ{\mathbb{Z}}

  \def\cB{\mathcal{B}}  \def\cC{\mathcal{C}}  
  \def\cF{\mathcal{F}}  \def\cG{\mathcal{G}}  \def\cH{\mathcal{H}}
\def\cI{\mathcal{I}}  \def\cJ{\mathcal{J}}    
\def\cM{\mathcal{M}}      
  \def\cR{\mathcal{R}}  \def\cS{\mathcal{S}}  \def\cT{\mathcal{T}}
\def\cU{\mathcal{U}}  \def\cV{\mathcal{V}}

 \def\be{\mathbf{e}}

\def\bp{\mathbf{p}}  
  \def\bu{\mathbf{u}}
\def\bv{\mathbf{v}} \def\bw{\mathbf{w}} \def\bx{\mathbf{x}}

	 \def\bB{\mathbf{B}} 
  
 \def\bH{\mathbf{H}}

  \def\bR{\mathbf{R}}
 \def\bT{\mathbf{T}} 
\def\bV{\mathbf{V}}

 \def\be{\mathbf{e}}

\def\bp{\mathbf{p}}  
  \def\bu{\mathbf{u}}
\def\bv{\mathbf{v}} \def\bw{\mathbf{w}} \def\bx{\mathbf{x}}

\title[ A complete conjugacy invariant]{A classification of pseudo-Anosov homeomorphisms I: The geometric type is a complete conjugacy invariant}
\author{Inti Cruz}
\address{Facultad de Contaduría y Administración, UNAM, Oaxaca, Mexico}
\email{incruzd@gmail.com}
\date{\today}


\begin{document}

\begin{abstract}
	Every  pseudo-Anosov homeomorphism $f$ admits infinitely many Markov partitions. A \textit{geometric Markov partition} is a Markov partition $\mathcal{R}$ in which each rectangle is equipped with a vertical orientation. To each pair $(f, \mathcal{R})$, consisting of a pseudo-Anosov homeomorphism $f$ and a geometric Markov partition $\mathcal{R}$, there is a naturally associated combinatorial object called its \textit{geometric type} $\bT(f, \mathcal{R})$.
	
	We prove, using symbolic dynamics, that two  pseudo-Anosov homeomorphisms are topologically conjugate via an orientation-preserving homeomorphism if and only if they admit geometric Markov partitions with the same geometric type. This result lays the groundwork for the algorithmic classification we will develop in subsequent work.
\end{abstract}

\maketitle

\tableofcontents

\section{Introduction}

Let $S$ be a smooth, closed, and orientable surface, and let $\mathrm{Hom}_+(S)$ denote the group of orientation-preserving self-homeomorphisms of $S$, with the composition as the group operation. There are two classical approaches to classifying these maps: up to topological conjugacy, as proposed by S. Smale in \cite{smale1967differentiable}, and up to isotopy, as initiated by W. Thurston in \cite{thurston1988geometry}. Each of these perspectives constitutes a foundational chapter in the development of mathematics in the 20\textsuperscript{th} and 21\textsuperscript{st} centuries.

The hyperbolic theory of dynamical systems was revolutionized by the work of S. Smale on structurally stable diffeomorphisms and by V. Anosov’s study of geodesic flows on negatively curved manifolds \cite{anosov1969geodesic}.  
According to some authors (see \cite{hasselblatt2002handbook}), it was either R. Thom or V. Arnold who introduced the so-called Arnold cat map: a diffeomorphism $f_A\colon \TT^2 \to \TT^2$ defined as the quotient of the linear action of the matrix  
$$
A = \begin{pmatrix} 2 & 1 \\ 1 & 1 \end{pmatrix}
$$  
on $\RR^2$, under the equivalence relation induced by the standard lattice $\ZZ \times \ZZ$.  
The map $f_A$ possesses an infinite number of periodic orbits, is topologically transitive, structurally stable, and notably preserves a pair of transverse, non-singular measured foliations (see \cite{franks1970anosov}).

This is an example of a more general construction: a $2 \times 2$ integer matrix $A$ with $\det(A) = 1$ and no eigenvalues of modulus $1$ is called a \emph{hyperbolic matrix}. Such matrices preserve the lattice $\ZZ \times \ZZ$ and descend to the quotient, yielding an \emph{Anosov diffeomorphism} $f_A\colon \TT^2 \to \TT^2$ that, like the cat map, preserves a pair of transverse, non-singular measured foliations.

In the 1980s, W. Thurston studied the group $\mathrm{Hom}_+(S)$ up to isotopy in \cite{thurston1988geometry}, introducing the notion of pseudo-Anosov homeomorphisms (see~\ref{Defi: pseudo-Anosov}) as a natural generalization of Anosov diffeomorphisms on the torus . A pseudo-Anosov homeomorphism is an orientation-preserving homeomorphism of a closed surface that preserves a pair of transverse measured foliations, which are uniformly contracted and expanded by the actions of $f$ and $f^{-1}$, respectively. Unlike toral Anosov diffeomorphisms, the invariant foliations of a pseudo-Anosov map may exhibit $k$-prong singularities for $k \geq 3$. This concept was later expanded and presented in detail in the collaborative exposition by A. Fathi, F. Laudenbach, and V. Poénaru \cite{fathi2021thurston}, which also allows for $1$-prong singularities. We adopt this broader definition in Definition~\ref{Defi: pseudo-Anosov}, closely following the exposition by B. Farb and D. Margalit in \cite{farb2011primer}. We refer the reader to their book for a comprehensive treatment of the classical theory.

The classification theorem of M. Handel and W. Thurston ( \cite{thurston1988geometry},\cite{handel1985new} states that every orientation-preserving homeomorphism of a closed surface is, up to isotopy, either periodic (some power is isotopic to the identity), pseudo-Anosov, or reducible. In the reducible case, the surface can be decomposed along a finite collection of disjoint, non-null-homotopic simple closed curves into subsurfaces on which the restriction of the homeomorphism is isotopic to one of the first two types. Later, M. Bestvina and M. Handel \cite{bestvina1995train} provided an algorithmic proof of the Thurston–Handel classification theorem based on the theory of train tracks. 

In this paper, we are interested in the classification of pseudo-Anosov homeomorphisms, a class that emerged from Thurston's classification theory up to topological conjugacy, in the spirit of Smale's school of dynamical systems. The mapping class group  $\cM\cC\cG(S)$ of a surface $S$ is the group of isotopy classes of orientation-preserving homeomorphisms of $S$. It is well known (see \cite[Exposé 12]{fathi2021thurston}) that two isotopic pseudo-Anosov homeomorphisms are topologically conjugate via a homeomorphism that is itself isotopic to the identity.

Therefore, classifying the isotopy classes that contain a pseudo-Anosov representative yields a classification of pseudo-Anosov homeomorphisms up to topological conjugacy by homeomorphisms isotopic to the identity. To be precise, such a classification must associate to each isotopy class a combinatorial object that determines whether the class is pseudo-Anosov, whether two pseudo-Anosov classes are conjugate, and which combinatorial invariants can be realized by pseudo-Anosov isotopy classes. Using tools such as train tracks and cellular decompositions, Thurston's school laid important groundwork toward the classification of pseudo-Anosov homeomorphisms. However, these approaches have limited effectiveness in addressing the conjugacy problem—namely, determining when two isotopy classes are topologically conjugate. In this context, we highlight the work of L. Mosher, who provided partial answers to the conjugacy problem in \cite{mosher1983pseudo, mosher1986classification}. We also mention the algorithm of J. Los \cite{los1993pseudo}, as well as the algorithm developed by M. Bestvina and M. Handel \cite{bestvina1995train}, which, starting from a decomposition of an isotopy class into Dehn twists, determines whether it is isotopic to a pseudo-Anosov, periodic, or reducible homeomorphism.More recent developments include the computational implementation of such algorithms by M. Bell in the SageMath program \emph{Flipper} \cite{bell2015recognising}.

We adopt a completely different approach: we make use of the theory of geometric Markov partitions (\ref{Defi: Geometric Markov partition}) and their associated geometric types to construct a family of complete invariants of conjugacy for the class of pseudo-Anosov homeomorphisms. We then address the realization and equivalence problems for such invariants. Markov partitions are a classical tool in the study of hyperbolic dynamics, as they allow one to encode the behavior of orbits of the original system using a finite set of symbols. Important applications to the ergodic theory of Anosov diffeomorphisms were developed by R. Bowen through the use of Markov partitions and their associated incidence matrices \cite{BowenAnosovbook}. A detailed exposition on symbolic dynamics, relevant to our present research, is available in \cite{MarcusLindSynbolic} The notion of \emph{geometric Markov partition} was first introduced by C. Bonatti and R. Langevin in \cite{bonatti1998diffeomorphismes} in their study of non-trivial saddle-type hyperbolic sets of $C^1$ structurally stable diffeomorphisms of surfaces (see \cite{hasselblatt2002handbook} for background on hyperbolic dynamics), with the goal of classifying their invariant neighborhoods. This project was later completed for basic pieces through the subsequent work of F. Béguin \cite{beguin1999champs, beguin2002classification, beguin2004smale}, who introduced a combinatorial object associated to each geometric Markov partition: its geometric type.

In a nutshell, a geometric Markov partition for a pseudo-Anosov homeomorphism $f \colon S \to S$ is a Markov partition $\mathcal{R} = \{R_i\}_{i=1}^{n}$ of $f$, where each rectangle is endowed with an orientation on its stable and unstable foliations. The combined orientation of these foliations is required to be coherent with the orientation of the surface $S$. The positive orientation of the unstable foliation is referred to as the vertical direction of the rectangle, and, correspondingly, the stable foliation defines the horizontal direction. The coefficient $a_{ij}$ of the incidence matrix associated to the pair $(f, \mathcal{R})$ counts how many horizontal sub-rectangles of $R_i\in \cR$ are mapped to the rectangle $R_j\in \cR$ under the action of $f$. The geometric type of $(f, \mathcal{R})$, denoted by $\mathcal{T}(f, \mathcal{R})$, encodes not only the number of sub-rectangles of $R_i$ mapped to $R_j$ by $f$, but also the ordering and the changes in orientation.

In this paper, we assume the existence of Markov partitions for any pseudo-Anosov homeomorphism as stated in \cite{fathi2021thurston} to prove the following theorem.

\begin{theo}\label{Theo: Total invariant}
	A pair of pseudo-Anosov homeomorphisms admits geometric Markov partitions with the same geometric type if and only if they are topologically conjugate via an orientation-preserving homeomorphism.
\end{theo}

In subsequent articles, we aim to address the following problems:
\begin{enumerate}
	\item Construct a canonical class of geometric types.
	\item Characterize those geometric types realized by geometric Markov partitions of pseudo-Anosov homeomorphisms.
	\item Provide an algorithm to determine when two geometric types correspond to conjugate pseudo-Anosov homeomorphisms.
\end{enumerate}

\section{Preliminaries} \label{Sec: Preliminares}.

The following definition includes the cases of Toral Anosov diffeomorphisms and (generalized) pseudo-Anosov homeomorphisms, which may exhibit singularities (spines) in their foliations. Novel generalizations of such maps were introduced by P. Boyland, A. De Carvalho, and T. Hall \cite{boyland2025unimodal}, \cite{Hall2004unimodal}, but these are not considered here.

\begin{defi}\label{Defi: pseudo-Anosov}
An orientation-preserving homeomorphism $f: S \to S$ is a \emph{generalized pseudo-Anosov} (abbr. \textbf{p-A} map) if there exist two transverse, $f$-invariant measured foliations $(\mathcal{F}^s, \mu^s)$ (stable) and $(\mathcal{F}^u, \mu^u)$ (unstable) that share the same singular set and singularity types, and there exists a \emph{stretch factor} $\lambda > 1$ such that:
$$
f_*(\mu^u) = \lambda \mu^u \quad \text{and} \quad f_*(\mu^s) = \lambda^{-1} \mu^s.
$$
The \emph{singular set} of $f$ is given by $\mathrm{Sing}(f) := \mathrm{Sing}(\mathcal{F}^s) = \mathrm{Sing}(\mathcal{F}^u)$.\\
 When $S = \mathbb{T}^2$ and $f$ is Anosov, we additionally require that $\mathrm{Sing}(f)$ is a finite set of $f$-periodic orbits.
\end{defi}
Let us recall that a pair of surface homeomorphisms $f \colon S \to S$ and $g \colon S' \to S'$ are said to be topologically conjugate (or abbr. conjugate) if there exists a homeomorphism $h \colon S \to S'$ such that 
\[
g = h \circ f \circ h^{-1}.
\]
We denote this relation by $f \sim_{\mathrm{Top}} g$.

\subsection{Geometric Markov partitions}
An open and connected set $r$ is trivially bi-foliated by $\mathcal{F}^s$ and $\mathcal{F}^u$ if the following holds:
\begin{itemize}
	\item For every $x \in r$, let $I_x$ be the unique connected component of the intersection $F^s_x \cap r$, where $F^s_x \in \mathcal{F}^s$ is the stable leaf passing through $x$, and let $J_x$ be the unique connected component of the intersection $F^u_x \cap r$, where $F^u_x \in \mathcal{F}^u$ is the unstable leaf passing through $x$.
\end{itemize}

A rectangle is the closure of any open and connected set $r$ which is triviality bi-foliated by $\mathcal{F}^s$ and $\mathcal{F}^u$.

\begin{defi}\label{Defi: Rectangle}
	A compact subset $R$ of $S$ is a parametrized rectangle adapted to $f$ if there exists a continuous function $\rho:\II^2 \rightarrow S$ whose image is $R$  and satisfying the following conditions:
	\begin{itemize}
		\item $\rho:\overset{o}{\II^2} \to S$ is an orientation preserving homeomorphism onto its image that we call \emph{interior of the rectangle } and is denoted by $\overset{o}{R}:=\rho(\overset{o}{\II^2})$.
		\item For every $t\in [0,1]$, $I_t:=\rho([0,1]\times \{t\})$ is contained in a unique leaf of $\mathcal{F}^s$, and $\rho\vert_{[0,1]\times \{t\}}$ is a homeomorphism onto its image. 
		The \emph{horizontal foliation} of $R$ is the decomposition of $R$ given by the \emph{stable intervals}: $\cI(R):=\{I_t\}_{t\in[0,1]}$.
		\item For every $t\in [0,1]$, $J_t:=\rho(\{t\} \times [0,1])$ is contained in a unique leaf of $\mathcal{F}^u$, and $\rho\vert_{\{t\}\times [0,1]}$ is a homeomorphism onto its image. 
		The \emph{vertical foliation}  is the decomposition of $R$ given by the \emph{vertical intervals}: $\cJ(R):=\{J_t\}_{t\in[0,1]}$
	\end{itemize}
	A function like $\rho$ is called a \emph{parametrization} of $R$.
\end{defi}
It is not difficult to see that every rectangle $R = \overline{r}$ admits a parametrization. To this end,  using the transverse measures of the foliations one can construct a homeomorphism  $\phi$ from $r$ to the interior of an affine rectangle $H \subset \mathbb{R}^2$. Then, one extends the inverse map $\phi^{-1}$ continuously to $H$, and, if necessary, composes $\phi^{-1}$ with an orientation-reversing homeomorphism and a linear map of the form $(x,y) \to (\alpha x , \beta y)$ from  $H$ to the unitary square, thus obtaining a parametrization of $R$. This motivates the following lemma.

\begin{lemm}\label{Lemm: every rectangle is a parametrized rectangle}
Every rectangle is a parametrized rectangle.
\end{lemm}

The parametrization of $R$ is not unique bu any parametrization must still send horizontal and vertical intervals of $\mathbb{I}^2$ to stable and unstable intervals contained in a single leaf of $\mathcal{F}^{s,u}$. This leads to the notion of \emph{equivalent parametrizations}.

\begin{defi}\label{Defi: equivalent parametrizations}
	Let $\rho_1$ and $\rho_2$ be two parametrizations of the rectangle $R$, and let
	\begin{equation}\label{Equa: Cambio cordenadas}
		\rho_2^{-1} \circ \rho_1 := (\varphi_s, \varphi_u) : (0,1)\times(0,1) \rightarrow (0,1)\times(0,1).
	\end{equation}
	
	The parametrizations are said to be equivalent if $\varphi_s$ and $\varphi_u : (0,1) \to (0,1)$ are increasing homeomorphisms.
\end{defi}

The vertical and horizontal foliations of the unit square $\mathbb{I}^2$ can each be jointly oriented in four possible ways. Among the four combinations of orientations, only two induce an orientation on $\mathbb{I}^2$ that agrees with the standard orientation of $\mathbb{R}^2$. Since such compatibility depends solely on the choice of a vertical orientation, we use this observation to define the notion of a \emph{geometric rectangle}.

\begin{defi}\label{Defi: Geometric rectangle}
	Let $\rho: \mathbb{I}^2 \to R$ be a (class of) parametrization of a rectangle $R$. We say that $R$ is \emph{geometrized} if we choose an orientation on the vertical leaves of $\mathbb{I}^2$ and then induce an orientation on the unstable foliation of $R$ via the map $\rho$. 
	
	Next, we endow the horizontal leaves of $\mathbb{I}^2$ with the unique orientation that, together with the chosen vertical direction, induces the standard orientation of $\mathbb{R}^2$. This orientation is then transferred to the stable foliation of $R$ via $\rho$.
\end{defi}

\begin{defi}\label{Defi: Vertical/horizontal subrec}
	Let $R$ be a rectangle. A rectangle $H \subset R$ is a \emph{horizontal sub-rectangle} of $R$ if, for all $x \in \overset{o}{H}$, the leaf of the horizontal foliation of $H$, denoted $\mathcal{I}(H)$, passing through $x$ coincides with the leaf of the horizontal foliation of $R$, denoted $\mathcal{I}(R)$, passing through $x$.\\
	Similarly, a rectangle $V \subset R$ is a \emph{vertical sub-rectangle} of $R$ if, for all $x \in \overset{o}{V}$, the leaf of the vertical foliation of $V$, denoted $\mathcal{J}(V)$, passing through $x$ coincides with the leaf of the vertical foliation of $R$, denoted $\mathcal{J}(R)$, passing through $x$.
\end{defi}

Now we can pose a formal definition of our main objects.

\begin{defi}[Geometric Markov partition]\label{Defi: Geometric Markov partition}
	Let $f: S \rightarrow S$ be a pseudo-Anosov homeomorphism. A \emph{Markov partition} for $f$ is a finite collection of rectangles $\mathcal{R} = \{ R_i \}_{i=1}^n$ satisfying the following properties:
	\begin{itemize}
		\item The surface is covered by the rectangles: $S = \bigcup_{i=1}^n R_i$.
		
		\item The rectangles have pairwise disjoint interiors, i.e., for all $i \neq j$, 
	$$
\overset{\circ}{R_i} \cap \overset{\circ}{R_j} = \emptyset.
	$$
		\item For every $i, j \in \{1, \ldots, n\}$, the closure of each non-empty connected component of $\overset{\circ}{R_i} \cap f^{-1}(\overset{\circ}{R_j})$ is a horizontal subrectangle of $R_i$.
		
		\item For every $i, j \in \{1, \ldots, n\}$, the closure of each non-empty connected component of $f(\overset{\circ}{R_i}) \cap \overset{\circ}{R_j}$ is a vertical subrectangle of $R_j$.
	\end{itemize}
	
	If, in addition, every rectangle in $\mathcal{R}$ is geometrized, we call $\mathcal{R}$ a \emph{geometric Markov partition}. In this case, we write $(f, \mathcal{R})$ to indicate that $\mathcal{R}$ is a geometric Markov partition for $f$. 	Finally, the families of horizontal and vertical subrectangles appearing in the last items are the \emph{horizontal} and \emph{vertical subrectangles} of the Markov partition $(f, \mathcal{R})$, respectively.
\end{defi}

Let us to introduce some notation and definitions that we must to appeal in the future.
\begin{defi}\label{Defi: L-R sides}
	Let $R$ be a geometric rectangle adapted to $f$, and let $\rho:\II^2\to R$ be any parametrization of $R$. Consider the following subsets of $R$:
	\begin{itemize}
		\item The \emph{left and right} sides: $\partial^u_{-1}R := \rho(\{0\} \times [0,1])$ and $\partial^u_{1}R := \rho(\{1\} \times [0,1])$, respectively. Each of these is called an $s$-\emph{boundary component} of $R$.
		
		\item The \emph{lower and upper} sides: $\partial^s_{-1}R := \rho([0,1] \times \{0\})$ and $\partial^s_{+1}R := \rho([0,1] \times \{1\})$, respectively. Each of these is called a $u$-\emph{boundary component} of $R$.
		
		\item The \emph{horizontal or stable boundary} $\partial^s R := \partial^s_{-1}R \cup \partial^s_{+1}R$, and the \emph{vertical or unstable boundary} $\partial^u R := \partial^u_{-1}R \cup \partial^u_{+1}R$.
		
		\item The \emph{boundary} of $R$: $\partial R := \partial^s R \cup \partial^u R$.
		
		\item The \emph{corners} of $R$: for $s, t \in \{0,1\}$, the corner labeled $C_{s,t} := \rho(s,t)$.
		
		\item For all $x \in \overset{o}{R}$, let $I_x \in \mathcal{I}(R)$ be the unique leaf of the horizontal foliation of $R$ passing through $x$, and let $J_x \in \mathcal{J}(R)$ be the unique leaf of the vertical foliation of $R$ passing through $x$.
	\end{itemize}
\end{defi}

We have similar definitions for a geometric Markov partition.

\begin{defi}\label{Defi: boundary points}
	Let $\mathcal{R} = \{R_i\}_{i=1}^n$ be a Markov partition for $f$. We define the following distinguished sets:
	\begin{enumerate}
		\item The \emph{horizontal or stable boundary} of $\mathcal{R}$ is the union of the $s$-boundaries components of the rectangles in $\mathcal{R}$:
		$$
		\partial^s \mathcal{R} := \bigcup_{i=1}^n \partial^s R_i.
		$$
		\item The \emph{vertical or unstable boundary} of $\mathcal{R}$ is the union of the $u$-boundaries of the rectangles in $\mathcal{R}$:
		$$
		\partial^u \mathcal{R} := \bigcup_{i=1}^n \partial^u R_i.
		$$
		\item The \emph{boundary} of $\mathcal{R}$ is the union of the vertical and horizontal  boundaries of the partition:
		$$
		\partial \mathcal{R} := \partial^s \mathcal{R} \cap \partial^u \mathcal{R}.
		$$
		\item The \emph{interior} of $\mathcal{R}$ is the union of the interiors of all the rectangles in $\mathcal{R}$:
		$$
	\overset{o}{\cR} := \bigcup_{i=1}^n \overset{o}{R_i}
		$$
	\end{enumerate}
	
	Let $p$ be a periodic point of $f$. Then:
	\begin{enumerate}
		\item $p$ is an $s$-\emph{boundary periodic point} of $\mathcal{R}$ if $p \in \partial^s \mathcal{R}$; a $u$-\emph{boundary periodic point} if $p \in \partial^u \mathcal{R}$; and a \emph{boundary periodic point} if $p \in \partial \mathcal{R}$. The sets of such periodic points are denoted  $\textbf{Per}^{s,u,b}(f, \mathcal{R})$, respectively.
		\item $p$ is an \emph{interior periodic point} if $p \in  \overset{o}{\cR}$. This set of inteiror periodic points is denoted $\textbf{Per}^I(f, \mathcal{R})$.
		\item $p$ is a \emph{corner periodic point} if there exists $i \in \{1, \ldots, n\}$ such that $p$ is a corner point of $R_i$. This set is denoted by $\textbf{Per}^C(f, \mathcal{R})$.
	\end{enumerate}
\end{defi}

\subsubsection{The Arnold' Cat map}

\subsubsection{ The Plykin attractor}

\subsection{Abstract geometric types}
We now introduce a class of abstract combinatorial objects called \emph{abstract geometric types}. Although their formal definition might seem little intuitive at first, these objects arise naturally when studying how rectangles in a geometric Markov partition evolve under iteration by the associated \textbf{p-A} homeomorphism. The terms we use in the definition are chosen carefully to reflect this geometric background.

\begin{defi}\label{Defi: Abstract geometric types}
    An \emph{abstract geometric type} is formally defined as an ordered quadruple:
    \begin{equation}\label{Equ: Defi geom type}
        T = \Big(n,\ \{(h_i,v_i)\}_{i=1}^n,\ \rho_T:\mathcal{H}(T)\to \mathcal{V}(T),\ \epsilon_T: \mathcal{H}(T) \to \{-1,1\}\Big),
    \end{equation}
    satisfying the following axioms:
    
    \begin{enumerate}[leftmargin=*,align=left]
        \item The parameters $n, h_i, v_i \in \mathbb{N}_+$ are strictly positive integers called:
        \begin{itemize}
            \item $n$ - the \emph{number of base rectangles} of $T$;
            \item $h_i$ - the \emph{number of horizontal subrectangles} of the $i$-th base rectangle;
            \item $v_i$ - the \emph{number of vertical subrectangles} of the $i$-th base rectangle.
        \end{itemize}
        
        \item The numbers of horizontal and vertical subrectangles satisfy the combinatorial balance condition:
        \begin{equation}\label{Equ: alpha(T}
        \sum_{i=1}^{n} h_i = \sum_{i=1}^n v_i =: \alpha(T) \in \mathbb{N}_+,
        \end{equation} 
        and we call $\alpha(T)$ the \emph{number of sub-rectangles} of $T$;
        
        \item The \emph{horizontal} and \emph{vertical labels} of $T$ are given by the disjoint sets:
        \begin{align}\label{Defi: Labels of T}
            \mathcal{H}(T) &:= \{(i,j) : 1 \leq i \leq n \text{ and } 1 \leq j \leq h_i\}, \\
            \mathcal{V}(T) &:= \{(k,l) : 1 \leq k \leq n \text{ and } 1 \leq l \leq v_k\},
        \end{align}
        which constitute the \emph{labeling system} of $T$.
        
        \item The structure includes:
        \begin{itemize}\label{Defi: Estructura de T}
            \item A bijection $\rho_T: \mathcal{H}(T) \rightarrow \mathcal{V}(T)$, called the \emph{permutation part} of $T$;
            \item An arbitrary function $\epsilon_T: \mathcal{H}(T) \rightarrow \{-1,1\}$, called the \emph{orientation part} of $T$.
        \end{itemize}
    \end{enumerate}
    The set of all abstract geometric types is denoted by $\mathcal{GT}$.
\end{defi}

Let $T \in \mathcal{GT}$ be an abstract geometric type. If no other geometric types are under discussion, we use the notation:
$$
T = (n, \{(h_i,v_i)\}, \rho, \epsilon).
$$

\subsection{The geometric type of a geometric Markov partition}

Let $f: S \to S$ be a \textbf{p-A} homeomorphism with stable and unstable foliations $(\mathcal{F}^s, \mu^s)$ and $(\mathcal{F}^u, \mu^u)$, respectively, and let $\mathcal{R} = { R_i }_{i=1}^n$ be a geometric Markov partition for $f$. We shall to make explicit a natural way to associate a unique geometric type to the pair \((f, \mathcal{R})\), which we now make explicit. 
Our first step is to label the vertical and horizontal subrectangles of $(f, \mathcal{R})$, as introduced at the end of Definition \ref{Defi: Geometric Markov partition}, according to the vertical and horizontal orientations of the rectangles in $\mathcal{R}$ in which they are contained.

\begin{defi}\label{Defi: Label and geometrization of subrectangles}
	Let $f: S \to S$ be a \textbf{p-A} homeomorphism, and let $\mathcal{R} = \{R_i\}_{i=1}^n$ be a geometric Markov partition of $f$.
	
	\begin{itemize}
		\item Let $h_i \geq 1$ be the number of horizontal subrectangles of $(f,\mathcal{R})$ contained in $R_i$. We label them from \emph{bottom to top} as $\{H^i_j\}_{j=1}^{h_i}$, according to the vertical direction of $R_i$.
		
		\item Let $v_k \geq 1$ be the number of vertical subrectangles contained in $R_k$. We label them from \emph{left to right} as $\{V^k_l\}_{l=1}^{v_k}$, according to the horizontal direction of $R_k$.
	\end{itemize}
	
	We assign to each horizontal subrectangle $H^i_j \subset R_i$ and each vertical subrectangle $V^k_l \subset R_k$ the same horizontal and vertical directions as those of the rectangles $R_i$ and $R_k$ in which they are contained.
\end{defi}

\begin{defi}\label{Defi: Labels of (f,R) }
The set of horizontal labels of $(f, \mathcal{R})$ is the formal set
$$
\mathcal{H}(f, \mathcal{R}) = \{ (i, j) : 1 \leq i \leq n \text{ and } 1 \leq j \leq h_i \},
$$
and the corresponding set of vertical labels is
$$
\mathcal{V}(f, \mathcal{R}) = \{ (k, \ell) : 1 \leq k \leq n \text{ and } 1 \leq \ell \leq v_k \}.
$$
\end{defi}

\begin{defi}\label{Defi: Permutation oritation (f,R) }
	Define the bijection $\rho : \mathcal{H}(f, \mathcal{R}) \to \mathcal{V}(f, \mathcal{R})$ by
\begin{equation}\label{Equa: permitation (f,R)}
		\rho(i, j) = (k, \ell) \quad \text{if and only if} \quad f(H^i_j) = V^k_\ell,
\end{equation}
 and orientation change function $\epsilon : \mathcal{H}(f, \mathcal{R}) \to \{-1, 1\}$ as
\begin{equation}\label{Equa: Orientation (f,R)}
	\epsilon(i, j) = 
\begin{cases}
	1 & \text{if the vertical directions of } f(H^i_j) \text{ and } V^k_\ell \text{ coincide,} \\
	-1 & \text{otherwise,}
\end{cases}
\end{equation}
	whenever $\rho(i, j) = (k, \ell)$.
\end{defi}

Now we can put all this information in a single definition.

\begin{defi}\label{Defi: geometric type of a Markov partition}
	Let $\mathcal{R} = \{R_i\}_{i=1}^n$ be a geometric Markov partition for $f$. The \emph{geometric type} of the pair $(f, \mathcal{R})$ is defined as
	$$
	\mathcal{T}(f, \mathcal{R}) := \left(n, \{(h_i, v_i)\}_{i=1}^n, \rho, \epsilon\right),
	$$
	where:
	\begin{itemize}
		\item $n$ is the number of rectangles in the family $\mathcal{R}$;
		\item $h_i$ and $v_i$ are the numbers of horizontal and vertical subrectangles of $(f, \mathcal{R})$ contained in $R_i$;
		\item $\rho : \mathcal{H}(f, \mathcal{R}) \to \mathcal{V}(f, \mathcal{R})$ is the bijection defined in Equation~\ref{Equa: permitation (f,R)};
		\item $\epsilon : \mathcal{H}(f, \mathcal{R}) \to \{-1, 1\}$ is the orientation function defined in Equation~\ref{Equa: Orientation (f,R)}.
	\end{itemize}
\end{defi}

\begin{defi}\label{Defi: The pseudo-anosov Class}
	An abstract geometric type $T \in \cG\cT$ is in the pseudo-Anosov class if there exists a pseudo-Anosov homeomorphism $f:S \to S$ with a geometric Markov partition $\cR$ such that the geometric type of the pair is the geometric type $T$:
	$$
	T = \cT(f, \cR).
	$$
	In this case, we say the pair $(f, \cR)$ realizes or is a realization of the geometric type $T$. The pseudo-Anosov class is denoted by $\cG\cT(\textbf{p-A})$.
\end{defi}

\section{The symbolic dynamics of a geometric type}

Let $\cR=\{R_i\}_{i=1}^0$ be a geometric Markov partition of  the \textbf{p-A} homeomorphisms $f:S\to S$. In \cite[Exposition 10]{fathi2021thurston} the incidence matrix of the pair $(f,\cR)$ is introduce as  is matrix whose coefficient  $a_{ij}$ is $1$ if $f(\overset{o}{R_i}) \cap \overset{o}{R_k} \neq \emptyset$ and $0$ otherwise, without taking into account the number of intersections, and such information is essential for the develop of our results.

\begin{defi}\label{Defi: Incidence matrix geo markov partition}
	Let $f$ be a \textbf{p-A} homeomorphism and let $\mathcal{R}$ be a geometric Markov partition of $f$. Let  $T := \cT(f,\cR)$ be the geometric type of the pair, where: $T = \left( n, \{h_i,v_i,\rho,\epsilon \} \right)$. The incidence matrix of the pair $(f, \mathcal{R})$ is the $n \times n$ integer matrix $A(f,\mathcal{R})$ whose coefficients are defined as follows:
	$$
	a_{ik} = \#\{j \in \{1,\cdots,h_i\} : \rho = (k,l)\}.
	$$
	
	i.e., it is equal to the number of horizontal sub-rectangles of $R_i$ that $f$ sends to vertical sub-rectangles of $R_k$.
\end{defi}

\begin{rema}\label{Rema: Notation incidece matriz Tipo}
	Let $T = \left( n, \{h_i, v_i, \rho, \epsilon \} \right)$ be an abstract geometric type. We can define its incidence matrix $A(T)$ as:
	$$
	a_{ik} = \#\{j \in \{1,\cdots,h_i\} : \rho = (k,l)\}.
	$$
	
	Since the coefficients in the incidence matrix $A(f,\cR)$ only depend on the geometric type $T$ of the pair $(f,\cR)$, the following notations are considered equivalent:
	$$
	A(f,\cR) = A(\cT(f,\cR)) = A(T).
	$$
\end{rema}

If $A$ is a square matrix and $n \in \mathbb{N}$, the coefficients of the $n$-th power of $A$ are denoted as $A^n = (a_{i,j}^{(n)})$.

\begin{defi}\label{Defi: mixign y binary.}
	Let $A = (a_{ij})$ be an $n \times n$ matrix with integer coefficients. Then $A$:
	\begin{itemize}
		\item is \emph{non-negative} and denoted $A \geq 0$ if for all $1 \leq i,j \leq n$, $a_{i,j} \in \mathbb{N}$.
		\item is \emph{positive definite} and denoted $A > 0$ if for all $1 \leq i,j \leq n$, $a_{i,j} \in \mathbb{N}_+$ is a positive integer.
		\item is \emph{binary} if all its coefficients are either $0$ or $1$.
	\end{itemize}
	
	Finally, if $A$ is a non-negative matrix, we say that $A$ is \emph{mixing} if there exists $N \in \mathbb{N}_+$ such that $A^N$ is positive definite.
\end{defi}

Let $I_n = \{1,\cdots,n\}$ be a finite set called the \emph{alphabet}. A bi-infinite word is a function from $\mathbb{Z} \to I_n$, and we denote its elements as:

\begin{equation}
	\bw = ( \cdots, w_{-2}, w_{-1}, \underline{w_0}, w_{1}, w_{2}, \dots)
\end{equation}

where the underline in $w_0$ indicates the position $0$ inside the word.

 Let $\Sigma$ be the set of bi-infinite words endowed with the topology induced by the metric: 
 $$d_{\Sigma}(\bw, \bv) = \sum_{z \in \mathbb{Z}} \frac{\delta(w_z, v_z)}{2^{|z|}},
 $$
  where $\delta(w_z, v_z) = 0$ if $w_z = v_z$ and $\delta(w_z, v_z) = 1$ otherwise. This compact set is called the total shift space in $n$-symbols. 
  
  Let $\sigma:\Sigma \to \Sigma$ be the \emph{shift} transformation defined as follows: If $\bw = ( \cdots, w_{-2}, w_{-1}, \underline{w_0}, w_{1}, w_{2}, \dots) \in \Sigma$, then:
\begin{equation}
	\sigma(\bw) = ( \cdots, w_{-1}, w_{0}, \underline{w_1}, w_{2}, w_{3}, \dots).
\end{equation}

If $A$ is an $n \times n$ matrix, mixing and binary, the set 
$$
\Sigma_A := \{\underline{w} = (w_z)_{z \in \mathbb{Z}} \in \Sigma : \forall z \in \mathbb{Z}, (a_{w_z,w_{z+1}}) = 1 \}
$$
is a compact and $\sigma$-invariant set, i.e. $\sigma(\Sigma_A) = \Sigma_A$ (\cite[Chapter 1]{KitchensSymDym}), and the \emph{sub-shift of finite type} associated with $A$ is the dynamical system $(\Sigma_A, \sigma_A)$, where $\sigma_A := \sigma|_{\Sigma_A}$.

According to \cite[Lemma 10.21]{fathi2021thurston}, if $f$ is a pseudo Anosov homeomorphism and $\cR$ is a geometric Markov partition, the incidence matrix $A(f,\cR)$  is mixing.  In the following definition we use the notation develop in Remark \ref{Rema: Notation incidece matriz Tipo}.

\begin{defi}\label{Defi: Symbolically presentable and induced shift}
	
	A geometric type $T$ in the pseudo-Anosov class $\subset \mathcal{G}\mathcal{T}(\textbf{pA})$ whose incidence matrix $A(T)$ is binary is called \emph{symbolically presentable}. The set of \emph{symbolically presentable} geometric types is denoted by $\mathcal{G}\mathcal{T}(\textbf{pA})^{sp}$.
	If $T \in \mathcal{G}\mathcal{T}(\textbf{pA})^{sp}$ is a symbolically presentable geometric type, the \emph{sub-shift of finite type} induced by $T$ is the one determined by its incidence matrix $A(T)$; that is, the symbolic dynamical system $(\Sigma_{A(T)}, \sigma_{A(T)})$.
\end{defi}

It is clear that not every geometric type in the pseudo-Anosov class has a binary incidence matrix. However, in the next subsection, we will prove that every pseudo-Anosov homeomorphism admits a Markov partition whose incidence matrix is binary. Moreover, it is not trivially obvious that if two pairs, consisting of a homeomorphism and a Markov partition, have the same geometric type, i.e., $\cT(f,\cR)=\cT(g,\cG)$, then they each admit Markov partitions with the same geometric type and binary incidence matrix.

\subsection{The Binary refinement.}

In this subsection we shall to prove the following proposition. 
\begin{prop}\label{Prop: Refinamiento binario}
	Let $T \in \cG\cT(\textbf{p-A})$ be a geometric type in the pseudo-Anosov class. Let $f: S \to S$ and $g: S' \to S'$ be pseudo-Anosov homeomorphisms, and let $\cR_f$ and $\cR_g$ be geometric Markov partitions of the respective homeomorphisms, such that the geometric types of the pairs coincides with $T$:
	$$
	T := \cT(f, \cR_f) = \cT(g, \cR_g).
	$$
	Then there exist geometric Markov partitions $\bB(\cR_f)$ and $\bB(\cR_g)$ such that:
	\begin{itemize}
		\item The geometric types of the respective pairs are the same: 
		$$
		\cB(T) = \cT(f, \bB(\cR_f)) = \cT(g, \bB(\cR_g)).
		$$
		\item The incidence matrix of the partitions , 
		$$
		A(\cB(T)) = A(f, \bB(\cR_f)) = A(g, \bB(\cR_g)),$$ 
		is binary.
	\end{itemize}
	The symbolically presentable geometric type $\cB(T) \in \cG\cT(\textbf{p-A})^{sp}$ is the \emph{binary refinement} of $T$, and the geometric Markov partitions $\bB(\cR_f)$ and $\bB(\cR_g)$ are the \emph{horizontal refinements} of $\cR_f$ and $\cR_g$, respectively.
\end{prop}

We begin by taking an arbitrary pair $(f, \cR)$ whose geometric type is $T$. We then proceed to construct the binary refinement $\bB(\cR)$ and compute its geometric type $\bB(T)$ using the information provided by $T$ and its implications for the dynamics for the rectangles in $\cR$. From our construction, it follows that the incidence matrix of the resulting  Markov partition is binary.  It shall be clear that, given another pair $(g, \cG)$ with the same geometric type $T$, we can apply the same construction to obtain a geometric Markov partition with the same refined geometric type.

The following lemma provides a useful criterion to determine whether a family of rectangles forms a Markov partition for $f$ by ensuring the $f$- and $f^{-1}$-invariance of its boundary. In the proof, we use the notation introduced in Definition~\ref{Defi: L-R sides}.

\begin{lemm}\label{Prop: Markov criterion boundary}
	Let $f: S \rightarrow S$ be a \textbf{p-A} homeomorphism, and let $\mathcal{R} = \{R_i\}_{i=1}^n$ be a family of rectangles whose union is $S$ and whose interiors are disjoint. Then, $\mathcal{R}$ is a Markov partition for $f$ if and only if the following conditions hold:
	\begin{itemize}
		\item The stable boundary of $\mathcal{R}$, $\partial^s\mathcal{R} := \cup_{i=1}^n \partial^s R_i$, is $f$-invariant.
		\item The unstable boundary of $\mathcal{R}$, $\partial^u\mathcal{R} := \cup_{i=1}^n \partial^u R_i$, is $f^{-1}$-invariant.
	\end{itemize}
\end{lemm}

	\begin{proof}
		If $\mathcal{R}$ is a Markov partition of $f$ and $I$ is an $s$-boundary component of $R_i$, then $I$ is also an $s$-boundary component of a horizontal sub-rectangle $H \subset R_i$ in $(f, \mathcal{R})$. Since $f(H) = V \subset R_k$, where $V$ is a vertical sub-rectangle of $R_k \in \mathcal{R}$, it follows that $f(I)$ must be contained in the $s$-boundary of $R_k$. This proves that $\partial^s \mathcal{R}$ is $f$-invariant. The $f^{-1}$-invariance of $\partial^u \mathcal{R}$ is similarly proved.

		Now, assume that $\partial^s \mathcal{R}$ is $f$-invariant and $\partial^u \mathcal{R}$ is $f^{-1}$-invariant. Let $C$ be a nonempty connected component of $f^{-1}(\overset{o}{R_k}) \cap \overset{o}{R_i}$. We claim that $C$ is the interior of a horizontal sub-rectangle of $R_i$. 
		Take $x \in C$ and let $\overset{o}{I_x}$ be the interior of the horizontal segment of $R_i$ passing through $x$, and let $\overset{o}{I_x'}$ be the connected component of $\mathcal{F}^s \cap C$ containing $x$. Clearly, $\overset{o}{I_x'} \subset \overset{o}{I_x}$, but if $I_x \neq I_x'$, at least one endpoint $z$ of $I_x'$ lies in $\overset{o}{I_x}$ and therefore,  $z$ is in the interior of $R_i$. 
		Moreover, $z$ must be equal to $f^{-1}(z')$ for some $z' \in \partial^u R_k$ (or we could extend $I_x'$ a little bit more inside $\overset{o}{R_i}$), and since $\partial^u \mathcal{R}$ is $f^{-1}$-invariant, $z \in \partial^u \mathcal{R}$, which is a contradiction as $\overset{o}{R_i} \cap \partial^u \mathcal{R} = \emptyset$. Therefore, $z \in \partial^u \mathcal{R} \cap \overset{o}{R_i} = \emptyset$. Similarly, we can show that $f(C)$ is the interior of a vertical sub-rectangle of $R_k$.
		
	\end{proof}

\begin{lemm}\label{Lemm: Binary partition}
	Let $T \in \cG\cT(\mathbf{p\text{-}A})$ be a geometric type in the pseudo-Anosov class. Then, for any pair $(f, \cR)$ realizing $T$, the family of horizontal sub-rectangles of $(f, \cR)$,
	$$
	\bB(f,\cR) := \{H^i_j\}_{(i,j) \in \cH(T)},
	$$
is a Markov partition of $f$ called \emph{Horizontal refinement} of $\cR$.
\end{lemm}

\begin{proof}
	We can assume that $f: S \to S$, $\cR = \{R_i\}_{i=1}^n$ 
	$$
	\cT(f,\cR)= T := \bigl(n, \{(h_i, v_i)\}_{i=1}^n,\rho, \epsilon\bigr)
	$$
	
	Clearly, every element in the family $	\bB(f,\cR) $ is rectangle and its union  $\bigcup \{H^i_j : (i,j) \in \cH\} = S$, since it equals the union of the rectangles in $\cR$.\\
	
	Two distinct horizontal sub-rectangles of the same $R_i \in \cR$ have disjoint interiors, and this also holds between horizontal sub-rectangles belonging to different rectangles in $\cR$. Therefore, the elements in $H(f,\cR)$ are rectangles with disjoint interiors.
	
	The unstable boundary of $\cB(f,\cR)$, denoted by $\partial^u H(f,\cR) := \bigcup_{(i,j)\in \cH} \partial^u H^i_j$, coincides with the unstable boundary of the Markov partition $\cR$, and hence is $f^{-1}$-invariant. The stable boundary of $H(f,\cR)$ is the union of the stable boundaries of the horizontal sub-rectangles, i.e., $\partial^s H(f,\cR) = \bigcup_{(i,j)\in \cH(T)} \partial^s H^i_j$. We can suppose that $f(H^i_j) = V^k_l$, where $V^k_l$ is a vertical sub-rectangle of $R_k \in \cR$. Since $\partial^s V^k_l \subset \partial^s R_k$, it follows that $f(\partial^s H^i_j) \subset \partial^s R_k$. Therefore, the stable boundary of $H(f,\cR)$ is $f$-invariant.
	
	Since the family $H(f,\cR)$ satisfies the conditions of Proposition \ref{Prop: Markov criterion boundary}, it follows that $H(f,\cR)$ is a Markov partition for $f$.
\end{proof}

Now we are ready to give a geometrization to our horizontal refinement, that is, to endow the family of rectangles with an order and a vertical orientation on each of them. To gain some intuition about the following definition see Figure \ref{Fig: Binary refinament}.

\begin{figure}[h]
	\centering
	\includegraphics[width=0.6\textwidth]{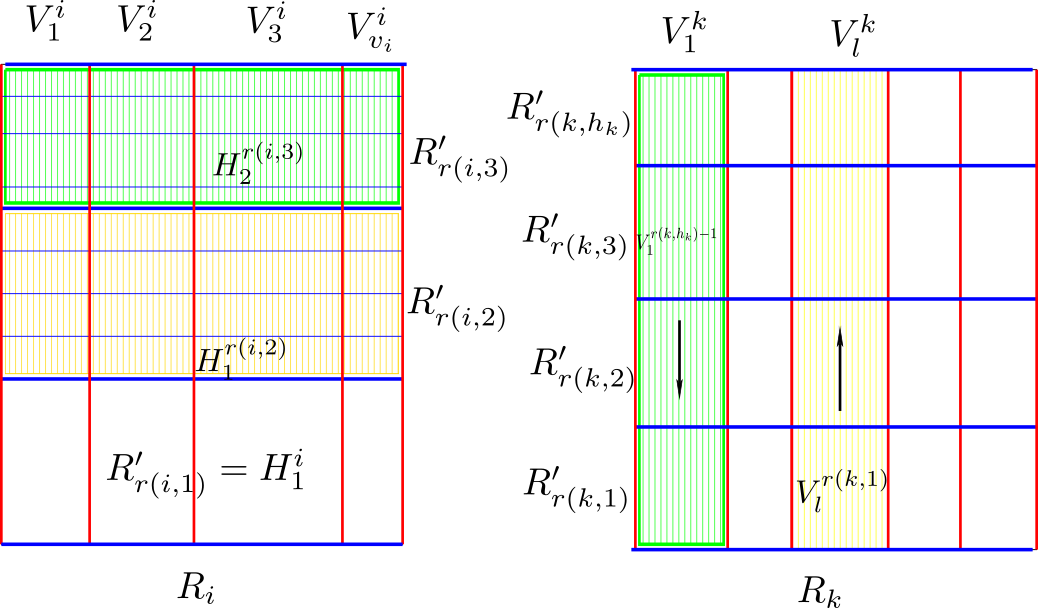}
	\caption{Binary refinement}
	\label{Fig: Binary refinament}
\end{figure}

\begin{defi}\label{Defi: Geomtrization refinamiento binario }
	Let $T \in \cG\cT(\textbf{p-A})$ be a geometric type in the pseudo-Anosov class:
	$$
	T := (n, \{(h_i, v_i)\}_{i=1}^n, \rho, \epsilon).
	$$
	
	Let $(f, \mathcal{R})$ be a pair realizing $T$. The binary refinement  of $(f,\mathcal{R})$ is the geometric Markov partition  of the \emph{p-A} homeomorphism $f$ whose rectangles are those in the horizontal refinement $\mathbf{B}(f, \mathcal{R})$ endowed with the following orientations and order:
	
	\begin{itemize}
		\item Assign to each rectangle $H^i_j$ the same vertical orientation as $R_i$.
		
		\item Endow the set of horizontal labels:
		$$
		\mathcal{H}(T) := \{(i, j) : 1 \leq i \leq n, \quad 1 \leq j \leq h_i \}
		$$
		with the \emph{lexicographic order} via the function
		$$
		r: \mathcal{H}(T) \to \{1, \ldots, \alpha(T)\},
		$$
		where
		$$
		\alpha(T) := \sum_{i=1}^n h_i,
		$$
		and
		$$
		r(i_0, j_0) = \sum_{i \leq i_0} h_i + j_0.
		$$
		
		\item The rectangles in $\bB(f, \mathcal{R})$ are indexed by the correspondence
		$$
		H_{r(i,j)} := H^i_j.
		$$
		\end{itemize}
	The binary refinement of $(f,\mathcal{R})$ is denoted by $(f,\bB(\mathcal{R}))$.
\end{defi}

In order to prove \ref{Prop: Refinamiento binario} we must stablish a few lemmas.

\begin{lemm}\label{Lemm: Incidece matrix refi bin is bin} 
 The incidence matrix of  the binary refinement of $(f, \bB(\cR))$ is binary.
\end{lemm}

\begin{proof}
By definition of the incidence matrix $A(f, \bB(\cR))$, if $(i,j), (k,j') \in \cH(T)$, its coefficient $a_{r(i,j)r(k,j')}$ is equal to the number of connected components in the intersection of $f(\overset{\circ}{H^i_j})$ with the rectangle $\overset{\circ}{H^{k}_{j'}}$. But $f(H^i_j) = V^{k'}_l$ for some $(k',l) \in \cV(T)$, and then
$$
f(\overset{\circ}{H^i_j}) \cap \overset{\circ}{H^{k}_{j'}} = \overset{\circ}{V^{k'}_l} \cap \overset{\circ}{H^{k}_{j'}}.
$$
In this manner, if $k = k'$, the intersection has only one connected component, as it is the intersection between a vertical and a horizontal subrectangle of the same rectangle. If $k \neq k'$, the intersection is empty, since the interiors of distinct rectangles in $\cR$ are disjoint.
Therefore, the coefficient $a_{r(i,j)r(k,j')}$ must be equal to $0$ or $1$ as was claimed.
\end{proof}

\begin{lemm}\label{Prop: Unique horizontal type}
	Let $T$ be a geometric type in the pseudo-Anosov class, and let $(f, \mathcal{R}_f)$ and $(g, \mathcal{R}_g)$ be two pairs that realize $T$. Then:
	\begin{itemize}
		\item The geometric type $\mathcal{T}(f, \mathcal{B}(\mathcal{R}_f))$ is uniquely determined by $T$ and can be computed through an algorithm that only uses information contained in $T$.
		\item The binary refinements $(f, \mathcal{B}(\mathcal{R}_f))$ and $(g, \mathcal{B}(\mathcal{R}_g))$ have the same geometric type.
	\end{itemize}
	
	We call the binary refinement of $T$, and denote it by $\mathcal{B}(T)$, the geometric type of any binary refinement of a pair realizing $T$, i.e.,
	$$
	\mathcal{B}(T) = \mathcal{T}(f, \mathcal{B}(\mathcal{R}_f)) = \mathcal{T}(g, \mathcal{B}(\mathcal{R}_g)).
	$$
\end{lemm}

\begin{proof}
Let to fix a notation for the geometric type of $(f,\bB(\cR)$:
$$
\cT(f, \bB(\cR))=\{n',\{(h_i',v_i')\}_{i'=1}^{n'}, \rho',\epsilon'\}.
$$
We need to determine all the parameters using the information provided by $T$.

Clearly the number of elements in $\cH(T)$ is equal to $n'$ then:
\begin{equation}\label{Equ: number rectangles in horizontal ref}
	n'=\sum_{i=1}^{n}h_i=\sum_{i=1}^n v_i=\alpha(T).
\end{equation}

Let $(i,j) \in \cH(T)$. For the rest of the proof, we shall suppose 
$$
(\rho, \epsilon)(i,j) = ((k,l), \epsilon(i,j)).
$$

A vertical subrectangle of $H_{r(i,j)}$ is equal to the closure of a connected component of the form
$$
f(\overset{\circ}{H^{i'}_{j'}}) \cap \overset{\circ}{H^i_j} = \overset{\circ}{V^i_l} \cap \overset{\circ}{H^i_j}.
$$

Therefore, $H_{r(i,j)}$ must contain at most $v_i$ vertical subrectangles. But for every rectangle $V^k_l$, there exists a unique $H^{i'}_{j'}$ such that $f(H^{i'}_{j'}) = V^k_l$. Therefore, the number of vertical subrectangles of $H_{r(i,j)}$ is equal to $v_i$, i.e.,
\begin{equation}\label{Equ: number vertical sub in the horizontal ref}
	v'_{r(i,j)}=v_{i}
\end{equation}

Moreover, these vertical sub-rectangles are ordered from left to right in a coherent way with respect to the horizontal orientation of $R_i$ as:

\begin{equation}\label{Equ: Vertical order of horizontal ype}
	\{ V^{r(i,j)}_{l} \}_{l=1}^{v_i}.
\end{equation}

Since $\rho(i,j)=(k.l)$, the number of horizontal sub-rectangles of $H_{r(i,j)}$, denoted as $h'_{r(i,j)}$, is equal to $h_{k}$ because  $f(H_{r(i, j)})=f(H^{i}_{j})=V^{k}_{l}$  intersects exactly $h_{k}$ distinct horizontal sub-rectangles of  $R_k$ and no other horizontal sub-rectangles of $(f,\bB(\cR))$. 

\begin{equation}\label{Equ: number horizontal sub in the horizontal ref}
	h'_{r(i,j)}=h_k \text{ if } \rho(i,j)=(k,l)
\end{equation}
These horizontal sub-rectangles are ordered in increasing order according to the vertical orientation of $H_{r(i,j)}$, which is inherited from $R_i$, in the following manner:

\begin{equation}\label{Equa: Horizontal rec of the horizontal ref}
	\{H^{r(i,j)}_{j'}\}_{j'=1}^{h_k} 
\end{equation}

We proceed to determine $\rho'$ and $\epsilon'$.  To compute $\rho'$, we need to consider the change of vertical orientation in $f(H^k_j)$ given by the sing of $\epsilon(i,j)$. We are going to  split the computations in the two cases.

Assume $\epsilon_T(i,j)=1$ and take  $j_0\in \{1,\cdots, h_{k}\}$ as the label of the horizontal sub-rectangle  $H^{r(i,j)}_{j_0}$ of $H_{r(i,j)}$ in $j_0$ position. Since $f(H_{r(i,j)})=V^k_l$ and it preserves the vertical orientation, the horizontal sub-rectangle of $R_k$ that intersects $f(H_{r(i,j)})=V^k_l$ at position $j_0$ with respect to the vertical orientation of $R_k$ have label $r(k,j_0)$. Also, $f(H^{r(i,j)}_{j_0})$ corresponds to the vertical sub-rectangle of $R_k$ that is at position $l$, so $f(H^{r(i,j)}_{j_0})$ is the vertical sub-rectangle $V^{r(k,j_0)}_l$ of $H_{r(k,j_0)}$.
 We can express this construction in terms of the geometric type using the formula:

\begin{equation}\label{Equa: horizontal type for preseving orientation }
	(\rho',\epsilon')(r(i,j),j_0):=(\rho'(r(i,j),j_0),\epsilon'(r(i,j),j_0)=(r(k,j_0),l,\epsilon_T(i,j)).
\end{equation}

Now lets to assume  $\epsilon_T(i,j)=-1$, this means that $f$ changes the vertical orientation of $H^i_j$ with respect to $V^k_l=f(H^i_j)$. 
This implies that the horizontal sub-rectangle of $R_k$ containing the image of $H^{r(i,j)}_{j_0}$ is located at position $j_0$, but with the inverse vertical orientation of $R_k$, which corresponds to position $(h_k-(j_0-1))$, with respect to the positive  orientation in $R_k$. Therefore, the horizontal sub-rectangle of $R_k$ that contains to $f(H^{r(i,j)}_{j_0})$ is:

$$
H^k_{h_k-(j_0-1)}=H_{r(k,h_k-(j_0-1))}
$$

The vertical sub-rectangle of $H_{r(k,h_k -(j_0-1)))}$ that contains $f(H^{r(i,j)}_{j_0})$ is a subset of $V^k_l$, so it is located at position $l$ with respect to the horizontal orientation of $R_k$. We can conclude that:

$$
f(H^{r(i,j)}_{j_0})=V^{r(k,h_k-(j_0-1))}_l.
$$

The vertical direction of the sub-rectangle $H^{r(i,j)}_{j_0}$ is preserved by the action of $f$ if and only if the vertical direction of $H^i_j$ is preserved up the action of $f$. Therefore we have following formula.

\begin{equation}\label{Equa: horizontal type for change orientation }
\epsilon'(r(i,j),j_0)=\epsilon(i,j)
\end{equation}

The equations \ref{Equa: horizontal type for preseving orientation } and  \ref{Equa: horizontal type for change orientation } are determined by $T$, and their computation is algorithmic.

The second point in our proposition follows directly from our construction and the procedure we explained to compute the geometric type of the binary refinement. This ends our proof

\end{proof}

Proposition \ref{Prop: Refinamiento binario} follows form Lemmas  \ref{Lemm: Binary partition}, \ref{Lemm: Incidece matrix refi bin is bin} and \ref{Prop: Unique horizontal type}.

\subsection{The projection and the sector codes}

Let $f:S \to S$ be a pseudo-Anosov homeomorphism and let $\cR$ be a geometric Markov partition of $f$ such that the incidence matrix $A(f,\cR)$ is binary. If the rectangles in the Markov partition are embedded in the surface, in \cite[Exposé 10]{fathi2021thurston}, the author introduces a function that semi-conjugates the sub-shift of finite type $(\Sigma_{A(f,\cR)}, \sigma_{A(f,\cR)})$ with the pseudo-Anosov homeomorphism $f$. This function is similar to the one we introduced in Equation \ref{Equa: proyection (r,R) }, and then proceed to  prove Proposition \ref{Prop:proyecion semiconjugacion}.

The disadvantage of the definition in \cite[Exposé 10]{fathi2021thurston} is that, in order for the function to be well-defined (i.e., to actually be a function), it is necessary to assume that the rectangles are embedded. Clearly, this is not the case in our definition.

\begin{defi}\label{Defi: proyection (f,R)}
	Let $f:S \to S$ be a pseudo-Anosov homeomorphism and $\cR$ a geometric Markov partition of $f$ such that the incidence matrix $A(f,\cR)$ is binary. Then the projection $\pi_{(f,\cR)}:\Sigma_{A(f,\cR)} \to S$ is the map that assigns to each $\bw = (w_z)_{z\in \ZZ} \in \Sigma_{A(f,\cR)}$ the set:
	\begin{equation}\label{Equa: proyection (r,R) }
		\pi_{(f,\cR)}(\bw) = \bigcap_{n\in \NN} \overline{\bigcap_{z=-n}^{n} f^{-z}(\overset{\circ}{R_{w_z}})}.
	\end{equation}
\end{defi}

The proof of Proposition \ref{Prop:proyecion semiconjugacion} is essentially the same as that given in \cite[Lemma 10.16]{fathi2021thurston}, but we must give a more careful characterization of the points in the fiber $\pi_{(f,\mathcal{R})}^{-1}(x)$ as projections of \emph{sector codes}. We dedicate the rest of this section to introducing such codes and using them to prove the finite-to-one property of our projection and refer the reader to the respective reference to full fill the classic details.

\begin{prop}\label{Prop:proyecion semiconjugacion}
	Let $(f,\cR)$ be a pair such that the incidence matrix $A(f,\cR)$ is binary. Then the projection $\pi_{(f,\cR)}:\Sigma_{A(f,\cR)} \rightarrow S$ is a continuous, surjective, and finite-to-one map. Moreover, $\pi_{(f,\cR)}$ semi-conjugates $f$ with $\sigma_{A(f,\cR)}$, that is,
	$$
	f \circ \pi_{(f,\cR)} = \pi_{(f,\cR)} \circ \sigma_{A(f,\cR)}.
	$$
\end{prop}

\subsubsection{The sectors of a point}We must formalize the notion of a sector of a point $x \in S$ as the germ of a sequence converging to $x$ in a very specific manner. We begin by introducing the notion of a regular neighborhood  (Figure \ref{Fig: Regular neighborhood}) , its existence is stated in following theorem, which corresponds to \cite[Lemme 8.1.4]{bonatti1998diffeomorphismes}.

\begin{figure}[h]
	\centering
	\includegraphics[width=0.2\textwidth]{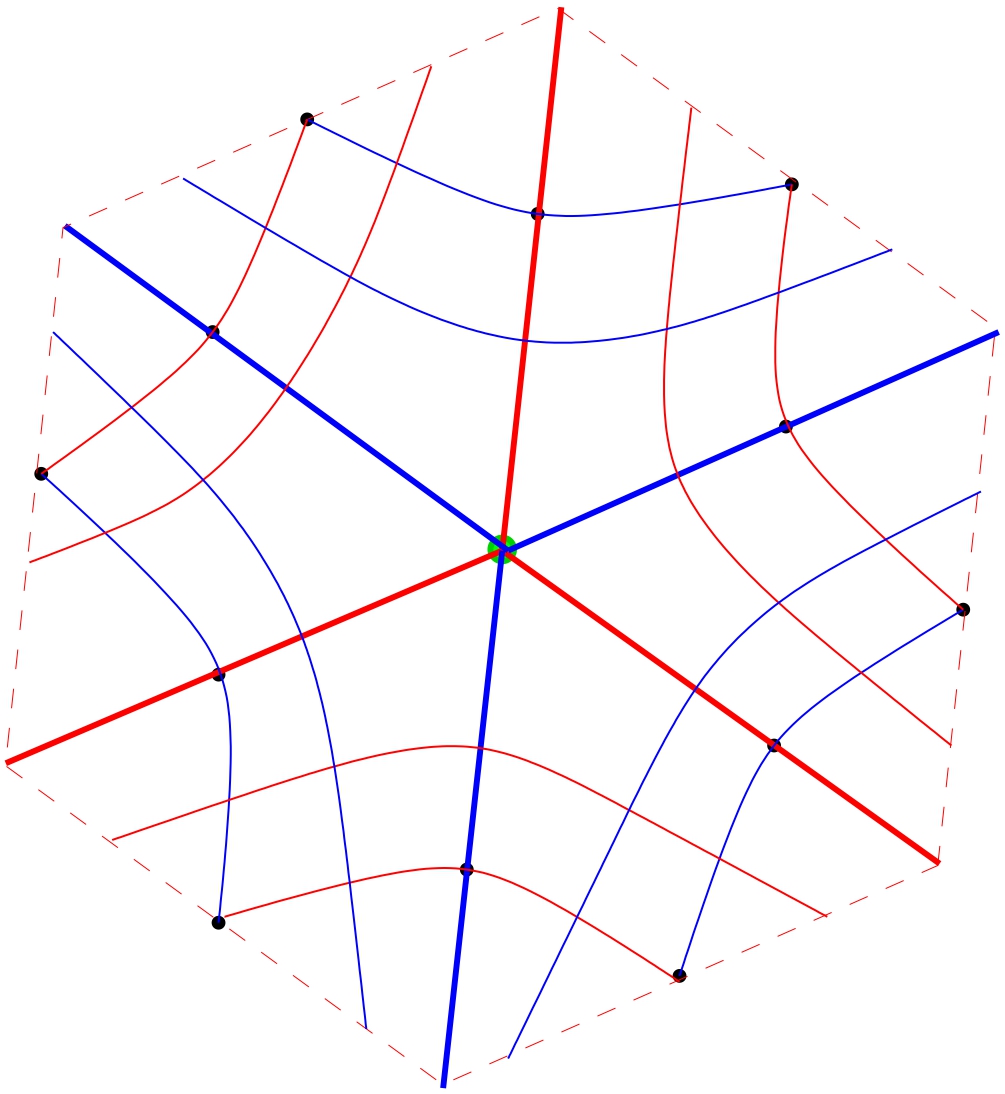}
	\caption{Regular neighborhood of a $3$-prong}
	\label{Fig: Regular neighborhood}
\end{figure}

\begin{theo}\label{Theo: Regular neighborhood}
Let $f: S \rightarrow S$ be a generalized pseudo-Anosov homeomorphism, and let $p \in S$. Assume that the stable leaf passing through $p$ has $k \geq 1$ separatrices. Then there exists $\epsilon_0 > 0$ such that for all $0 < \epsilon < \epsilon_0$, there is a neighborhood $D$ of $p$ with the following properties:
	\begin{itemize}
		\item The boundary of $D$ consists of $k$ segments of unstable leaves alternating with $k$ segments of stable leaves.
		\item For every (stable or unstable) separatrix $\delta$ of $p$, the connected component of $\delta \cap D$ containing $p$ has (stable or unstable) measure equal to $\epsilon$.
	\end{itemize}
	
	We say that $D$ is a \emph{regular neighborhood} of $p$ with side length $\epsilon$. When necessary, we denote it by $D(p,\epsilon)$.
\end{theo}

Let $D(p,\epsilon)$ be a regular neighborhood of $p$ with side length $0<\epsilon\leq\epsilon_0$. We assume that $p$ has $k$ separatrices. Using the orientation of $S$, we label the stable and unstable separatrices of $p$ cyclically in the counterclockwise direction as $\{\delta_i^s\}_{i=1}^k(\epsilon)$ and $\{\delta_i^u\}_{i=1}^k(\epsilon)$, respectively.

We define $\delta_i^s(\epsilon)$ as the connected component of $\delta_i^s \cap D(p,\epsilon)$ containing $p$, and $\delta_i^u(\epsilon)$ as the connected component of $\delta_i^u \cap D(p,\epsilon)$ containing $p$. We assume that $\delta_i^u(\epsilon)$ is located between $\delta_i^s(\epsilon)$ and $\delta_{i+1}^s(\epsilon)$, where $i$ is taken modulo $k$.

The connected components of $\operatorname{Int}(D(p,\epsilon_0)) \setminus \left(\bigcup_{i=1}^k \delta_i^s(\epsilon_0) \cup \delta_i^u(\epsilon_0)\right)$ are labeled with a cyclic order and denoted as $\{E(\epsilon_0)_j(p)\}_{j=1}^{2k}$, where the boundary of $E(\epsilon_0)_1(p)$ consists of $\delta^s_1(\epsilon_0)$ and $\delta^u_1(\epsilon_0)$. These conventions lead to the following definitions.

\begin{defi}\label{Defi:converge in a sector }
	Let $\{x_n\}$ be a sequence converging to $p$. We say that $\{x_n\}$ \emph{converges to $p$ in the sector} $j$ if and only if there exists $N \in \mathbb{N}$ such that for every $n > N$, $x_n \in E(\epsilon_0)_j(p)$. The set of sequences converging to $p$ in the sector $j$ is denoted by $E(p)_j$.
\end{defi}

The set of sequences that converge to $p$ in a sector is $\bigcup_{j=1}^{2k} E(p)_j$. We are going to define an equivalence relation on this set.

\begin{defi}\label{Defi: Sector equiv}
	Let $\{x_n\}$ and $\{y_n\}$ be sequences that converge to $p$ in a sector. We say they are in the same sector of $p$, and write $\{x_n\} \sim_q \{y_n\}$, if and only if $\{x_n\}$ and $\{y_n\}$ belong to the same set $E(p)_j$.
\end{defi}

\begin{rema}\label{Rema: caracterisation sim-p}
	The previous definition is equivalent to the existence of $j \in \{1, \ldots, 2k\}$ and $N \in \mathbb{N}$ such that for all $n \geq N$, $x_n, y_n \in E(\epsilon_0)_j(p)$. We use this characterization to prove the following lemma.
\end{rema}

\begin{lemm}\label{Lemm: Equiv relation}
	In the set of sequences that converge to $p$ in a sector, the relation $\sim_q$ is an equivalence relation. Moreover, each equivalence class coincides with a set $E(p)_j$ for some $j \in \{1, \ldots, 2k\}$.
\end{lemm}

\begin{proof}
	The nontrivial property is transitivity. Suppose $\{x_n\} \sim_q \{y_n\}$ and $\{y_n\} \sim_q \{z_n\}$. Then there exist $j, j' \in \{1, \ldots, 2k\}$ and $N_1, N_2 \in \mathbb{N}$ such that $x_n, y_n \in E(\epsilon_0)_j(p)$ for $n > N_1$ and $y_n, z_n \in E(\epsilon_0)_{j'}(p)$ for $n > N_2$. Let $N := \max\{N_1, N_2\}$. Then $y_n \in E(\epsilon_0)_j(p) \cap E(\epsilon_0)_{j'}(p)$ for $n > N$, which implies $j = j'$. Hence, $x_n, z_n \in E(\epsilon_0)_j(p)$ for $n > N$, so $\{x_n\} \sim_q \{z_n\}$.
\end{proof}

\begin{defi}\label{Defi: Sector}
	The equivalence class of sequences converging to $p$ in sector $j$ is called the \emph{sector} $e(p)_j$ of $p$.
\end{defi}

This notion is important in view of the next proposition, which establishes that generalized pseudo-Anosov homeomorphisms have a well-defined action on the sectors of a point.

\begin{lemm}\label{Lemm: image secto is a sector}
	Let $f$ be a generalized pseudo-Anosov homeomorphism and $p$ any point in the underlying surface with $k$ different separatrices. If $\{x_n\} \in e(p)_j$, then there exists a unique $i \in \{1,\ldots, 2k\}$ such that $\{f(x_n)\} \in e(f(p))_i$. In other words, the image of the sector $e(p)_j$ is the sector $e(f(p))_i$.
\end{lemm}

\begin{proof}
	Note that $p$ is a $k$-prong singularity, a regular point, or a spine if and only if $f(p)$ is a $k$-prong singularity, a regular point, or a spine. Therefore, $p$ and $f(p)$ have the same number of sectors. Let $0<\epsilon<\epsilon_0$ be such that $f(D(p,\epsilon)) \subset D(f(p),\epsilon_0)$. Such $\epsilon$ exists because $f$ is continuous.
	
	Let $\delta^s(p)_j$ and $\delta^u(p)_j$ be the separatrices of $p$ that bound the set $E(\epsilon)_j(p) \subset S$. Then $f(\delta^s(p)_j)$ and $f(\delta^u(p)_j)$ are contained in two contiguous separatrices of $f(p)$, which determine a unique set $E(\epsilon_0)_{i}(f(p))$ for some $i\in \{1,\ldots,2k\}$.
	
	Let $N \in \mathbb{N}$ such that, for all $n > N$, $x_n \in E(\epsilon)_j(p)$. Then, for all $n > N$, $f(x_n) \in E(\epsilon_0)_i(f(p))$, and thus the sequence $\{f(x_n)\}$ belongs to the sector $e(f(p))_i$.
\end{proof}

\begin{lemm}\label{Lemm: sector contined unique rectangle}
	Let $f: S \rightarrow S$ be a generalized pseudo-Anosov homeomorphism with a Markov partition $\mathcal{R}$. Let $x \in S$ and let $e$ be a sector of $x$. Then, there exists a unique rectangle in the Markov partition that contains the sector $e$.
\end{lemm}

\begin{proof}
	Let $\{x_n\}$ be a sequence that converges to $x$ within the sector $e$. Consider a canonical neighborhood $U$ of size $\epsilon > 0$ around $x$, and let $E$ be the unique connected component of $U$ minus the local stable and unstable manifolds of $x$ that contains the sequence $\{x_n\}$.
	
	By choosing $\epsilon$ small enough, we can assume that the local stable separatrix $I$ of $x$ that bounds $E$ is contained in at most two rectangles of the Markov partition, and similarly, the local unstable separatrix $J$ of $x$ is contained in at most two rectangles. By choosing the correct side of the local separatrices, we can find rectangles $R$ and $R'$ in the Markov partition, a horizontal subrectangle $H \subset R$ containing $I$ in its upper or lower boundary, and a vertical subrectangle $V \subset R'$ whose upper or left boundary contains $J$. These rectangles can be chosen small enough such that the intersection of their interiors is a rectangle contained within $E$, denoted as $\mathring{Q} := \mathring{H} \cap \mathring{V} \subset E$. This implies that $R = R'$, since the intersection of interiors of distinct rectangles in the Markov partition is empty.
	
	Moreover, by considering a subsequence of $\{x_n\}$, we do not change its equivalence class. Therefore, $\{x_n\} \subset \mathring{Q} \subset R$. This completes the proof.
\end{proof}

\subsubsection{Sector codes} For all $x\in S$, we shall to  construct an element of $\Sigma_{A(f,\cR)}$ that projects to $x$. The \emph{sector codes} we define below will do the job. It was shown in Lemma \ref{Lemm: sector contined unique rectangle} that each sector is contained in a unique rectangle of the Markov partition, and Lemma \ref{Lemm: image secto is a sector} shows that the image of a sector is a sector. This allows for the following definition.

\begin{defi}\label{Defi: Sector codes}
	Let $f$ be pseudo-Anosov homeomorphism and le $\cR=\{R_i\}_{i=1}^n$ be a geometric Markov partition of $f$ such that the incidence matrix $A(f,\cR)$ is binary. Let $x \in S$ be a point with sectors $\{e_1(x),\cdots, e_{2k}(x)\}$ (where $k$ is the number of stable or unstable separatrices in $x$).	The \emph{sector code} of $e_j(x)$ is the sequence:
	\begin{equation}
\be_j(x)=(e(x,j)_z)_{z\in \ZZ} \in \Sigma,
	\end{equation}
given by the rule: $e(x,j)_z := i$, where $i \in \{1,\dots,n\}$ is the index of the unique rectangle in $\cR$ such that the sector $f^z(e_j(x))$ is contained in the rectangle $R_i$.
\end{defi}

 The space $\Sigma$ of bi-infinite sequences is larger than $\Sigma_{f,\cR}$. We need to show that every sector code is, in fact, an \emph{admissible code}, i.e., that $\be_j(x) \in \Sigma_{A(f,\cR)}$.

\begin{lemm}\label{Lemm: sector code is admisible}
	For every $x \in S$, every sector code $\be := \be_j(x)$ is an element of $\Sigma_{A(f,\cR)}$.
\end{lemm}

\begin{proof}
	Let $A = (a_{ij})$ be the incidence matrix. The code $\be = (e_z)$ is in $\Sigma_A$ if and only if for all $z \in \ZZ$, $a_{e_z e_{z+1}} = 1$. By definition, this happens if and only if $f(\overset{o}{R_{e_z}}) \cap \overset{o}{R_{e_{z+1}}} \neq \emptyset$.
	
	Let $\{x_n\}_{n \in \NN}$ be a sequence converging to $f^z(x)$ and contained in the sector $f^z(e)$. By Lemma \ref{Lemm: sector contined unique rectangle}, the sector $f^z(e)$ is contained in a unique rectangle $R_{e_z}$, and we can assume $\{x_n\} \subset R_{e_z}$. Moreover, there exists $N \in \NN$ such that $x_n \in \overset{o}{R_{e_z}}$ for all $n \geq N$. Recall that the sector $f^z(e)$ is bounded by two consecutive local stable and unstable separatrices of $f^z(x)$: $F^s(f^z(x))$ and $F^u(f^z(x))$. If for every $n \in \NN$, $x_n$ is contained in the boundary of $R_{e_z}$, then this boundary component is a local separatrix of $x$ between $F^s(f^z(x))$ and $F^u(f^z(x))$, which is not possible.
	
	Since the image of a sector is a sector, the sequence $\{f(x_n)\}$ converges to $f^{z+1}(x)$ and is contained in the sector $f^{z+1}(e)$. The argument in the last paragraph also applies to this sequence, and $f(x_n) \in \overset{o}{R_{e_{z+1}}}$ for $n$ large enough. This proves that $f(\overset{o}{R_{e_z}}) \cap \overset{o}{R_{e_{z+1}}} \neq \emptyset$.
\end{proof}

The sector codes of a point $x$ are not only admissible; as the following lemma shows, they are, in fact, the only codes in $\Sigma_{A(f,\cR)}$ that project to $x$.  

\begin{lemm}\label{Lemm: every code is sector code }
	If $\bw = (w_z) \in \Sigma_{A(f,\cR)}$ projects under $\pi_{(f,\cR)}$ to $x$, then $\bw$ is equal to a sector code of $x$.
\end{lemm}

\begin{proof}
	For each $n \in \NN$, we take the rectangle $F_n = \cap_{j=-n}^n f^{-j}(\overset{o}{R_{w_j}})$, which is non-empty because $\bw \in \Sigma_{A(f,\cR)}$. The following properties hold:
	
	\begin{itemize}
		\item[i)] $\pi_{A(f,\cR)}(\bw) = x \in \overline{F_n}$ for every $n \in \ZZ$.
		
		\item[ii)] For all $n \in \NN$, $F_{n+1} \subset F_n$.
		
		\item[iii)] For every $n \in \NN$, there exists at least one sector $e$ of $x$ contained in $F_n$. If this were not the case, there would exist $\epsilon > 0$ such that the regular neighborhood of size $\epsilon$ around $x$, given by Theorem \ref{Theo: Regular neighborhood} is disjoint from $F_n$, but $x \in \overline{F_n}$ as was state in item i) and this is a contradiction.
		
		\item[iv)] If the sector $e \subset F_n$, then for every $m \in \ZZ$ such that $\vert m \vert \leq n$:
		$$
		f^m(e) \subset f^m(F_n) = \cap_{j=m-n}^{m+n} f^{-j}(\overset{o}{R_{w_{m-j}}}) \subset \overset{o}{R_m},
		$$
		which implies that $e_m = w_m$ for all $m \in \{-n, \dots, n\}$.
	\end{itemize}
	
	By item ii), if a sector $e$ is not in $F_n$, then $e$ is not in $F_{n+1}$. Together with the fact that for all $n$ there is always a sector in $F_n$ (item iii)), we deduce that there is at least one sector $e$ of $x$ that is contained in $F_n$ for all $n$. Then, we apply point iv) to deduce $e_z = w_z$ for all $z$.
\end{proof}

Let $x$ be a point with $k$ stable and $k$ unstable separatrices. Then $x$ has at most $2k$ sector codes projecting to $x$ and we have next corollary.

\begin{coro}\label{Coro: Caracterisation fibers}
	For all $x \in S$, if $x$ has $k$ separatrices, then $\pi_f^{-1}(x) = \{\underline{e_j(x)}\}_{j=1}^{2k}$. In particular, $\pi_f$ is finite-to-one.
\end{coro}

This ends the proof of Proposition \ref{Prop:proyecion semiconjugacion}.

\subsection{The quotient space is a surface}\label{subsec: quotien surface}

There is a natural equivalence relation in $\Sigma_{A(f,\cR)}$ defined via the projection $\pi_{(f,\cR)}$. Two codes $\bw$ and $\bv$ in $\Sigma_{A(f,\cR)}$ are  $f$-related  if and only if $\pi_{(f,\cR)}(\bw) = \pi_{(f,\cR)}(\bv)$, and the relation is denote by $\bw \sim_{(f,\cR)} \bv$.

The quotient space is denoted by $\Sigma_{(f,\cR)} = \Sigma_{A{(f,\cR)}}/\sim_{(f,\cR)}$ and $[\bw]_{(f,\cR)}$ is the equivalence class of $\bw$. If $\bw\sim_{(f,\cR)} \bv$, then:
$$
[\pi_{(f,\cR)}]([\bw]_{(f,\cR)}) = \pi_{(f,\cR)}(\bw) = \pi_{(f,\cR)}(\bv) = [\pi_{(f,\cR)}]([\bv]_{(f,\cR)}).
$$

Furthermore, since $\pi_{(f,\cR)}: \Sigma_{A(f,\cR)} \to S$ is a continuous function, $\Sigma_{A(f,\cR)}$ is compact, and $S$ is a Hausdorff topological space, the \emph{closed map lemma} implies that $\pi_{(f,\cR)}$ is a closed map. As $\pi_f$ is also surjective and finite to one, it follows that $\pi_{(f,\cR)}$ is a quotient map. Therefore, the projection $\pi_{(f,\cR)}$ induces a homeomorphism $[\pi_{(f,\cR)}]: \Sigma_{(f,\cR)} \rightarrow S$ in the quotient space.

The shift also behaves well under this quotient, since:
$$
[\sigma_{(f,\cR)}) ([\bw]_{(f,\cR)})]_{(f,\cR)} := [\sigma_{A(f,\cR)}(\bw)]_{(f,\cR)},
$$

If $\bw \sim_{(f,\cR)} \bv$, the semi-conjugacy  between of $f$ and $\sigma_{A(f,\cR)}$ through $\pi_{(f,\cR)}$ implies that:
$$
[\sigma_{A(f,\cR)}(\bw)] \sim_{(f,\cR)} [\sigma_{A(f,\cR)}(\bv)].
$$
Thus, the quotient map $[\sigma_{A(f,\cR)}]: \Sigma_{A(f,\cR)} \rightarrow \Sigma_{(f,\cR)}$ is well-defined and, in fact, a homeomorphism. Moreover,
\begin{eqnarray*}
	[\pi_{(f,\cR)}] \circ [\sigma_{A(f,\cR)}]([\bw]_{(f,\cR)})
	&=& [\pi_{(f,\cR)}]([\sigma_{A(f,\cR)}(\bw)]_{(f,\cR)}) \\
	&=& \pi_{(f,\cR)}(\sigma{A(f,\cR)}(\bw)) \\
	&=& f \circ \pi_{(f,\cR)}(\bw) \\
	&=& f \circ [\pi_{(f,\cR)}]([\bw]_{(f,\cR)}).
\end{eqnarray*}

Therefore, $[\pi_f]$ determines a topological conjugacy between $f$ and $[\sigma]_{A(f,\cR)}$.  
This implies that $\Sigma_{A(f,\cR)}$ is a surface homeomorphic to $S$, and $[\sigma]_{(f,\cR)}$ is topologically conjugate to a pseudo-Anosov homeomorphism. We summarize this discussion in the following proposition.

\begin{prop}\label{Prop: quotien by f}
	The quotient space $\Sigma_{(f,\cR)} := \Sigma_{A(f,\cR)}/\sim_{(f,\cR)}$ is homeomorphic to the surface $S$, and the quotient shift $[\sigma_{A(f,\cR)}]: \Sigma_{(f,\cR)} \rightarrow \Sigma_{(f,\cR)}$ is  topologically conjugate to the  pseudo-Anosov homeomorphism $f: S \rightarrow S.$ via the quotient projection:
	$$
	[\pi_{(f,\cR)}]: \Sigma_{(f,\cR)} \rightarrow S.
	$$
\end{prop}

If we have two pseudo-Anosov maps $f: S_f \rightarrow S_f$ and $g: S_g \rightarrow S_g$ with respective geometric  Markov partitions $\cR_f$ and $\cR_g$  such that the  geometric type of the pairs is the same, maybe n after a horizontal refinement, they share the same incidence matrix $A=A(f,\cR_f)=A(g,\cR_g)$ with entries in $\{0,1\}$ and are associated with the same sub shift of finite type $(\Sigma_A, \sigma_A)$. 

However, the projections $\pi_{(f,\cR_f)}$ and $\pi_{(g,\cR_g)}$ are not necessarily the same o we cant even compose them. In particular, while $\Sigma_{A(f,\cR_f)}/\sim_{(f,\cR_f)}$ is homeomorphic to $S_f$, and $\Sigma_{A(g,\cR_g)}/\sim_{(g,\cR_g)}$ is homeomorphic to $S_g$, we cannot conclude that $S_f$ is homeomorphic to $S_g$.To convince yourself observe that, given $x \in S_f$ and $y \in S_g$, we do not know whether $\pi_{(f,\cR_f)}^{-1}(x) \cap \pi_{(g,\cR_g)}^{-1}(y) \neq \emptyset$ implies that $\pi_{(f,\cR_f)}^{-1}(x) = \pi_{(g,\cR_g)}^{-1}(y)$.

It was shown in Lemma \ref{Lemm: every code is sector code } that every code in $\pi_{(f,\cR_f)}^{-1}(x)$ is a sector code of $x$. Therefore, if $\pi_{(f,\cR_f)}^{-1}(x) \cap \pi_{(g,\cR_g)}^{-1}(y) \neq \emptyset$, there is a common sector code for both $x$ and $y$, but this does not imply a unique (or continuous) correspondence between the sets of sectors of $x$ and $y$. For example, it is possible that $x$ has a different number of prongs than $y$. This ambiguity cannot be resolved by examining the incidence matrix alone, but it is addressed by incorporating the geometric type.

In the next section, we will construct an equivalence relation $\sim_T$ in $\Sigma_{A(f,\cR)}$ in terms of the geometric type $T$, in such a manner that, if the geometric types of $(f,\cR_f)$ and $(g,\cR_g)$ are both equal to $T$, then the quotient sub-shifts $[\sigma_{A(f,\cR_f)}]$ and $[\sigma_{A(g,\cR_g)}]$ are topologically conjugate. This will be enough to prove our main theorem \ref{Theo: Total invariant}, as it will follow that $f$ and $g$ must be topologically conjugate.

\section{The gemetric type induce a decomposition of the shift space.}

The objective of this section is give a constructive proof of the following proposition.

\begin{prop}\label{Prop: The relation determines projections}
	Let $T \in \cG(\textbf{p-A})^{sp}$ be a symbolically presentable geometric type, and let $A := A(T)$ be its incidence matrix. Consider the subshift of finite type $(\Sigma_A, \sigma_A)$ associated with $T$.
	
	Then there exists an equivalence relation $\sim_T$ on $\Sigma_A$, algorithmically determined by $T$, such that the following holds:
	
	If $(f, \mathcal{R})$ is a pair that realizes the geometric type $T$, and if $\pi_{(f, \mathcal{R})} : \Sigma_A \to S$ is the projection induced by $(f, \mathcal{R})$, then for any pair of codes $\bw, \bv \in \Sigma_A$, we have
	$$
	\bw \sim_T \bv \quad \text{if and only if} \quad \pi_{(f, \mathcal{R})}(\bw) = \pi_{(f, \mathcal{R})}(\bv),
	$$
	that is, their projections coincide.
\end{prop}

We start by decomposing $\Sigma_{A(T)}$ into three subsets: $\Sigma_{I(T)}$, $\Sigma_{S(T)}$, and $\Sigma_{U(T)}$, corresponding to the interior, the $s$-boundary, and the $u$-boundary codes of $T$, respectively. Using the information encoded in $T$ and some hand-crafted techniques, we introduce three relations, $\sim_I$, $\sim_S$, and $\sim_U$, defined on these subsets. These will later be extended to an equivalence relation $\sim_T$ on the entire space $\Sigma_A$.

Finally, we will prove that for any pair $(f, \mathcal{R})$ realizing $T$, we have
$$
\pi_{(f, \mathcal{R})}(\bw) = \pi_{(f, \mathcal{R})}(\bv) \quad \text{if and only if} \quad \bw \sim_T \bv.
$$
The subsets and the equivalence relation $\sim_T$ will be constructed step by step throughout this section. To proceed, we fix some notation that will be used throughout:

\begin{itemize}
	\item $T \in \cG(\textbf{p-A})^{sp}$ is a symbolically presentable geometric type, given by:
	$$
	T = \left(n, \{h_i, v_i, \rho, \epsilon\} \right).
	$$
	
	\item Its incidence matrix is denoted by $A := A(T)$.
	
	\item The pair $(f, \mathcal{R})$ have geometric type, $T$ and the invariant foliations of $f$ are $(\cF^s,\mu^s)$ and $(\cF^u,\mu^u)$.
	
	\item To avoid excessive notation, we write $\pi_f$ instead of $\pi_{(f, \mathcal{R})}$ when the choice of homeomorphism and partition is clear from context, especially while we develop a proof.
\end{itemize}

\subsection{Periodic  points and codes} 

The following lemma will be used repeatedly in the arguments that follow.

\begin{lemm}\label{Lemm: Boundary of Markov partition is periodic}
	Both the upper and lower boundaries of each rectangle in the Markov partition $\mathcal{R}=\{R_i\}_{i=1}^n$ lie on the stable leaf of some periodic point of $f$. Similarly, the left and right boundaries lie on the unstable leaf of some periodic point, of period at most $2n$
\end{lemm}

\begin{proof}
	Let $x$ be a point on the stable boundary of some rectangle $R_i \in \mathcal{R}$. Since the stable boundary of the Markov partition is $f$-invariant, for all $n \geq 0$, the point $f^n(x)$ remains on the stable boundary of some rectangle in the partition.
	
	As there are only $2n$ stable boundary components in the partition, there exist integers $n_1, n_2 \in \{1, \ldots, 2n\}$ with $n_1 < n_2$ such that $f^{n_1}(x)$ and $f^{n_2}(x)$ lie on the same stable boundary component. This implies that the stable leaf containing $x$ is periodic, with period less than $2n$, and hence corresponds to the stable leaf of some periodic point $p$ of period less than or equal to $2n$.
	
	A similar argument applies to the unstable (vertical) boundaries.
\end{proof}

The following lemma is classical in the literature of symbolic dynamical systems, so we state it without proof.

\begin{lemm}\label{Lemm: Periodic to periodic}
	Let $T \in \cG\cT(\textbf{p-A})^{sp}$, and let $(f, \mathcal{R})$ be a pair realizing $T$. If $\bw \in \textbf{Per}(\sigma_{A(T)})$ is a periodic code, then $\pi_{(f, \mathcal{R})}(\bw)$ is a periodic point of $f$. Moreover, if $p$ is a periodic point of $f$, then $\pi_{(f, \mathcal{R})}^{-1}(p) \subset \textbf{Per}(\sigma_{A(T)})$; that is, all codes projecting to $p$ are periodic.
\end{lemm}

The following lemma characterizes the periodic boundary points of $(f, \mathcal{R})$ in terms of their iterations under $f$. 

\begin{lemm} \label{Lemm: no periodic boundary points}
	Let $\bw \in \Sigma_{A(T)}$ be a code such that for every $k \in \mathbb{Z}$, $f^k(\pi_{(f,\cR)}(\bw)) \in \partial^s \mathcal{R}$. Then $\pi_{(f,\cR)}(\bw)$ is a periodic point of $f$. Similarly, if $f^k(\pi_{(f,\cR)}(\bw)) \in \partial^u \mathcal{R}$ for all integers $k$, then $\bw$ is a periodic code.
\end{lemm}

\begin{proof}
	Suppose that $\bw$ is non-periodic. By Lemma~\ref{Lemm: Periodic to periodic}, $x := \pi_f(\bw)$ is non-periodic and lies on the stable boundary of $\mathcal{R}$, which is a compact set and each  connected  component have finite $\mu^u$-length.
	
	Let $[x, p]^s$ be the stable segment joining $x$ to a periodic point $p$ on its stable leaf, which exists by Lemma~\ref{Lemm: Boundary of Markov partition is periodic}. Then, for all $m \in \mathbb{N}$,
	$$
	\mu^u([f^{-m}(x), f^{-m}(p)]^s) = \mu^u(f^{-m}[x, p]^s) = \lambda^m \mu^u([x, p]^s),
	$$
	which diverges as $m \to \infty$ and can be contained in a stable interval of fixed length. Therefore, there exists $m \in \mathbb{N}$ such that $f^{-m}(x) \notin \partial^s \mathcal{R}$, contradicting the hypothesis. A similar argument applies to the unstable case.
\end{proof}

The proof of Corollary~\ref{Coro: interior periodic points unique code} follows from the fact that all sector codes of a point satisfying the hypothesis are equal.

\begin{coro}\label{Coro: interior periodic points unique code}
	Let $\bw \in \Sigma_{A(T)}$ be a periodic code. If $\pi_{(f,\cR)}(\bw) \in \overset{o}{\mathcal{R}}$, then for all $z \in \mathbb{Z}$, $f^z(\pi_{(f,\cR)}(\bw)) \in \overset{o}{\mathcal{R}}$; that is, it remains in the interior of the Markov partition. Moreover, $\pi_{(f,\cR)}^{-1}(\pi_{(f,\cR)}(\bw)) = \{\bw\}$, meaning the code is unique.
\end{coro}

\begin{proof}
	Suppose that for some $z \in \mathbb{Z}$, $f^z(x)$ lies on the stable (or unstable) boundary of the Markov partition. Since $x = \pi_f(\bw))$ is periodic and the stable and unstable boundaries are $f$-invariant, the entire orbit of $x$ would lie in the boundary of $\mathcal{R}$. This contradicts the assumption that $x \in \overset{o}{\mathcal{R}}$.
		Therefore, $f^z(x) \in \overset{o}{\mathcal{R}}$ for all $z \in \mathbb{Z}$. In this case, all sector codes of $x$ are contained within the interior of the same rectangle, and hence all its sector codes are identical. This implies that the preimage $\pi_f^{-1}(x)$ consists of a single code.
\end{proof}

\subsection{Totally interior points} The fact that the entire orbit of a point $x$ remains in the interior of the partition is a key property, as it distinguishes those points lying outside the stable and unstable laminations of periodic boundary points. We now introduce a name for such points.

\begin{defi}\label{Defi: totally interior points}
	Let $(f, \mathcal{R})$ be a realization of $T$. The \emph{totally interior points} of $(f, \mathcal{R})$ are those points $x \in \bigcup \mathcal{R} = S$ such that for all $z \in \mathbb{Z}$, $f^z(x) \in \overset{o}{\mathcal{R}}$. This set is denoted by $\mathrm{Int}(f, \mathcal{R}) \subset S$.
\end{defi}

Now we can characterizes the points of the surface $S$ where the projection $\pi_f$ is invertible as the \emph{totally interior points} of $\cR$.

\begin{prop}\label{Prop: Characterization injectivity of pi_f}
	Let $x$ be any point in $S$. Then $\vert \pi_f^{-1}(x) \vert = 1$ if and only if $x$ is a totally interior point of $\mathcal{R}$.
\end{prop}

\begin{proof}
	Suppose $\vert \pi_f^{-1}(x) \vert = 1$. Then, for all $z \in \mathbb{Z}$, the sector codes of $f^z(x)$ are all equal. Therefore, for all $z \in \mathbb{Z}$, the sectors of $f^z(x)$ lie entirely within the interior of the same rectangle. Moreover, the union of all these sectors forms an open neighborhood of $f^z(x)$ contained in the interior of a rectangle. Hence, $f^z(x) \in \overset{o}{\mathcal{R}}$.
	
	Conversely, if $f^z(x) \in \overset{o}{\mathcal{R}}$ for all $z \in \mathbb{Z}$, then all sectors of $f^z(x)$ lie in the same rectangle. This implies that the sector codes of $x$ coincide. By Lemma \ref{Lemm: every code is sector code }, it follows that $\vert \pi_f^{-1}(x) \vert = 1$.
\end{proof}

\begin{lemm}\label{Lemm: Characterization unique codes}
	A point $x \in S$ is a totally interior point of $\mathcal{R}$ if and only if it does not lie on the stable or unstable leaf of any periodic boundary point of $\mathcal{R}$.
\end{lemm}

\begin{proof}
	Suppose that for all $z \in \mathbb{Z}$, $f^z(x) \in \overset{o}{\mathcal{R}}$, and that $x$ lies on the stable (or unstable) leaf of some periodic boundary point $p \in \partial^s R_i$ (respectively, $p \in \partial^u R_i$). The contraction (or expansion) along the stable (or unstable) leaf of $p$ implies the existence of some $z \in \mathbb{Z}$ such that $f^z(x) \in \partial^s R_i$ (respectively, $f^z(x) \in \partial^u R_i$), contradicting the assumption.
	
	Conversely, Lemma \ref{Lemm: Boundary of Markov partition is periodic} implies that the only stable or unstable leaves intersecting the boundary $\partial^{s,u} R_i$ are those of the $s$-boundary and $u$-boundary periodic points. Therefore, if $x$ is not on these laminations, neither are its iterates $f^z(x)$, and hence $f^z(x) \in \overset{o}{\mathcal{R}}$ for all $z \in \mathbb{Z}$.
\end{proof}

\subsubsection{Codes that project to the stable or unstable boundary lamination.}

We need to determine the subset of codes in $\Sigma_{A(T)}$ that project onto the stable and unstable leaves of the periodic boundary points. To achieve this, we introduce an abstract family of codes in $\Sigma_{A(T)}$ and then proceed to prove that these are precisely the codes that project to such laminations on the surface.

\begin{defi}\label{Defi: s,u-leafs}
	Let $\bw \in \Sigma_{A(T)}$. The set of \emph{stable leaf codes} of $\bw$ is defined as
	$$
	\underline{F}^s(\bw) := \{\bv \in \Sigma_{A(T)} : \exists Z \in \mathbb{Z} \text{ such that } v_z = w_z \text{ for all } z \geq Z\}.
	$$
	Similarly, the set of \emph{unstable leaf codes} of $\underline{w}$ is defined by
	$$
	\underline{F}^u(\bw) := \{\bv \in \Sigma_{A(T)} : \exists Z \in \mathbb{Z} \text{ such that } v_z = w_z \text{ for all } z \leq Z\}.
	$$
\end{defi}

Let us introduce some notation. For $\bw \in \Sigma_{A(T)}$, we define its positive part as
$$
\bw_+ := (w_n)_{n \in \mathbb{N}} \quad \text{(with } 0 \in \mathbb{N}),
$$
and its negative part as
$$
\bw_- := (w_{-n})_{n \in \mathbb{N}}.
$$
We denote by $\Sigma_{A(T)}^+$ the set of positive codes of $\Sigma_{A(T)}$, i.e., the collection of all positive parts of codes in $\Sigma_{A(T)}$. Similarly, $\Sigma_{A(T)}^-$ denotes the set of negative codes of $\Sigma_{A(T)}$.

Let $x \in S$ be a point. We denote by $F^s(x)$ the stable leaf  of $\mathcal{F}^s$ passing through $x$, and by $F^u(x)$ the unstable leaf of $\mathcal{F}^u(f)$ passing through $x$.

\begin{prop}\label{Prop: Projection foliations}
	Let $\bw \in \Sigma_{A(T)}$. Then,
	$$
	\pi_f(\underline{F}^s(\bw)) \subset F^s(\pi_f(\bw)).
	$$
	Furthermore, suppose $\pi_f(\bw) = x$.
	\begin{itemize}
		\item If $x$ is not a $u$-boundary point, then for every $y \in F^u(x)$, there exists a code $\bv \in \underline{F}^s(\bw)$ such that $\pi_f(\bv) = y$.
		
		\item If $x$ is a $u$-boundary point and $\bw_0 = w_0$, then for every $y$ in the stable separatrix of $x$ that enters the rectangle $R_{w_0}$, there exists a code $\bv \in \underline{F}^s(\bw)$ projecting to $y$, i.e., $\pi_f(\bv) = y$.
	\end{itemize}
	
	A similar statement holds for the unstable manifold of $\bw$ and its corresponding projection onto the unstable manifold of $\pi_f(\bw)$.
\end{prop}

\begin{proof}
	
	Let $\bv \in \underline{F}^s(\bw)$ be a stable leaf code of $\bw$. By the definition of stable leaf codes, we have $w_z = v_z$ for all $z \geq k$ for some $k \in \mathbb{N}$. Consequently, $\pi_f(\sigma^k(\bw)), \pi_f(\sigma^k(\bv)) \in R_{w_k}$. Since the positive parts of $\bw$ and $\bv$ coincide from $k$ onwards, they determine the same horizontal sub-rectangles of $R_{w_k}$ where the codes $\sigma^k(\bw)$ and $\sigma^k(\bv)$ are projected. For $n \in \mathbb{N}$, let $H_n$ be the rectangle defined by
	$$
	H_n = \bigcap_{z=0}^n f^{-z}(R_{w_{z+k}}) = \bigcap_{z=0}^n f^{-z}(R_{v_{z+k}}).
	$$
	The intersection of all $H_n$ forms a stable segment of $R_{w_k}$. Moreover, each $H_n$ contains the rectangles
	$$
	\overset{o}{Q_n} = \bigcap_{z=-n}^{n} f^{-z}(\overset{o}{R_{w_{z+k}}}) = \bigcap_{z=-n}^{n} f^{-z}(\overset{o}{R_{v_{z+k}}}).
	$$
	Therefore, the projections $\pi_f(\bv)$ and $\pi_f(\bw)$ lie on the same stable leaf. 
	
	\begin{figure}[h]
		\centering
		\includegraphics[width=1\textwidth]{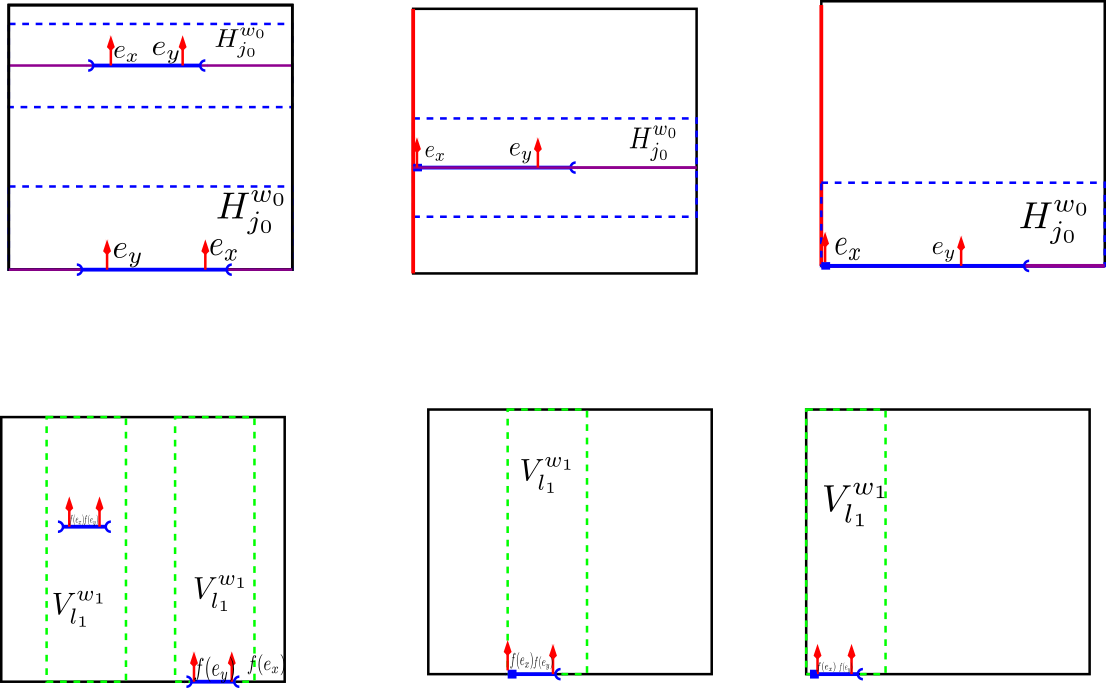}
		\caption{Projections of the stable leaf codes}
		\label{Fig: Proyection fol}
	\end{figure}
	
	For the other part of the argument, we refer to the illustration in Figure \ref{Fig: Proyection fol} for some visual intuition. Consider the case when $x = \pi_f(\bw)$ is not a $u$-boundary point. Recall that the incidence matrix $A$ is binary.
	
	Let $y \in F^s(x)$ with $y \neq x$. Suppose $y$ is not periodic (if it is, replace $y$ by $x$). Since $x$ is not a $u$-boundary point, there exists a small interval $I$ properly contained in the stable segment of $\mathcal{R}_{w_0}$ passing through $x$. This allows us to apply the following argument on both stable separatrices of $x$. By the definition of stable leaf, there exists $k \in \mathbb{N}$ such that $f^k(y) \in I$. If we prove there is $\bv \in \underline{F}^s(\bw)$ such that $\pi_f(\bv) = f^k(y)$, then $\sigma^{-k}(\bv) \in \underline{F}^s(\bw)$ is a code projecting to $y$, i.e., $\pi_f(\sigma^{-k}(\bv)) = y$. Thus, we can assume $y \in I$. This corresponds to the left side of Figure \ref{Fig: Proyection fol}.
	
	There are two possibilities for $f^z(y)$: either it lies on the stable boundary of $R_{w_0}$ or it does not. In either case, the code $\bw$ corresponds to a sector code $\be_x$ of $x$. The sector $e_x$ lies inside a unique horizontal sub-rectangle of $R_{w_0}$, denoted $H^{w_0}_{j_0}$. Therefore, the sector $f(\be_x)$ lies inside $f(H^{w_0}_{j_0}) = V^{w_1}_{l_1}$, implying $w_1 = (\be_x)_1$.
	
	Consider the sector $\be_y$ of $y$ such that, in the stable direction, it points towards $x$ and, in the unstable direction, it points the same way as $\be_x$. Thus, $\be_y$ lies inside $H^{w_0}_{j_0}$, so $f(\be_y)$ is contained in $V^{w_1}_{l_1}$ and $(\be_y)_1 = w_1$.
	
	In fact, $f^{n}(\be_y)$ and $f^{n}(\be_x)$ lie in the same rectangle $R_{w_n}$ for all $n \in \mathbb{N}$. We deduce that the positive part of $\be_y$ coincides with the positive part of $\bw$. Hence, $\be_y \in \underline{F}^s(\bw)$.
	
	
	In the case where $x \in \partial^u \mathcal{R}$, there is a slight variation. Suppose $\bw = \be_x$, where $\be_x$ is a sector of $x$. This sector code lies inside a unique horizontal sub-rectangle of $R_{w_0}$, $H^{w_0}_{j_0}$, such that $f(H^{w_0}_{j_0}) = V^{w_1}_{l_1}$. This horizontal sub-rectangle contains a unique stable interval $I$ with $x$ as one of its endpoints. Consider $y \in I$.
	
	Similarly to the previous case, define a sector $\be_y$ of $y$ contained in $H^{w_0}_{j_0}$, i.e., $\be_y = w_0$. This implies $f(\be_y)$ is contained in $V^{w_1}_{l_1}$. Then, $(\be_y)_1 = w_1$. Applying this inductively for all $n \in \mathbb{N}$, we get $(\be_y)_n = w_n$; hence, $\be_y \in \underline{F}^s(\bw)$ and projects to $y$.
\end{proof}

\subsection{Boundary codes and $s,u$-generating functions
	\texorpdfstring{ }{ (s,u-generating functions)}}

We proceed with the construction of the codes that project onto the boundary of the Markov partition, $\partial^{s,u}\mathcal{R}$. Assume that $\mathcal{R} = \{R_i\}_{i=1}^n$ is a geometric Markov partition of $f$ with geometric type $T$. For each $i \in \{1, \ldots, n\}$, we label the boundary components of $R_i$ as follows:

\begin{itemize}
	\item $\partial^s_{+1} R_i$ denotes the upper stable boundary of the rectangle $R_i$.
	\item $\partial^s_{-1} R_i$ denotes the lower stable boundary of $R_i$.
	\item $\partial^u_{-1} R_i$ denotes the left unstable boundary of $R_i$.
	\item $\partial^u_{+1} R_i$ denotes the right unstable boundary of $R_i$.
\end{itemize}

Using these labeling conventions we introduce the following definitions.

\begin{defi}\label{Defi: s,u boundary labels of T}
	Let $T = \{n, \{(h_i, v_i)\}_{i=1}^n, \Phi_T\}$ be an abstract geometric type. The $s$-\emph{boundary labels} of $T$ are defined as the formal set
	
\begin{equation}
	\mathcal{S}(T) := \{(i, \epsilon) : i \in \{1, \ldots, n\}, \, \epsilon \in \{1, -1\}\}.
\end{equation}
	Similarly, the $u$-\emph{boundary labels} of $T$ are defined as the formal set
\begin{equation}
\mathcal{U}(T) := \{(k, \epsilon) : k \in \{1, \ldots, n\}, \, \epsilon \in \{1, -1\}\}.
\end{equation}
	
\end{defi}

It is important to note that this definition depends only on the value of $n$ in the geometric type $T$, and therefore does not rely on any specific realization. We now define an inclusion from the set of $s$-boundary labels into the set of horizontal labels of the geometric type, via the function:
$$
\theta : \cS(T) \to \mathcal{H}(T),
$$
which is defined as:
\begin{equation}\label{Equa: theta T relabel}
	\theta(i,-1) = 1 \quad \text{and} \quad \theta(i,1) = h_i.
\end{equation}

Since the $s$-boundary $\partial^s_{-1}R_1 = \partial^s_{-1}H^i_1$, and $\partial^s_{+1}R_1 = \partial^s_{+1}H^i_{h_i}$, the function $\theta$ chooses the sub-rectangle of $R_i$ that contains the stable boundary we are analyzing.

Take $(i,\epsilon) \in \cS(T)$ and look at Figure \ref{Fig: theta T}. For $\delta \in \{-1,1\}$, where is $f(\partial^s_{\delta}R_i)$ located? We can use the geometric type to answer this question. Assume that $\delta = 1$ and $\rho(i,h_i) = (k,l)$:
\begin{itemize}
	\item Since $\partial^s_{+1}R_i = \partial^s_{+1}H^i_j$, then $f(\partial^s_{+1}R_i) \subset \partial^s R_k$.
	\item If $\epsilon(i,h_i) = 1$, then $f(\partial^s_{+1}R_i) \subset \partial^s_{+1}R_k$; and if $\epsilon(i,h_i) = -1$, then $f(\partial^s_{+1}R_i) \subset \partial^s_{-1}R_k$.
\end{itemize}

 	\begin{figure}[h]
	\centering
	\includegraphics[width=0.6\textwidth]{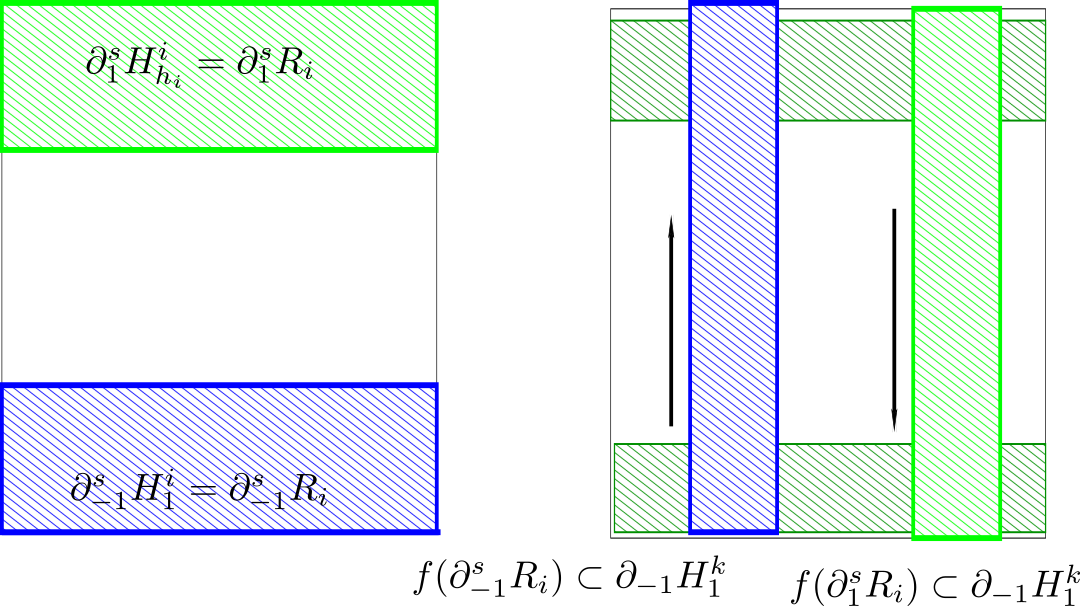}
	\caption{The $s$-generation function}
	\label{Fig: theta T}
\end{figure}

In the analogous  situation where, $\delta = -1$ and $\rho(i,1) = (k,l)$: 
\begin{itemize}
	\item Since $\partial^s_{-1}R_i = \partial^s_{-1}H^i_1$, then $f(\partial^s_{-1}R_i) \subset \partial^s R_k$.
	\item If $\epsilon(i,h_i) = 1$, then $f(\partial^s_{-1}R_i) \subset \partial^s_{-1}R_k$; and if $\epsilon(i,h_i) = -1$, then $f(\partial^s_{-1}R_i) \subset \partial^s_{+1}R_k$.
\end{itemize}

In this situation, if the original label is $(i_0,\delta_0)$, the index $j_1$ of the corresponding rectangle in $\cR$ is given by the formula 
$$
i_1 = \bp_1\circ\rho(i_0, \theta(i_0, \delta_0)),
$$

where $\bp_1: \cH(T) \to \{1, \ldots, n\}$ is the projection onto the first coordinate,  while the corresponding $s$-boundary component of $R_{i_1}$ is determined by 
 $$
 \delta_1 = \delta_0 \cdot \epsilon(i_0, \theta(i_0, \delta_0))
 $$

The $s$-generating function introduced below is a map $\cS(T) \to \cS(T)$ such that, given a label $(i_0,\delta_0)$, it returns $(i_1,\delta_1)$ if and only if $f(\partial^s_{\delta_0}R_{i_0}) \subset \partial^s_{\delta_1} R_{i_1}$. In this way, positive iterations of the $s$-generating function recover the orbit of the corresponding boundary component.

\begin{defi}\label{Defi: s-boundary generating funtion}
	The $s$-generating function of $T$ is the function
	$$
	\Gamma(T): \cS(T) \rightarrow \cS(T),
	$$
	such that to every $s$-boundary label $(i_0,\delta_0) \in \mathcal{S}(T)$ it associate:
	
	\begin{equation}\label{Equ: Gamma generation funtion}
		\Gamma(T)(i_0,\delta_0) = (\bp_1\circ \rho(i_0,\theta(i_0,\epsilon_0)), \,
		\delta_0 \cdot \epsilon(i_0,\theta(i_0,\epsilon_0)) ) :=(i_1,\delta_1).
	\end{equation}
	
\end{defi}
We shall to introduce a similar $u$-generating function that captures the negative orbit of each $u$-boundary component of the rectangles in the Markov partition. For this reason, we define
$$
\eta : \cU(T) \to \mathcal{V}(T),
$$
as the map given by:
\begin{equation}\label{Equa: eta relabel}
	\eta(i,-1) = 1 \quad \text{and} \quad \eta(i,1) = v_i.
\end{equation}

\begin{defi}\label{Defi: u-boundary generating funtion}
	The $u$-\emph{generating function} of $T$ is
	$$
	\Upsilon(T): \cU(T) \rightarrow \cU(T),
	$$
	such that for every $u$-boundary label $(k_0,\delta_0) \in \mathcal{S}(T)$, it assigns:
	\begin{equation}\label{Equ: Upsilon generation funtion}
		\Upsilon(T)(k_0,\delta_0) = \left( \bp_1 \circ \rho^{-1}(k_0,\eta(k_0,\delta_0)), \,
		\delta_0 \cdot \epsilon\left( \rho^{-1}( \eta(k_0,\delta_0) ) \right) \right) := (k_1,\delta_1).
	\end{equation}
\end{defi}

We will focus on state and prove some result about the $s$-generating function, the ideas for the $u$-generation function are totally symmetric. The orbit of  under $\Gamma(T)$ of $(i_0,\delta_0)\in \cS(T)$ is the set:
$$
\{(i_m,\delta_m)\}_{m\in \NN} := \{\Gamma(T)^m(i_0,\delta_0) : m \in \NN\}.
$$

Let $\bp_1: \cS(T),\cU(T) \rightarrow \{1,\dots,n\}$ be the projection onto the first component. This allows us to produce positive and negative codes associated with every $s$- or $u$-boundary label:

\begin{eqnarray}\label{Equa: $s$-boundary code +}
	\underline{I}^{+}(i_0,\delta_0) := \{ \bp_1 \circ \Gamma(T)^m(i_0,\delta_0) \}_{m \in \NN} = \{i_m\}_{m \in \NN}
\end{eqnarray}

\begin{eqnarray}\label{Equa: $u$-boundary code -}
	\underline{J}^{-}(k_0,\delta_0) := \{ \bp_1 \circ \Upsilon(T)^m(k_0,\delta_0) \}_{m \in \NN} = \{k_{-m}\}_{m \in \NN}
\end{eqnarray}

\begin{defi}\label{Defi: positive negative boundary codes}
	
	Let $(i_0,\delta_0) \in \cS(T)$, and let $\underline{I}^+(i_0,\delta_0) \in \Sigma^+$ be its $s$-\emph{boundary positive code}, as determined by Equation \ref{Equa: $s$-boundary code +}. The set of $s$-boundary positive codes of $T$ is denoted as:
	$$
	\underline{\cS}^{+}(T) = \{\underline{I}^{+}(i,\delta) : (i,\delta) \in \cS(T)\} \subset \Sigma^+.
	$$
	
	Let $(k_0,\delta_0) \in \cU(T)$, and let $\underline{J}^{-}(k_0,\delta_0) = \{k_{-m}\}_{m=0}^\infty$ be its $u$-\emph{boundary negative code}, as determined by Equation \ref{Equa: $u$-boundary code -}. The set of $u$-boundary negative codes of $T$ is denoted as:
	$$
	\underline{\cU}^{-}(T) = \{\underline{J}^{-}(k,\delta) : (k,\delta) \in \cU(T)\} \subset \Sigma^{-}.
	$$
	
\end{defi}

\begin{lemm}\label{Lemm: positive code admisible}
	Every $s$-boundary positive code $\underline{I}^{+}(i_0,\delta_0)$ belongs to $\Sigma_A^+$. Similarly, every $u$-boundary negative code $\underline{J}^{-}(k_0,\delta_0)$ belongs to $\Sigma_A^-$.
\end{lemm}

\begin{proof}
The incidence matrix $A(T)$  is a binary matrix and we shall use this hypothesis. In the construction of the positive code $\underline{I}^{+}(i_0,\delta_0)$, the entry $i_1$ corresponds to the index of the unique rectangle (as $A(T)$ is binary) in $\cR$ such that: 
$$
f(H^{i_0}_{\theta(i_0, \delta_0)}) = V^{i_1}_{l_1},
$$ 
where $\theta(i_0, \delta_0) = 1$ or $h_i$. In either case, we have:
$$
f^{-1}(\overset{\circ}{R_{i_1}}) \cap \overset{\circ}{R_{i_0}} = \overset{\circ}{H^{i_0}_{\theta(i_0, \delta_0)}} \neq \emptyset,
$$
and therefore $a_{i_0, i_1} = 1$. By induction, the same reasoning applies to the entire sequence, so that $a_{i_n, i_{n+1}} = 1$ for all $n \in \mathbb{N}$. This shows that $\underline{I}^{+}(i_0,\delta_0)$ is the positive part of an admissible code, completing the proof.

A similar argument applies to the $u$-boundary negative codes.
\end{proof}

Our first step in order to  prove that each $s$-boundary label uniquely determines a set of codes that projects to a single boundary component of the partition is the following Proposition.

\begin{prop}\label{Prop: s code Inyective}	
	The map $I: \cS(T) \rightarrow \Sigma_{A(T)}^+$ defined by $I(i,\delta) := \underline{I}^+(i,\delta)$ is injective. Similarly, the map $J: \cU(T) \rightarrow \Sigma_{A(T)}^-$ defined by $J(i,\delta) := \underline{J}^-(i,\delta)$ is also injective.
\end{prop}

\begin{proof}
	If $i_0 \neq i_0'$, then the sequences $\underline{I}^+(i_0,\delta_0)$ and $\underline{I}^+(i_0',\delta_0')$ differ in their first term. In the remaining case, where $i_0 = i_0'$, we compare $\underline{I}^+(i_0,1) = \{i_m\}$ and $\underline{I}^+(i_0,-1) = \{i'_m\}$, corresponding to the sequences $\{\Gamma(T)^m(i_0,1)\}$ and $\{\Gamma(T)^m(i_0,-1)\}$, and aim to show that there exists $m \in \mathbb{N}$ such that $i_m \neq i'_m$. Our approach begins with the following technical lemma:

\begin{lemm}\label{Lemm: positive h-i}
	If $T$ is a symbolically presentable geometric type and $\underline{I}^+(i_0,1) = \{i_m\}_{m \in \mathbb{N}}$ is a $s$-boundary positive code,  then there exists $M \in \mathbb{N}$ such that  the rectangle $R_{i_M}$ has more than one horizontal subrectangle, i.e., $h_{i_M} > 1$.
\end{lemm}

\begin{proof}
	Since $T$ is in the pseudo-Anosov class, there exists a realization $(f,\mathcal{R})$ of $T$. The sequence $\{i_m\}$ takes only finitely many values (at most $n$), so there exist $m_1 < m_2$ such that $R_{i_{m_1}} = R_{i_{m_2}}$. If $h_{i_m} = 1$ for all $m_1 \leq m \leq m_2$, then $f^{m_2 - m_1}(R_{i_{m_1}})$ is contained in a vertical subrectangle of $R_{i_{m_1}}$.
	
	Due to vertical expansion, this would imply that $R_{i_{m_1}}$ eventually collapses to a stable interval, contradicting the definition of a rectangle in a Markov partition. Hence, there must exist $M \geq 0$ such that $h_{i_M} > 1$.
\end{proof}

In view of Lemma \ref{Lemm: positive h-i}, the value  
$$
M := \min \{ m \in \mathbb{N} : h_{i_m} > 1 \}
$$
exists. If there is some $0 \leq m \leq M$ such that $i_m \neq i'_m$, then the proof is complete. Otherwise, the following lemma addresses the remaining case.

\begin{lemm}\label{Lemm: diferen positive code}
	Suppose that for all $0 \leq m \leq M$, we have $i_m = i'_m$. Then $i_{M+1} \neq i'_{M+1}$.
\end{lemm}

\begin{proof}
Observe that $\delta_0 = -\delta'_0$ and for all $0 \leq m \leq M$, if $\Gamma(T)^m(i_0,+1) = (i_m, \epsilon_m)$, then
$$
\Gamma(T)^m(i_0, -1) = (i_m, -\epsilon_m),
$$

they have the same index $i_m = i'_m$ but opposite sign in the second component, i.e., $\delta_m = -\delta'_m$. \\
In effect, without loss of generality we can assume $\delta_0 = 1$ and $\delta'_0 = -1$ and by hypothesis $1 = h_{i_0} = h_{i'_0}$ therefore: 
$$
\theta(i_0, \delta_0) = (i_0, h_{i_0}) = (i'_0, 1) = \theta(i'_0, \delta'_0).
$$
By the equation \ref{Equ: Gamma generation funtion} of the $s$-generating function:
$$
\delta_1 = \delta_0 \cdot \epsilon(i_0, h_{i_0}) = -\delta'_0 \cdot \epsilon(i'_0, 1) = -\delta'_1,
$$
and the argument continues by induction. In particular,
$$
\Gamma(T)^M(i_0, 1) = (i_M, \delta_M) \quad \text{and} \quad \Gamma(T)^M(i'_0, -1) = (i_M, -\delta_M).
$$

The incidence matrix of $T$ has coefficients in $\{0,1\}$, and since $1 \neq h_{i_M}$, if $\rho_T(i_M, 1) = (k,l)$ then $\rho_T(i_M, h_{i_M}) = (k', l')$ with $k \neq k'$.

Consider the case when $\theta(i_M, \delta_M) = (i_M, 1)$ and $\theta(i_M, \delta'_M) = (i_M, h_{i_M})$. Applying the formula for $\Gamma(T)$, we get:
$$
\Gamma(T)^{M+1}(i_0, 1) = \Gamma(T)(i_M, 1) = (\bp_1\circ \rho(i_M, 1), \delta_M \cdot \epsilon(i_M, 1)) = (k, \delta_{M+1}),
$$
and
$$
\Gamma(T)^{M+1}(i_0, -1) = \Gamma(T)(i_M, -1) = (\bp_1 \circ \rho(i_M, h_{i_M}), -\delta'_M \cdot \epsilon(i_M, h_{i_M})) = (k', \delta'_{M+1}).
$$
Therefore, $i_{M+1} = k \neq k' = i'_{M+1}$.

The case $\theta(i_M, \delta_M) = (i_M, h_{i_M})$ and $\theta(i_M, \delta'_M) = (i_M, 1)$ is treated similarly. This proves the lemma.
\end{proof}

The proposition \ref{Prop: s code Inyective} follows directly from the previous lemma. The analogous result for negative codes associated with $u$-boundary labels is proven using a fully symmetric argument involving the $u$-generating function.

\end{proof}

If $\Gamma(T)(i,\delta) = (i_1, \delta_1)$, then applying the shift to this code produces another $s$-boundary positive code, explicitly given by
$$
\sigma(\underline{I}^+(i,\delta)) = \underline{I}^+(i_1, \delta_1),
$$
 Since there are exactly $2n$ different $s$-boundary positive codes, there exist natural numbers $k_1 < k_2$, with $k_1, k_2 \leq 2n$, such that
$$
\sigma^{k_1}(\underline{I}^+(i,\delta)) = \sigma^{k_2}(\underline{I}^+(i,\delta)),
$$
which shows that the code $\underline{I}^+(i,\delta)$ is pre-periodic. This leads to the following corollary:

\begin{coro}\label{Coro: preperiodic finite s,u boundary codes}
	There are exactly $2n$ distinct $s$-boundary positive codes and $2n$ distinct $u$-boundary negative codes. Moreover, every $s$-boundary positive code and every $u$-boundary negative code is pre-periodic under the action of the shift $\sigma$.
	
	In addition, for each $s$-boundary positive code $\underline{I}^+(i,\delta)\in \underline{\cS}^+(T)$, there exists some $k \leq 2n$ such that $\sigma^k(\underline{I}^+(i,\delta))$ is periodic. Similarly, for every $u$-boundary negative code $\underline{J}^-(i,\delta)\in \underline{\cU}^-(T)$, there exists some $k \leq 2n$ such that $\sigma^{-k}(\underline{J}^-(i,\delta))$ is periodic.
\end{coro}

Now we are ready to define a family of admissible codes that project onto the stable and unstable leaves of periodic boundary points of $\mathcal{R}$.

\begin{defi}\label{Defi: s,u-boundary codes}
	The set of $s$-\emph{boundary codes} of $T$ is defined as
	\begin{equation}
		\underline{\cS}(T) := \{\bw \in \Sigma_{A(T)} : \bw_+ \in \underline{\cS}^{+}(T)\}.
	\end{equation}
	
	Similarly, the set of $u$-\emph{boundary codes} of $T$ is defined as
	\begin{equation}
		\underline{\cU}(T) := \{\bw \in \Sigma_{A(T)} :\bw_- \in \underline{\cU}^{-}(T)\}.
	\end{equation}
\end{defi}

\begin{prop}\label{Prop: positive codes are boundary}
	Let $T$ be a symbolically presentable geometric type, and let $(f,\cR)$ be a pair that realizes $T$. Suppose that $(i,\delta) \in \cS(T)$ is an $s$-boundary label of $T$, and let $\bw \in \underline{\cS}(T)$ be a code such that $\underline{I}^+(i,\delta) = \bw_+$. Then, $\pi_{(f,\cR)}(\bw) \in \partial^s_{\delta} R_i$.
	
	Similarly, if $(i,\delta) \in \cU(T)$ is a $u$-boundary label of $T$, and $\bw \in \underline{\cU}(T)$ is a code such that $\underline{J}^-(i,\delta) = \bw_-$, then $\pi_{(f,\cR)}(\bw) \in \partial^u_{\delta} R_i$.
\end{prop}

\begin{proof}
	Set $(i,\delta) = (i_0,\delta_0)$ to make the notation consistent, and let $\underline{I}^+(i,\delta) = \{i_m\}_{m\in \NN}$. For every $s \in \NN$, define the rectangles $\overset{o}{H_s} := \bigcap_{m=0}^s f^{-m}(\overset{o}{R_{i_m}})$. The limit of the closures of these rectangles, as $s \to \infty$, is a unique stable segment of $R_{i_0}$. The proof will be complete if $\partial^s_{\delta_0} R_{i_0} \subset H_s := \overline{\overset{o}{H_s}}$ for all $s \in \NN$. Thus, it suffices to show that for all $s \in \NN$,
	$$
	f^s(\partial^s_{\delta_0} R_{i_0}) \subset \partial_{\delta_s} R_{i_s}.
	$$
	We will prove it by induction on $s$:

	\emph{Base case}: When $s = 0$, this is immediate since $f^0(\partial^s_{\delta_0} R_{i_0}) = \partial^s_{\delta_0} R_{i_0}$.\\

	\emph{Inductive hypothesis}: Assume that $f^s(\partial_{\delta_0} R_{i_0}) \subset \partial_{\delta_s} R_{i_s}$.\\
	
	\emph{Inductive step}: We want to show that
	$$
	f^{s+1}(\partial_{\delta_0} R_{i_0}) \subset \partial_{\delta_{s+1}} R_{i_{s+1}},
	$$
Consider two cases:
	
	\begin{itemize}
		\item[1)] $\delta_s = 1$: In this case, $f^s(\partial^s_{\delta_0} R_{i_0}) \subset \partial_{+1}^s H^{i_s}_{h_{i_s}}$ and then:
		$$
		f^{s+1}(\partial^s_{\delta_0} R_{i_0}) \subset f( \partial_{+1}^s H^{i_s}_{h_{i_s}}) \subset \partial^s_{\delta'_{s+1}} R_{i'_{s+1}},
		$$
		where $\delta'_{s+1} = \delta_s \cdot \epsilon(i_s, h_{i_s})$.
		
		\item $\delta_s = -1$: In this case, $f^s(\partial^s_{\delta_0} R_{i_0}) \subset \partial_{-1}^s H^{i_s}_1$, then
		$$
		f^{s+1}(\partial^s_{\delta_0} R_{i_0}) \subset \partial^s_{\delta'_{s+1}} R_{i'_{s+1}},
		$$
		where $\delta'_{s+1} = \delta_s \cdot \epsilon(i_s, 1)$.
	\end{itemize}
	
	In both cases, we conclude:
	$$
	f^{s+1}(\partial^s_{\delta_0} R_{i_0}) \subset \partial^s_{\delta'_{s+1}} R_{i'_{s+1}},
	$$
	and moreover,
	$$
	(i'_{s+1}, \delta'_{s+1}) = \left(\bp_1 \circ \rho(i_s, \theta_T(i_s, \delta_s)), \delta_s \cdot \epsilon(i_s, \theta_T(i_s, \delta_s))\right) = \Gamma(T)^{s+1}(i_0, \delta_0) = (i_{s+1}, \delta_{s+1}).
	$$
	Therefore, $f^{s+1}(\partial_{\delta_0} R_{i_0}) \subset \partial_{\delta_{s+1}} R_{i_{s+1}}$, as claimed, and the result follows.
	
	The unstable case is completely analogous.
\end{proof}

\begin{prop}\label{Prop: boundary points have boundary codes}
	Let $T$ be a symbolically presentable geometric type, and let $(f,\cR)$ be a pair realizing $T$. If a code $\bw \in \Sigma_{A(T)}$ projects under $\pi_{(f,\cR)}$ to the stable boundary of the Markov partition, i.e., $\pi_{(f,\cR)}(\bw) \in \partial^s \cR$, then $\bw \in \underline{\cS}(T)$. \\
	Similarly, if $\pi_{(f,\cR)}(\bw) \in \partial^u \cR$, then $\bw \in \underline{\cU}(T)$.
\end{prop}

\begin{proof}
	Since $A(T)$ has entries in $\{0,1\}$, the sequence $\bw_+$and the geometric type,  determines  for all $m \in \NN$, an horizontal label $(w_m, j_m) \in \mathcal{H}(T)$ such that $\rho(w_m, j_m) = (w_{m+1},l_{l_{m+1}})\in \cV(T)$ and then let $\epsilon(w_m, j_m) = \delta_{m+1} \in \{1, -1\}$. \\
	Let $x:=\pi_f(\bw)$, by $f$-invariance of the $s$-boundary of $\cR$, $f^m(x) \in \partial^s \cR$ and there are only two possibilities (except when $h_{w_m} = 1$, in which case both cases coincide):
	$$
	w_{m+1} = \bp_1\circ \rho(w_m, 1) \quad \text{or} \quad w_{m+1} =\bp_1 \circ \rho (w_m, h_{w_m}).
	$$
	
	This allows us to define $\delta_m \in \{-1, +1\}$ as the unique number such that:
	\begin{equation}\label{Equa: determine epsilon m}
		w_{m+1} = \bp_1\circ \rho(w_m, \theta_T(w_m, \delta_m)).
	\end{equation}
	Moreover, $\delta_m$ determines $\delta_{m+1}$ via the formula:
	$$
	\delta_{m+1} = \delta_m \cdot \epsilon(w_m, \theta(w_m, \delta_m)).
	$$
	
	In summary:
	\begin{equation}\label{Equa: w determine by gamma}
		\Gamma(T)(w_m, \delta_m) = (w_{m+1}, \delta_{m+1}),
	\end{equation}
	which follows the rule dictated by the $s$-generating function. Thus, if we know $\delta_M$ for some $M \in \NN$, we can reconstruct $\sigma^M(\bw)_+ = \underline{I}^+(w_M, \delta_M)$, and it remains only to determine at least one such $\delta_M$.
	
	An adaptation of Lemma \ref{Lemm: positive h-i} yields the existence of a minimal $M \in \NN$ such that $h_{w_M} > 1$. Then, equation \eqref{Equa: determine epsilon m} uniquely determines $\delta_M$. To recover $\delta_0$, we proceed backwards. Since $h_{w_m} = 1$ for all $m < M$, we have:
	$$
	\delta_M = \delta_{M-1} \cdot \epsilon(w_{M-1}, 1),
	$$
	so that $\delta_{M-1} = \delta_M \cdot \epsilon(w_{M-1}, 1)$. Repeating this argument inductively, we obtain $\delta_0$, and then use equation \eqref{Equa: w determine by gamma} to obtain:
	$$
	\Gamma(T)^m(w_0, \delta_0) = (w_m, \delta_m).
	$$
	
	Hence, $\bw_+ = \underline{I}^+(w_0, \delta_0)$.
	
	The unstable case is entirely analogous.
\end{proof}

 It is important to note that the cardinality of each of the sets that we are introduced below is less than or equal to $2n$.

\begin{defi}\label{Defi: s,u-boundary periodic codes}
	The set of $s$-\emph{boundary periodic codes} of $T$ is defined as
	\begin{equation}
		\text{Per}(\underline{\cS}(T)) := \{\bw \in \underline{\cS}(T) : \bw \text{ is periodic} \}.
	\end{equation}
	Similarly, the set of $u$-\emph{boundary periodic codes} of $T$ is defined as
	\begin{equation}
		\text{Per}(\underline{\cU}(T)) := \{\bw \in \underline{\cU}(T) : \bw \text{ is periodic} \}.
	\end{equation}
\end{defi}

\subsection{Boundary leaf codes} Now we can describe the subsets of $\Sigma_{A(T)}$ that must project onto the stable and unstable leaves of the periodic boundary points of any realization $(f, \cR)$ of $T$ via the projection $\pi_{(f, \cR)}$.

\begin{defi}\label{Defi: stratification Sigma A}
	We define the set of $s$-\emph{boundary leaf codes} of $T$ as
	\begin{equation}
		\Sigma_{\cS(T)} := \{ \bw \in \Sigma_{A(T)} : \exists\, k \in \NN \text{ such that } \sigma_{A(T)}^k(\bw) \in \underline{\cS}(T) \}.
	\end{equation}
	We define the set of $u$-\emph{boundary leaf codes} of $T$ as
	\begin{equation}
		\Sigma_{\cU(T)} := \{ \bw \in \Sigma_{A(T)} : \exists\, k \in \NN \text{ such that } \sigma_{A(T)}^{-k}(\bw) \in \underline{\cU}(T) \}.
	\end{equation}
	Finally, the set of \emph{totally interior codes} of $T$ is defined by
	\begin{equation*}
		\Sigma_{\cI(T)} := \Sigma_{A(T)} \setminus \left( \Sigma_{\cS(T)} \cup \Sigma_{\cU(T)} \right).
	\end{equation*}
\end{defi}

The importance of this partition of $\Sigma_A$ lies in the fact that their projections are well defined in the seance of the following lemma.

\begin{lemm}\label{Lemm: Projection Sigma S,U,I}
	Let $T \in \cG\cT(\textbf{p-A})^{sp}$ be a symbolically presentable geometric type, and let $(f,\cR)$ be a realization of $T$. Then:
	\begin{itemize}
		\item A code $\bw \in \Sigma_{A(T)}$ belongs to $\Sigma_{\cS(T)}$ if and only if its projection $\pi_{(f,\cR)}(\bw)$ lies within the stable leaf of a $s$-boundary periodic point of $\cR$.
		
		\item A code $\bw \in \Sigma_{A(T)}$ belongs to $\Sigma_{\cU(T)}$ if and only if its projection $\pi_f(\bw)$ lies within the unstable leaf of a $u$-boundary periodic point of $\cR$.
		
		\item A code $\bw \in \Sigma_A$ belongs to $\Sigma_{\cI(T)}$ if and only if its projection $\pi_{(f,\cR)}(\bw)$ is contained in $\operatorname{Int}(f,\cR)$, i.e. is a totally interior point.
	\end{itemize}
\end{lemm}
\begin{proof}
	The $s$-boundary leaf codes satisfy Definition~\ref{Defi: s,u-leafs}, and according to Proposition~\ref{Prop: Projection foliations}, if $\bw \in \Sigma_{\cS(T)}$, then its projection lies on the same stable manifold than a $s$-boundary component of $\cR$. The stable boundary components of a Markov partition is contained into the stable manifolds of its $s$-boundary periodic points. Moreover, for each $\bw \in \Sigma_{\cS(T)}$, there exists $k = k(\bw) \in \mathbb{N}$ such that $\sigma^k(\bw)_+$ is a periodic positive code (Corollary~\ref{Coro: preperiodic finite s,u boundary codes}). This positive code corresponds to a periodic point on the boundary of $\cR$, within whose stable manifold $\pi_f(\bw)$ lies. This proves one direction of the first assertion in the lemma.
	
	Now suppose that $\bv$ is a periodic $s$-boundary code such that, $\pi_f(\bv) \in \partial^s \cR$,  by definition $\bv \in \underline{\cS}(T)$. If $\pi_f(\bw)$ lies on the same stable leaf than $\bv$, then there exists $k \in \NN$ such that $\pi_f(\sigma^k(\bw))$ lies on the same stable boundary component of $\cR$ than $\pi_f(\bv)$. Proposition~\ref{Prop: boundary points have boundary codes} then implies that $\sigma^k(\bw) \in \underline{\cS}(T)$, and hence $\bw \in \Sigma_{\cS(T)}$ by definition. This completes the proof of the first item. A similar argument applies to the unstable case.

	As shown in Lemma~\ref{Lemm: Characterization unique codes}, totally interior points are disjoint from the stable and unstable laminations generated by boundary periodic points. If $\bw \in \Sigma_{\cI(T)}$ and $\pi_f(\bw)$ lies on the stable or unstable leaf of an $s$- or $u$-boundary periodic point, then we would have $\bw \in \Sigma_{\cS(T)} \cup \Sigma_{\cU(T)}$, which is a contradiction. Therefore, $\pi_f(\bw) \in \operatorname{Int}(f,\cR)$. Conversely, if $\pi_f(\bw)$ does not lie on the stable or unstable leaf of any boundary periodic point, then $\bw \notin \Sigma_{\cS(T)} \cup \Sigma_{\cU(T)}$, and so $\bw \in \Sigma_{\cI(T)}$. This concludes the proof.
\end{proof}

We have obtained the decomposition
$$
\Sigma_{A(T)} = \Sigma_{\cI(T)} \cup \Sigma_{\cS(T)} \cup \Sigma_{\cU(T)},
$$
and have characterized the image of each of these sets under the projection $\pi_{(f,\cR)}$. In the following section, we use this decomposition to define relations on each of these subsets, and then extend these to an equivalence relation on the entire space $\Sigma_{A(T)}$.

\section{The equivalent relation}

We must construct an equivalence relation $\sim_s$ on $\Sigma_{\cS(T)}$ and state some of its properties. An analogous equivalence relation $\sim_u$ can be defined on $\Sigma_{\cU(T)}$, but we will simply extrapolate our argument to the unstable case while providing full details for the construction of $\sim_s$. As usual, given a pair $(f,\cR)$, the corresponding projection $\pi_{(f,\cR)}$ will be denoted simply by $\pi_f$ whenever there is no risk of confusion.

\subsection{The $s$-boundary equivalence relation
	\texorpdfstring{}{ (s-boundary equivalence relation)}}

Let $\bw \in \Sigma_{\cS(T)} \setminus \text{Per}(\sigma_A)$ be a non-periodic $s$-boundary leaf code of $T$. Since it is not periodic, there exists a non-positive integer $k := k(\bw) \in \ZZ_{-}$ with the following properties: 
\begin{itemize}
	\item $\sigma^k(\bw) \notin \underline{\cS(T)}$, i.e., $\pi_f(\sigma^k(\bw)) \notin \partial^s \cR$;
	\item but, $\sigma^{k+1}(\bw) \in \underline{\cS(T)}$, i.e., $\pi_f(\sigma^{k+1}(\bw)) \in \partial^s \cR$.
\end{itemize}

\begin{lemm}\label{Lemm: minumun k}
	The number $k := k(\bw)$ is the unique integer such that $f^{k}(\pi_f(\bw)) \in \cR \setminus \partial^s \cR$, and for all $k' > k(\bw)$, we have $f^{k'}(\pi_f(\bw)) \in \partial^s \cR \setminus \text{Per}(f)$.
\end{lemm}

\begin{proof}
	Since $f^{k+1}(\pi_f(\bw)) \in \partial^s \cR$ and $\partial^s \cR$ is $f$-invariant, it follows that $f^{k+m}(\pi_f(\bw)) \in \partial^s \cR$ for all $m \geq 1$. Therefore, $k$ is the maximum integer such that $f^k(\pi_f(\bw)) \in \cR \setminus \partial^s \cR$, making it unique.
	
	Finally, since $\bw$ is not periodic, Lemma~\ref{Lemm: Periodic to periodic} implies that its projection cannot be a periodic point of the homeomorphism $f$.
\end{proof}

\begin{lemm}\label{Lemm: Diferen horizonl sub}
Let $\bw \in \Sigma_{\cS(T)} \setminus \text{Per}(\sigma_A)$ be a non-periodic code, $k=k(\bw)\in \ZZ_{-}$ as in Lemma \ref{Lemm: minumun k} and  $x := \pi_f(\bw)$ it projection. Then there are  indices $i \in \{1, \dots, n\}$ and $j \in \{1, \dots, h_i - 1\}$ such that:
\begin{itemize}
	\item  $f^k(x) \in R_i \setminus \partial^s R_i$;
	\item $f^k(x)$ lies in two adjacent horizontal subrectangles of $R_i$, i.e., $f^k(x) \in \partial^s_{+1} H^i_j$ and $x \in \partial^s_{-1} H^i_{j+1}$.
\end{itemize}
\end{lemm}

\begin{proof}
	By Lemma \ref{Lemm: minumun k}, there exists an $i \in \{1, \cdots, n\}$ such that $f^k(x) \in R_i \setminus \partial^s R_i$; therefore, $f^k(x)$ is contained in certain horizontal sub-rectangle $H^i_j$ of $R_i$. But since $f(\overset{o}{H^i_j}) \subset \overset{o}{R_{w_{k+1}}}$, the point $f^k(x)$ cannot be contained in $\overset{o}{H^i_j}$; it can only lie in $\partial H^i_j$. 
	
	Let $I_{f^k(x)}$ be the horizontal leaf of $H^i_j$ that contains $f^k(x)$. By Lemma \ref{Lemm: minumun k}, $f^{k+1}(x) \in \partial^s R_{w_{k+1}}$, and in fact $f(I_{f^k(x)}) \subset \partial^s R_{w_{k+1}}$. This implies that $I_{f^k(x)}$ must be equal to a stable boundary component of $H^i_j$, otherwise the interior of $H^i_j$ would intersect the boundary of $R_{w_{k+1}}$ (which is not possible).
	
	Then $I_{f^k(x)}$ is the stable boundary of $H^i_j$, but it cannot be an $s$-boundary component of $R_i$. Therefore, there exists a different horizontal sub-rectangle $H^i_{j'}$ such that $I_{f^k(x)}$ is the common $s$-boundary component of $H^i_j$ and $H^i_{j'}$. Clearly, $j = j' - 1$ or $j = j' + 1$, and this concludes our proof.
\end{proof}

\subsubsection{The mechanism of stable identification}

We describe the conditions under which $s$-boundary leaf codes must be equivalent.

\begin{figure}[h]
	\centering
	\includegraphics[width=0.5\textwidth]{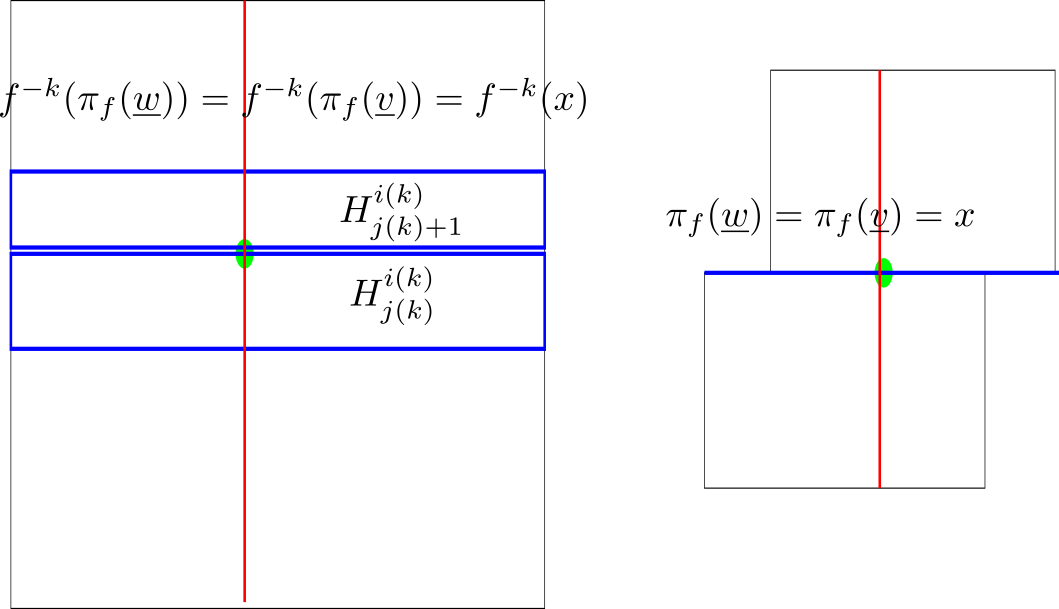}
	\caption{The stable identification mechanism}
	\label{Fig: stable identification}
\end{figure}

The mechanism is illustrated in Figure \ref{Fig: stable identification}. Let $\bw \in \Sigma_{\cS(T)} \setminus \text{Per}(\sigma_A)$, and let $k = k(\bw) \in \ZZ_{-}$ be as in Lemma \ref{Lemm: minumun k}, with $x := \pi_f(\bw)$ its projection. 

According to Lemma \ref{Lemm: Diferen horizonl sub}, we can assume that $f^k(x) \in \partial^s_{+1} H^i_j$ and $f^k(x) \in \partial^s_{-1} H^i_{j+1}$, and then:

\begin{itemize}
	\item $f^{k+1}(x) \in \partial^s_{\delta_0} R_{i_0}$, where:
	\begin{equation}
		i_0 = \bp_1 \circ \rho(i,j) \quad \text{and} \quad \delta_0 = \epsilon(i,j)
	\end{equation}
	
	\item $f^{k+1}(x) \in \partial^s_{\delta'_0} R_{i'_0}$, where:
	\begin{equation}
		i'_0 = \bp_1 \circ \rho(i,j+1) \quad \text{and} \quad \delta'_0 = -\epsilon(i,j+1)
	\end{equation}
\end{itemize}

This analysis suggests the following definition of the $\sim_s$ equivalence relation. Recall that $T = \big(n, \{v_i, h_i\}, \rho, \epsilon\big)$ is a symbolically presentable geometric type, and $(f, \cR)$ is any realization of $T$.

\begin{defi}\label{Defi: s realtion in Sigma S no per}
	Let $\bw, \bv \in \Sigma_{\cS(T)} \setminus \text{Per}(\sigma_A)$. We say they are $s$-related, and we write $\bw \sim_{s} \bv$, if they are equal or there exists $k \in \ZZ$ such that the following conditions hold:
	
	\begin{itemize}
		\item[i)] $\sigma^k(\bw), \sigma^k(\bv) \notin \underline{\cS(T)}$, but $\sigma^{k+1}(\bw), \sigma^{k+1}(\bv) \in \underline{\cS(T)}$.
		
		\item[ii)] $w_k = v_k$ and $h_{w_k} = h_{v_k} > 1$.
		
		\item[iii)] Let $i = w_k = v_k$. There exists $j \in \{1, \cdots, h_i - 1\}$ such that exactly one of the following options occurs:
		\begin{equation}\label{Equa: Sim s obtion 1}
			\bp_1 \circ \rho(i,j) = w_{k+1} \quad \text{and} \quad \bp_1 \circ \rho(i,j+1) = v_{k+1},
		\end{equation}
		or
		\begin{equation}\label{Equa: Sim s obtion 2}
			\bp_1 \circ \rho(i,j) = v_{k+1} \quad \text{and} \quad \bp_1 \circ \rho(i,j+1) = w_{k+1}.
		\end{equation}
		
		\item[iv)] Suppose the positive part of $\sigma^{k+1}(\bw)$ is equal to the $s$-boundary code $\underline{I}^{+}(w_{k+1}, \delta_w)$ and the positive part of $\sigma^{k+1}(\bv)$ is equal to the $s$-boundary code $\underline{I}^{+}(v_{k+1}, \delta_v)$. Then:
		
		If equation \eqref{Equa: Sim s obtion 1} holds, then
		\begin{equation}\label{Equa: sim s epsilon 1}
			\delta_w = \epsilon(w_k, j) \quad \text{and} \quad \delta_v = -\epsilon(v_k, j+1),
		\end{equation}
		
		but if equation \eqref{Equa: Sim s obtion 2} holds, then
		\begin{equation}\label{Equa: sim s epsilon 2}
			\delta_v = \epsilon(v_k, j) \quad \text{and} \quad \delta_w = -\epsilon(w_k, j+1).
		\end{equation}
		
		\item[v)] The negative codes $\sigma^{k}(\bw)_{-}$ and $\sigma^{k}(\bv)_{-}$ are equal.
	\end{itemize}
\end{defi}

\begin{prop}\label{Prop: sim s equiv in Sigma non per}
	The relation $\sim_{s}$ is an equivalence relation on $\Sigma_{\cS(T)} \setminus \text{Per}(\sigma_{A(T)})$.
\end{prop}

\begin{proof}
Reflexivity and symmetry are straightforward from the definition, so we focus on transitivity.

Assume $\bw \sim_{s} \bv$ and $\bv \sim_s \bu$. The number $k \in \ZZ$ in item (i) is unique, as it is given by:
$$
k(\bw) = \min\{ z \in \ZZ : \sigma^z(\bw) \in \underline{\cS(T)} \} - 1,
$$
therefore $k := k(\bw) = k(\bv) = k(\bu)$ is the same for all three codes. Since $w_k = v_k$ and $v_k = u_k$, it follows that $w_k = u_k$, and $h_{w_k}, h_{v_k}, h_{u_k} > 1$. Thus, we set $i = w_k = v_k = u_k$, and without loss of generality, there exists $j \in \{1, \cdots, h_i - 1\}$ such that:
$$
\bp_1 \circ \rho(i, j+1) = w_{k+1} \quad \text{and} \quad \bp_1 \circ \rho(i, j) = v_{k+1}.
$$

Moreover, since $\sigma^{k+1}(\bw)_+ = \underline{I}^+(w_{k+1}, \delta_w)$ and $\sigma^{k+1}(\bv)_+ = \underline{I}^+(v_{k+1}, \delta_v)$, we have:
$$
\delta_w = -\epsilon(i, j+1) \quad \text{and} \quad \delta_v = \epsilon(i, j).
$$

Because $\bv \sim_s \bu$, there exists a unique $j' \in \{1, \cdots, h_i - 1\}$ such that:
$$
\bp_1 \circ \rho(i, j') = u_{k+1}.
$$
However, the relation $\bv \sim_s \bu$ implies that $j' = j + 1$ or $j' = j - 1$. If we prove that $j' = j + 1$, then necessarily $\bu = \bw$, and the proof is complete. Therefore, let us analyze the case $j' = j - 1$.

Suppose $\sigma^{k+1}(\bu)_+ = \underline{I}^+(u_{k+1}, \delta_u)$. Since $j = (j - 1) + 1$, we are in the situation of Equation~\ref{Equa: Sim s obtion 1}, and applying Equation~\ref{Equa: sim s epsilon 1} we obtain:
$$
\delta_v = -\epsilon(i, j) \quad \text{and} \quad \delta_u = \epsilon(i, j - 1).
$$
Thus,
$$
\epsilon(i, j) = -\epsilon(i, j),
$$
which is a contradiction.

Hence, $j' = j + 1$, and the positive part of $\sigma^{k+1}(\bw)$ coincides with the positive part of $\sigma^{k+1}(\bu)$.

Item (v) implies that $\sigma^k(\bw)_- = \sigma^k(\bv)_-$ and $\sigma^k(\bv)_- = \sigma^k(\bu)_-$. Therefore, the negative part of $\sigma^k(\bw)$ equals the negative part of $\sigma^k(\bu)$, which implies $\bw \sim_s \bu$.
\end{proof}

It remains to extend the relation $\sim_s$ to the $s$-boundary periodic codes, that is, to the set $\text{Per}(\underline{\cS(T)}) := \text{Per}(\sigma_{A(T)}) \cap \Sigma_{\underline{\cS(T)}}$.

\begin{defi}\label{Defi: sim s in per}
	Let $\alpha, \beta \in \text{Per}(\underline{\cS(T)})$ be $s$-boundary periodic codes. We say they are $s$-related, written $\alpha \sim_s \beta$, if and only if they are equal or there exist $\bw, \bv \in \Sigma_{\cS(T)} \setminus \text{Per}(\sigma)$ such that:
	\begin{itemize}
			\item There exists $k \in \ZZ$ such that $\sigma^k(\bw)_+ = \alpha_+$ and $\sigma^k(\bv)_+ = \beta_+$.
			\item $\bw \sim_s \bv$,
	\end{itemize}
\end{defi}

\begin{prop}\label{Prop: sim s equiv in Sigma S }
	The relation $\sim_s$, as defined for non-periodic codes in Definition~\ref{Defi: s realtion in Sigma S no per} and extended to periodic codes in Definition~\ref{Defi: sim s in per}, is an equivalence relation on $\Sigma_{\cS(T)}$.
\end{prop}

\begin{proof}
	We have already established that $\sim_s$ is an equivalence relation in the non-periodic case. It remains to consider the periodic setting. Reflexivity holds trivially, while symmetry and transitivity are inherited from the relation on $\Sigma_{\cS(T)} \setminus \text{Per}(\sigma)$ to the periodic codes, by virtue of the first item in Definition~\ref{Defi: sim s in per}.
\end{proof}

\begin{prop}\label{Prop: s-relaten implies same projection}
	Let $T$ be a geometrically presentable geometric type, and let $(f,\cR)$ be any realization of $T$. Let $\bw, \bv \in \Sigma_{\cS(T)}$ be two $s$-boundary leaf codes. If $\bw \sim_{s} \bv$, then
	$$
	\pi_{(f,\cR)}(\bw) \;=\; \pi_{(f,\cR)}(\bv).
	$$
\end{prop}

\begin{proof}
	Assume $\bw$ and $\bv$ are $s$-related. For simplicity, take the integer $k = 0$ and let $w_0 = v_0 = i$. Set
	$$
	x_w := \pi_f(\bw), 
	\quad \text{and} \quad
	x_v := \pi_f(\bv).
	$$
	Since their negative codes agree, $x_w$ and $x_v$ lie on the same unstable segment of $R_i$. We now show that they lie on the same horizontal segment of $R_i$.
	
	By Item~(iii) of Definition~\ref{Defi: s realtion in Sigma S no per}, and by Equation~\ref{Equa: Sim s obtion 2}, we may assume, without loss of generality, that
	$$
	x_w \in H^i_{j+1} \quad \text{and} \quad x_v \in H^i_j.
	$$
	Since $f(x_w)$ lies on the stable boundary of $\cR$, we have $x_w \in \partial^s H^i_{j+1}$ and $x_v \in \partial^s H^i_j$. We must show that they lie on the shared $s$-boundary component of these horizontal subrectangles. Assume:
	$$
	\bw_+ = \underline{I}^+(w_1, \delta_w) \quad \text{and} \quad \bv_+ = \underline{I}^+(v_1, \delta_v).
	$$
	
	If $x_w \in \partial^s_{+1} H^i_{j+1}$, then $f(x_w)$ would lie on the boundary corresponding to $\delta_w = \epsilon(i, j+1)$, contradicting the relation $\delta_w = -\epsilon(i, j+1)$ from Equation~\ref{Equa: sim s epsilon 2}. Thus, $\delta_w = -1$, and $x_w$ lies on the lower boundary of $H^i_{j+1}$.
	
	Similarly, if $\delta_v = -1$, then $f(x_v)$ would lie on the boundary corresponding to $\epsilon_v = -\epsilon(i, j)$, which contradicts Equation~\ref{Equa: sim s epsilon 2}, since $\epsilon_v = \epsilon(i, j)$. Therefore, $\delta_v = +1$, and $x_v$ lies on the upper boundary of $H^i_j$.
	
	But the lower boundary of $H^i_{j+1}$ coincides with the upper boundary of $H^i_j$, so $x_w$ and $x_v$ lie on the same stable segment of $R_i$. Hence, their projections coincide:
	$$
	\pi_f(\bw) = x_w \;=\; x_v = \pi_f(\bv),
	$$
	as required.
\end{proof}

\subsection{The $u$-boundary equivalence relation
	\texorpdfstring{}{ (u-boundary equivalence relation)}}

In the same spirit as $\sim_s$, there is an equivalence relation $\sim_u$ for the elements in $\Sigma_{\cU(T)}$. We introduce it here for completeness, but we omit the proof that it is an equivalence relation, as it is entirely analogous to Proposition~\ref{Prop: sim s equiv in Sigma S }.

\begin{defi}\label{Defi: u relation in Sigma U no per}
	Let $\bw, \bv \in \Sigma_{\cU(T)} \setminus \text{Per}(\sigma)$. We say they are $u$-related, and we write $\bw \sim_u \bv$, if and only if they are equal or there exists $z \in \ZZ$ such that:
	\begin{itemize}
		\item[i)] $\sigma^z(\bw), \sigma^z(\bv) \notin \underline{\cU(T)}$, but $\sigma^{z-1}(\bw), \sigma^{z-1}(\bv) \in \underline{\cU(T)}$.
		
		\item[ii)] $w_z = v_z := k \in \{1, \cdots, n\}$ and $v_k > 1$.
		
		\item[iii)] There exists $l \in \{1, \cdots, v_k - 1\}$ such that exactly one of the following holds:
		\begin{equation}\label{Equa: Sim u 1 option}
			\bp_1 \circ \rho^{-1}(k, l) = w_{z-1} \quad \text{and} \quad \bp_1 \circ \rho^{-1}(k, l+1) = v_{z-1},
		\end{equation}
		or
		\begin{equation}\label{Equa: Sim u 2 option}
			\bp_1 \circ \rho^{-1}(k, l) = v_{z-1} \quad \text{and} \quad \bp_1 \circ \rho^{-1}(k, l+1) = w_{z-1}.
		\end{equation}
		
		\item[iv)] Suppose the negative code of $\sigma^{z-1}(\bw)$ is equal to the negative $u$-boundary code $\underline{J}^-(w_{z-1}, \delta_w)$, and the negative code of $\sigma^{z-1}(\bv)$ is equal to the negative $u$-boundary code $\underline{J}^-(v_{z-1}, \delta_v)$. Then:
		
		If Equation~\eqref{Equa: Sim u 1 option} holds:
		\begin{equation}\label{Equa: Sim u 1 epsilon}
			\delta_w = \epsilon \circ \rho^{-1}(k, l) \quad \text{and} \quad \delta_v = -\epsilon \circ \rho^{-1}(k, l+1),
		\end{equation}
		and if Equation~\eqref{Equa: Sim u 2 option} holds:
		\begin{equation}\label{Equa: Sim u 2 epsilon}
			\delta_v = \epsilon \circ \rho^{-1}(k, l) \quad \text{and} \quad \delta_w = -\epsilon \circ \rho^{-1}(k, l+1).
		\end{equation}
		
		\item[v)] The positive codes of $\sigma^z(\bw)$ and $\sigma^z(\bv)$ are equal.
	\end{itemize}
\end{defi}

\begin{defi}\label{Defi: sim u in per}
	Let $\alpha, \beta \in \text{Per}(\Sigma_{\cU(T)})$ be periodic $u$-boundary codes. They are $u$-related, and we write $\alpha \sim_u \beta$, if and only if there exist $\bw, \bv \in \Sigma_{\cU(T)} \setminus \text{Per}(\sigma)$ such that:
	\begin{itemize}
			\item There exists $p \in \ZZ$ such that $\sigma^p(\bw)_-=\alpha_-$, and 
			 $\sigma^p(\bv)_- = \beta_-$.
			\item $\bw \sim_u \bv$,
	\end{itemize}
\end{defi}

Using the same techniques as in Proposition~\ref{Prop: sim s equiv in Sigma S }, we can prove the following:

\begin{prop}\label{Prop: sim u equiv in Sigma U }
	The relation $\sim_u$ in $\Sigma_{\cU(T)}$ is an equivalence relation.
\end{prop}

Similarly to Proposition \ref{Prop: s-relaten implies same projection} we have following result.

\begin{prop}\label{Prop: u-relaten implies same projection}
	Let $\bw, \bv \in \Sigma_{\cU(T)}$ be two $u$-boundary leaf codes. If $\bw \sim_u \bv$, then $\pi_f(\bw) = \pi_f(\bv)$.
\end{prop}

\subsection{The interior equivalence relation}

As proved in Lemma~\ref{Lemm: Projection Sigma S,U,I}, totally interior codes are the only ones that project to totally interior points of any realization $(f,\cR)$ of $T$, and we use that property to introduce the following definition.

\begin{defi}\label{Defi: I sim relation}
	Let $\bw, \bv \in \Sigma_{\cI(T)}$ be two totally interior codes. They are $I$-related, and we write $\bw \sim_I \bv$ if and only if $\bw = \bv$.
\end{defi}

The following is a direct consequence of Proposition~\ref{Prop: Characterization injectivity of pi_f}, where we characterized totally interior points as having a unique code projecting to them.

\begin{prop}\label{Prop: totally interior points projection sim I}
	The relation $\sim_I$ is an equivalence relation on $\Sigma_{\cI nt(T)}$, and two codes $\bw, \bv \in \Sigma_{\cI nt(T)}$ are $\sim_I$-related if and only if $\pi_f(\bw) = \pi_f(\bv)$, i.e., if and only if they project to the same point.
\end{prop}

\subsection{The equivalence relation $\sim_T$ on $\Sigma_{A(T)}$\texorpdfstring{ }{ -sim_T on Sigma_{A(T)}}}

 Finally, we are ready to define the relation $\sim_T$ on $\Sigma_A$ as claimed in Proposition~\ref{Prop: The relation determines projections}. It consists essentially of the equivalence relation generated by $\sim_s$, $\sim_u$, and $\sim_I$.

\begin{defi}\label{Defi: Sim-T equivalent relation}
	Let $\bw, \bv \in \Sigma_{A(T)}$, they are $T$-related, and write $\bw \sim_T \bv$, if and only if one of the following disjoint situations occurs:
	\begin{itemize}
		\item[i)] $\bw, \bv \in \Sigma_{\cI(T)}$ and $\bw \sim_I \bv$, i.e., they are equal.
		
		\item[ii)] $\bw, \bv \in \Sigma_{\cS(T)} \cup \Sigma_{\cU(T)}$ and there exists a finite sequence of codes $\{ \bx_i \}_{i=1}^m \subset \Sigma_{\cS(T)} \cup \Sigma_{\cU(T)}$ such that either
		\begin{equation}
			\bw \sim_s \bx_1 \sim_u \bx_2 \sim_s \cdots \sim_s \bx_m \sim_u \bv,
		\end{equation}
		or
		\begin{equation}
			\bw \sim_u \bx_1 \sim_s \bx_2 \sim_u \cdots \sim_u \bx_m \sim_s \bv.
		\end{equation}
	\end{itemize}
\end{defi}

\begin{prop}
	The relation $\sim_T$ is an equivalence relation on $\Sigma_A$.
\end{prop}

\begin{proof}
	If $\bw \in \Sigma_{\cI(T)}$, then $\sim_T$ coincides with the equality relation, which is clearly reflexive, symmetric, and transitive.
	
	Now assume $\bw, \bv \in \Sigma_{\cS(T)} \cup \Sigma_{\cU(T)}$. In this case, $\sim_T$ is defined via s finite sequence of relations involving $\sim_s$ and $\sim_u$ and the properties of reflexivity and symmetry of $\sim_T$  follow directly from the corresponding properties of $\sim_s$ and $\sim_u$.
	
	To verify transitivity, suppose $\bw \sim_T \bv$ and $\bv \sim_T \bu$. Then, by definition, there exist finite sequences connecting $\bw$ to $\bv$ and $\bv$ to $\bu$ through compositions of $\sim_s$ and $\sim_u$. Concatenating these sequences yields a finite chain from $\bw$ to $\bu$, showing that $\bw \sim_T \bu$. Hence, $\sim_T$ is transitive and a equivalence relation on $\Sigma_{A(T)}$.
\end{proof}

\begin{lemm}\label{lemma: simT related iterations related}
	If two codes $\bw, \bv \in \Sigma_A$ are $\sim_T$-related, then for all $n \in \ZZ$, 
	$$
	\sigma^n(\bw) \sim_T \sigma^n(\bv).
	$$
\end{lemm}

\begin{proof}
	If $\bw \sim_s \bv$, the integer $k \in \ZZ$ in Item~(i) of Definition~\ref{Defi: s realtion in Sigma S no per} is replaced by $k-1$ for $\sigma(\bw)$ and $\sigma(\bv)$, and all remaining conditions still hold then $	\sigma(\bw) \sim_s \sigma(\bv)$. The same applies to the relation $\sim_u$ using $z$ instead of $z-1$ in Definition~\ref{Defi: u relation in Sigma U no per}. By induction, this property extends to all $n \in \ZZ$.
	
	The relation $\sim_I$ is equality, so the claim is immediate in that case. Thus, the statement holds for all types of $\sim_T$-related codes.
\end{proof}

The final property we need to verify in order to complete the proof of Proposition~\ref{Prop: The relation determines projections} is that the relation $\sim_T$ precisely characterizes when two codes project to the same point, independently of the realization. This is established in the following result.

\begin{prop}\label{Prop: T related iff same projection}
	Let $T \in \cG\cT(\textbf{p-A})^{sp}$ be a symbolically presentable geometric type, and let $(f, \cR)$ be any realization of $T$. Let $\bw, \bv \in \Sigma_{A(T)}$ be any admissible codes. Then $\bw \sim_T \bv$ if and only if
$$
	\pi_{(f, \cR)}(\bw) = \pi_{(f, \cR)}(\bv).
$$
\end{prop}

\begin{proof}
If $\bw \sim_T \bv$, then we have two possibilities:

\begin{itemize}
	\item  $\bw, \bv \in \Sigma_{\cI(T)}$ and they are $\sim_I$-related. By Proposition~\ref{Prop: totally interior points projection sim I}, this occurs if and only if $\pi_f(\bw) = \pi_f(\bv)$ and this case is over.
	
	\item $\bw, \bv \in \Sigma_{\cS(T)} \cup \Sigma_{\cU(T)}$. Using Propositions~\ref{Prop: s-relaten implies same projection} and~\ref{Prop: u-relaten implies same projection} alternately, we deduce that
	$$
	\pi_f(\bw) = \pi_f(\bx_1) = \cdots = \pi_f(\bx_m) = \pi_f(\bv),
	$$
	completing one direction of the proposition.
\end{itemize}

Now suppose that $x = \pi_f(\bw) = \pi_f(\bv)$. We need to prove that $\bw$ and $\bv$ are $\sim_T$-related. By Lemma~\ref{Lemm: every code is sector code }, the only codes that project to the same point are the sector codes of such point. Therefore, $\bw$ and $\bv$ are sector codes of $x$, and the following lemma implies our proposition.

\begin{lemm}\label{Lemm: sector codes}
	Let $\{\be_i\}_{i=1}^{2k}$ be the sector codes of the point $x = \pi_f(\be_i)$. Then $\be_i \sim_T \be_j$ for all $i,j \in \{1, \ldots, 2k\}$.
\end{lemm}

\begin{proof}
	If $x$ is a totally interior point (Figure~\ref{Fig: Sim T items i y ii}   item $b)$ ), Corollary~\ref{Coro: interior periodic points unique code} implies that all sector codes of $x$ are equal, and then $\underline{e_j} \sim_T \underline{e_j}$. 
	
	The remaining situation is when  $x$ is in the stable or unstable lamination generated by $s,u$-boundary periodic points. Numbering the sectors of $x$ in cyclic order, we consider three situations depending on where $x$ is located:
	\begin{itemize}
		\item[i)] $x \in \cF^s(\mathrm{Per}^s(\cR))$ but $x \notin \cF^u(\mathrm{Per}^u(\cR))$ (Item $c)$ in Figure~\ref{Fig: Sim T items i y ii}).
		\item[ii)] $x \in \cF^u(\mathrm{Per}^u(\cR))$ but $x \notin \cF^s(\mathrm{Per}^s(\cR))$ (Item $a)$ in Figure~\ref{Fig: Sim T items i y ii}).
		\item[iii)] $x \in \cF^s(\mathrm{Per}^s(\cR)) \cap \cF^u(\mathrm{Per}^u(\cR))$ (Item $d)$ in Figure~\ref{Fig: Sim T items i y ii}).
	\end{itemize}
	
	In either case, we first consider that $x$ is not periodic as this implies no sector code of $x$ is periodic and the point $x$ has $4$ sectors because it is not a periodic point of $f$ and therefore not a singularity.
	
	\hfill
	
\textbf{Item} $i)$  According to Lemma~\ref{Lemm: minumun k}, there exists a unique integer $k$ such that $f^{k+1}(x) \in \partial^s \cR$ but $f^{k}(x) \notin \partial^s \cR$. Since $x \notin \cF^u(\mathrm{Per}^u(\cR))$, it follows that $f^z(x) \notin \partial^u \cR$ for all $z \in \mathbb{Z}$. In particular, $f^z(x) \in \overset{o}{\cR}$ for all $z \leq k$, and $f^k(x)$ has only four quadrants, like $x$. Then:

\begin{itemize}
	\item For all $z < k$, since all the quadrants of $f^{z}(x)$ are in the same rectangle, we have
	$$
	(\be_1)_{z} = (\be_2)_{z} = (\be_3)_{z} = (\be_4)_{z}.
	$$
	
	\item Since $f^{k+n}(x) \in \partial^s \cR$, its quadrants are arranged as in Item~$c)$ of Figure~\ref{Fig: Sim T items i y ii}, so
	$$
	(\be_1)_{k+n} = (\be_2)_{k+n} \quad \text{and} \quad (\be_3)_{k+n} = (\be_4)_{k+n}.
	$$
	for all $n\geq 1$ and therefore  $\be_1 \sim_s \be_4$ and $\be_2 \sim_s \be_3$.
	
	\item Since $f^z(x) \notin \partial^u \cR$ for all $z \in \mathbb{Z}$, the sector codes of $x$ satisfy $\be_2 = \be_1$ and $\be_3 = \be_4$ then they are $\sim_{s}$ and $\sim_{u}$ relate.
\end{itemize}

	In conclusion, $(\be_2) \sim_u \be_1$ and $\be_3 \sim_u \be_4$, and then:
	$$
\be_1 \sim_s \be_4 \sim_u \be_3 \sim_s \be_2 \sim_u \be_1.
	$$
	So $\be_i \sim_T \be_j$ for $i,j = 1,2,3,4$.
	
	\hfill
	
\textbf{Item} $ii)$ is similarly proved. In this case, for some iteration, $f^k(x)$ lies in a rectangle as in Item~$a)$ of Figure~\ref{Fig: Sim T items i y ii}, and the negative iterates of $f^k(x)$ maintain this configuration. Therefore, for all $n \in \mathbb{N}$, the sector codes satisfy
$$
(\be_1)_{k-n} = (\be_4)_{k-n} \quad \text{and} \quad (\be_2)_{k-n} = (\be_3)_{k-n},
$$
and since $f^{k+n}(x) \in \overset{o}{\cR}$, for all $n \in \mathbb{N}_+$, the terms $(\be_{\sigma})_{k+n}$ are equal for $\sigma = 1,2,3,4$. This implies that:
\begin{itemize}
	\item $\be_1 \sim_u \be_4$ and $\be_2 \sim_u \be_3$,
	\item $\be_1 \sim_s \be_2$ and $\be_3 \sim_s \be_4$.
\end{itemize}

Hence, $\be_i \sim_T \be_j$ for all $i,j = 1,2,3,4$.

	\begin{figure}[h]
		\centering
		\includegraphics[width=0.8\textwidth]{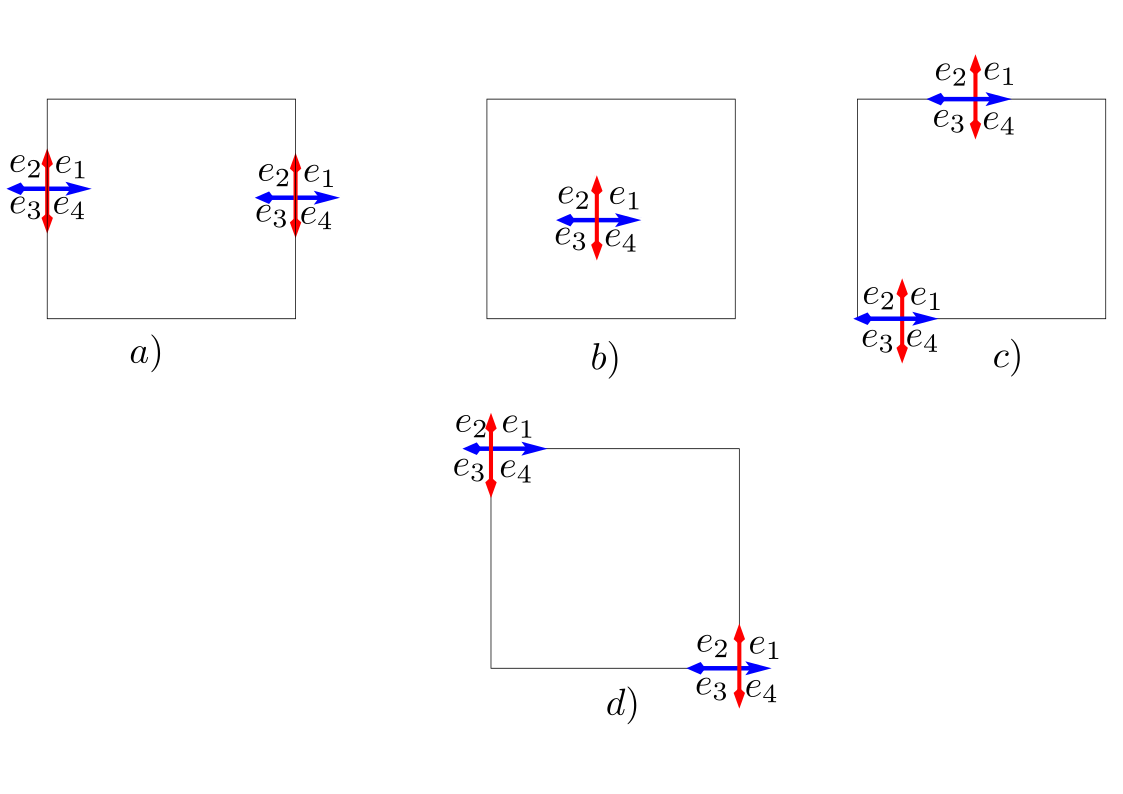}
		\caption{The sector codes that are identified.}
		\label{Fig: Sim T items i y ii}
	\end{figure}
	
	\textbf{Item} $iii)$ is the most technical situation. There are unique integer numbers $k(s), k(u) \in \mathbb{Z}$ defined as:
	$$
	k(s) := \max \{ k \in \mathbb{Z} : f^k(x) \notin \partial^s \cR \text{ but } f^{k+1}(x) \in \partial^s \cR \},
	$$
	and
	$$
	k(u) := \min \{ k \in \mathbb{Z} : f^k(x) \notin \partial^u \cR \text{ but } f^{k-1}(x) \in \partial^u \cR \}.
	$$
	
The $f$-invariance of $\partial^s \cR$ implies that for all $k > k(s)$, we have $f^k(x) \in \partial^s \cR$, while for all $k \leq k(s)$, $f^k(x) \notin \partial^s \cR$. Similarly, the $f^{-1}$-invariance of $\partial^u \cR$ implies that for all $k < k(u)$, we have $f^k(x) \in \partial^u \cR$, whereas for all $k \geq k(u)$, $f^k(x) \notin \partial^u \cR$. There are three possible cases to consider: $k(u) < k(s)$, $k(u) = k(s)$, and $k(u) > k(s)$. We divide the proof accordingly.

	\textbf{First case:} $k(u) < k(s)$. This implies that for $k\in\{k(u),\cdots,k(s)\}$ $f^{k}(x) \in \overset{o}{\cR}$ and then:
	\begin{itemize}
		\item The sector codes of $f^(x)$ take the same value at $k$.
		\item For all $k \geq k(s)$, the configuration of the sectors of $f^k(x)$ is like in Item $c)$  in Figure~\ref{Fig: Sim T items i y ii}.
		\item For all $k < k(u)$, the configuration of the sectors of $f^k(x)$ is like in Item $a)$  in Figure~\ref{Fig: Sim T items i y ii}.
	\end{itemize}
	
	In view of Lemma~\ref{lemma: simT related iterations related}, we deduce:
	$$
	\be_1 \sim_s \be_4 \sim_u \be_3 \sim_s \be_2 \sim_u \be_1.
	$$
	Hence, $\be_i \sim_T \be_j$ for all $i,j=1,\ldots,4$.
	
	\textbf{Second case:} $k(u) = k(s)$. Analogously, $f^{k(u)}(x) = f^{k(s)}(x) \in \overset{o}{\cR}$, and we repeat the previous analysis to deduce that $x$ has $4$ sector codes, all $\sim_T$ related. Figure~\ref{Fig: k(u) less than k(s)} illustrates this.
	
	\begin{figure}[h]
		\centering
		\includegraphics[width=0.7\textwidth]{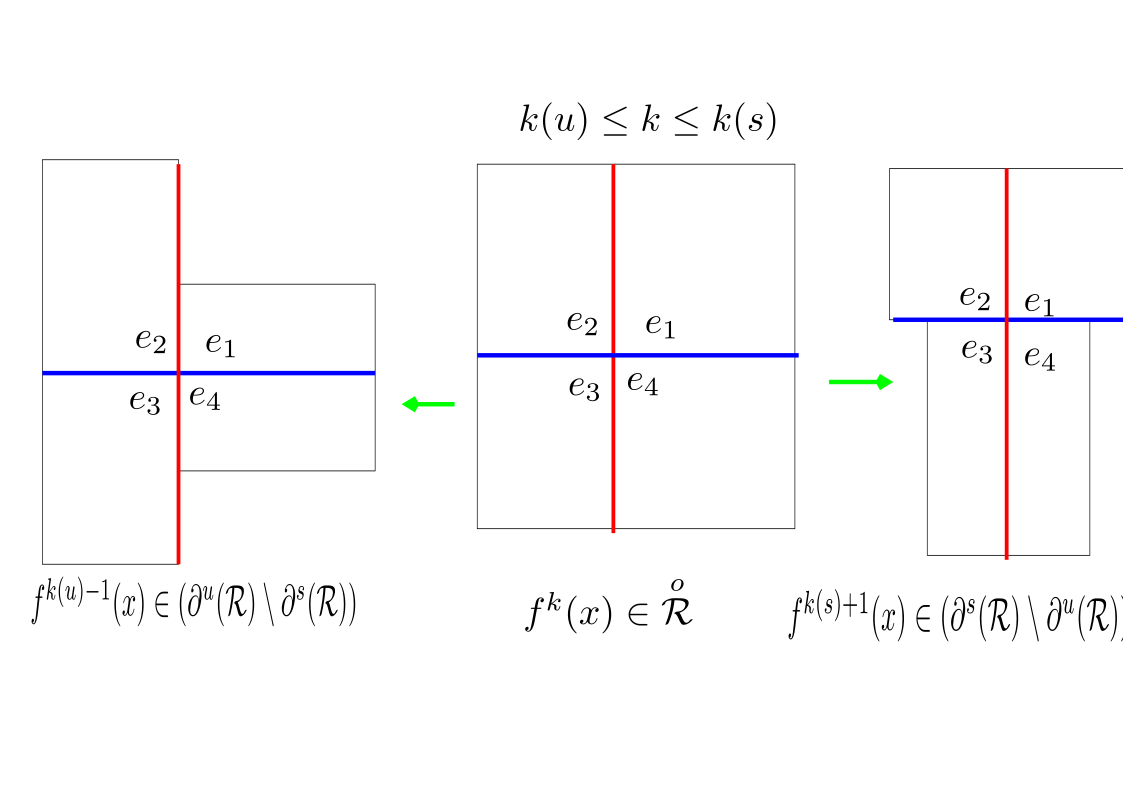}
		\caption{Situation $k(u) \leq k(s)$.}
		\label{Fig: k(u) less than k(s)}
	\end{figure}
	
	\textbf{Third case:} $k(s) < k(u)$. Here, $f^{k(s)}(x) \notin \partial^s \cR$ but $f^{k(s)}(x) \in \partial^u \cR$, and $f^{k(u)}(x) \in \partial^s \cR$ but $f^{k(u)}(x) \notin \partial^u \cR$. For all $k$ between $k(s)$ and $k(u)$, the point $f^k(x)$ is a corner point (see Item $d)$ in Figure~\ref{Fig: Sim T items i y ii}). This leads to the following consequences:
	\begin{itemize}
		\item For all $n \in \mathbb{N}$, the sectors $f^{k(s)-n}(e_1)$ and $f^{k(s)-n}(e_4)$ lie in the same rectangle of $\cR$, and similarly, the sectors $f^{k(s)}(e_2)$ and $f^{k(s)}(e_3)$ lie in the same rectangle (see Figure~\ref{Fig: k(s) less than k(s)}).
		\item For all $n \in \mathbb{N}$, the sectors $f^{k(u)+n}(e_1)$ and $f^{k(u)+n}(e_2)$ lie in the same rectangle, and similarly, the sectors $f^{k(u)+n}(e_3)$ and $f^{k(u)+n}(e_4)$ lie in the same rectangle.
	\end{itemize}
	
Then the negative parts of the sector codes $\sigma^{k(s)-n}(\be_1)$ and $\sigma^{k(s)-n}(\be_4)$ are equal, and similarly, the negative parts of the sector codes $\sigma^{k(u)+n}(\be_2)$ and $\sigma^{k(u)+n}(\be_3)$ are also equal.
Using Lemma~\ref{lemma: simT related iterations related} then can we deduce:
	$$
	\be_1 \sim_s \be_4 \quad \text{and} \quad \be_2 \sim_s \be_3.
	$$
	
Similarly the sector codes $\sigma^{k(u)}(\underline{e_1})$ and $\sigma^{k(u)}(\underline{e_2})$ have equal negative part and then $\be_1 \sim_u \be_2$. Applying this process to the other pair of sectors and usingLemma~\ref{lemma: simT related iterations related} again, we get:
	$$
	\be_1 \sim_u \be_2 \quad \text{and} \quad \be_3 \sim_u \be_4
	$$
	
	Putting all together:
	$$
	\be_1 \sim_s \be_4 \sim_u \be_3 \sim_s \be_2 \sim_u \be_1.
	$$
	This proves $\be_i \sim_T \be_j$ for all $i,j=1,2,3,4$. 
	\begin{figure}[h]
		\centering
		\includegraphics[width=0.6\textwidth]{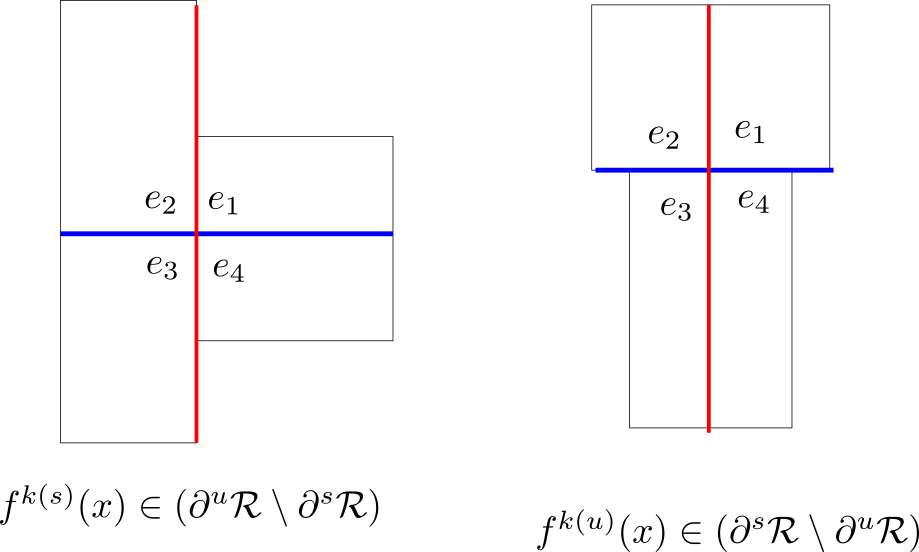}
		\caption{Situation $k(s) \leq k(u)$.}
		\label{Fig: k(s) less than k(s)}
	\end{figure}
	
It remains to address the periodic coding case. Let $x$ be a periodic point with $2k$ sectors, labeled cyclically with respect to the surface orientation. Two adjacent sectors $\be_i, \be_{i+1}$ are $\sim_s$-related if they share a small stable local separatrix; indeed, the stable separatrix lies in (at most) a unique pair of rectangles, and the boundary codes of such sides are $\sim_s$-related. This is precisely the mechanism of stable boundary identification illustrated in Figure~\ref{Fig: stable identification}, which motivated the definition of $\sim_s$. If $\be_i$ and $\be_{i+1}$ share the same local unstable separatrix, then they are $\sim_u$-related for the same reason. Thus, one can move from one sector of $x$ to another through a finite number of intermediate sectors, alternating between $\sim_s$ and $\sim_u$ relations. Hence, $\be_i \sim_T \be_j$ for all $i,j = 1, \ldots, 2k$. This concludes the proof.

\end{proof}

\end{proof}

\section{The geometric type is a total conjugacy invariant.}

We must to prove Theorem \ref{Theo: Total invariant}  in this  section and our first step is the following proposition.

\begin{prop}\label{Prop: cociente T}
	Let $T$ be a symbolically modelable geometric type, and let $(f, \cR)$ be a realization of $T$, where $f: S \rightarrow S$. Let $(\Sigma_{A(T)}, \sigma_{A(T)})$ be the subshift of finite type associated to $T$, and let $\pi_{(f,\cR)}: \Sigma_{A(T)} \rightarrow S$ be the projection induced by the realization.
	
	Then the quotient space $\Sigma_T = \Sigma_{A(T)} / \sim_T$ coincides with $\Sigma_f := \Sigma_{A(T)} / \sim_f$ and is homeomorphic to $S$. The subshift  $\sigma_{A(T)}$ descends to the quotient as a homeomorphism $\sigma_T: \Sigma_T \rightarrow \Sigma_T$, topologically conjugate to $f$ via the quotient homeomorphism:
	$$
	[\pi_{(f,\cR)}] : \Sigma_T  \rightarrow S.
	$$
\end{prop}

\begin{proof}
	As Proposition~\ref{Prop: The relation determines projections} indicates, the relation $\sim_f$, defined by $\bw \sim_f \bv$ if and only if $\pi_f(\bw) = \pi_f(\bw)$, have same equivalent classes than $\sim_T$, therefore, the quotient spaces $\Sigma_T$ and $\Sigma_f$ are the same.
	
	Proposition~\ref{Prop: quotien by f} implies that $S$ and $\Sigma_f$ are homeomorphic. Hence, $\Sigma_T$ is homeomorphic to $S$, which is a closed surface. Moreover, the shift map $\sigma_{A(T)}$ descends to the quotient under $\sim_f$ to a homeomorphism $[\sigma_{A(T)}]$, which is topologically conjugate to $f$ via the quotient map $[\pi_f]$. Therefore the shift $\sigma_{A(T)}$ also descends to the quotient under $\sim_T$, to obtain the homeomorphism $\sigma_T : \Sigma_T \rightarrow \Sigma_T$, which coincides with $[\sigma_{A(T)}]$. We conclude, as in Proposition~\ref{Prop: quotien by f}, that
	$$
	[\pi_f]^{-1} \circ f \circ [\pi_f] = \sigma_T.
	$$
\end{proof}

\subsection{The induced geometric Markov partition}

Let $\cR=\{R_i\}_{i=1}^n$ be a geometric Markov partition of a \textbf{p-A} homeomorphism $f:S \rightarrow S$, and let $g:S' \rightarrow S'$ be another \textbf{p-A} homeomorphism topologically conjugate to $f$ via an orientation-preserving homeomorphism $h:S \rightarrow S'$. Since $h(\partial^s \cR) = \partial^s h(\cR)$, we have:
$$
h \circ f \circ h^{-1}(h(\partial^s \cR)) = h \circ f(\partial^s \cR) \subset h(\partial^s \cR),
$$
so $h(\cR)$ has a $g$-invariant horizontal boundary, and clearly, it has a $g^{-1}$-invariant vertical boundary. Therefore, the family of rectangles $h(\cR) = \{h(R_i)\}_{i=1}^n$ is a Markov partition for $g$.

The function $h$ maps the foliations and singularities of $f$ to those of $g$, while preserving their joint orientation. Thus, if $r:[0,1] \times [0,1] \rightarrow R \subset S$ is a parametrized rectangle for $f$, then $h \circ r: [0,1] \times [0,1] \rightarrow h(R) \subset S'$ is a parametrized rectangle for $g$, and we endow $h(R)$ with the geometrization induced by $h \circ r$ and the vertical orientation of the unit square used to geometrize $R$.

\begin{defi}\label{Defi: induced geometric Markov partition}
	Let $f:S \rightarrow S$ and $g:S' \rightarrow S'$ be pseudo-Anosov homeomorphisms topologically conjugate via an orientation-preserving homeomorphism $h:S \rightarrow S'$. If $\cR = \{R_i\}_{i=1}^n$ is a geometric Markov partition of $f$, the \emph{geometric Markov partition} of $g$ \emph{induced} by $h$ is the Markov partition $h(\mathcal{R}) = \{h(R_i)\}_{i=1}^n$, where each rectangle $h(R_i)$ inherits its vertical direction as the direct image under $h$ of the vertical direction of $R_i$.
\end{defi}

\begin{lemm}\label{Lemm: Conjugated partitions same type.}
	Let $f$ and $g$ be two pseudo-Anosov homeomorphisms topologically conjugated via an orientation-preserving homeomorphism $h$. Let $\cR = \{R_i\}_{i=1}^n$ be a geometric Markov partition for $f$, and let $h(\cR) = \{h(R_i)\}_{i=1}^n$ be the geometric Markov partition of $g$ induced by $h$. In this situation, $H$ is a horizontal sub-rectangle of $(f, \mathcal{R})$ if and only if $h(H)$ is a horizontal sub-rectangle of $(g, h(\mathcal{R}))$. Similarly, $V$ is a vertical sub-rectangle of $(f, \mathcal{R})$ if and only if $h(V)$ is a vertical sub-rectangle of $(g, h(\mathcal{R}))$.
\end{lemm}

\begin{proof}
	Observe that $h(\overset{o}{R_i}) = \overset{o}{h(R_i)}$. Therefore, $C$ is a connected component of $\overset{o}{R_i} \cap f^{\pm}(\overset{o}{R_j})$ if and only if $h(C)$ is a connected component of $\overset{o}{h(R_i)} \cap g^{\pm}(\overset{o}{h(R_j)})$.
\end{proof}

\begin{prop}\label{Prop: Conjugated partitions same types}
	Let $f$ and $g$ be generalized pseudo-Anosov homeomorphisms conjugated through a homeomorphism $h$ that preserves the orientation, i.e., $g = h \circ f \circ h^{-1}$. Let $\cR=\{R_i\}_{i=1}^n$ be a geometric Markov partition for $f$, and let $h(\cR)=\{h(R_i)\}_{i=1}^n$ be the geometric Markov partition of $g$ induced by $h$. In this situation, the geometric types of $(g,h(\cR))$ and $(f,\cR)$ are the same.
\end{prop}

\begin{proof}
	Let $T(f, \cR) = (n, \{h_i, v_i\}, \rho, \epsilon)$ and $T(g, h(\cR)) = (n', \{h'_i, v'_i\}, \rho', \epsilon')$ be the geometric types of the corresponding geometric Markov partitions. A direct consequence of Lemma~\ref{Lemm: Conjugated partitions same type.} is that: 
	$$
	n = n', \quad h_i = h'_i, \quad \text{and} \quad v_i = v'_i.
	$$
	
	Let $\{\bH^i_j\}_{j=1}^{h_i}$ be the set of horizontal sub-rectangles of $h(R_i)$, labeled with respect to the induced vertical orientation in $h(R_i)$. Similarly, define $\{\bV^k_l\}_{l=1}^{v_k}$ as the set of vertical sub-rectangles of $h(R_k)$, labeled with respect to the induced horizontal orientation in $h(R_k)$. Since $h$ preserves both the vertical and horizontal labeling of the respective sub-rectangles, it is clear that:
	$$
	h(H^i_j) = \bH^i_j \quad \text{and} \quad h(V^k_l) = \bV^k_l.
	$$
	
	Using the conjugacy by $h$, if $f(H^i_j) = V^k_l$, then:
	$$
	g(\bH^i_j) = g(h(H^i_j)) = h \circ f \circ h^{-1}(h(H^i_j)) = h(f(H^i_j)) = h(V^k_l) = \bV^k_l.
	$$
	
	This implies that, in the geometric types, $\rho = \rho'$.
	
	Now suppose that $g(\bH^i_j) = \bV^k_l$. The homeomorphism $h$ preserves the vertical orientations between $R_i$ and $h(R_i)$, as well as between $R_k$ and $h(R_k)$. Therefore, $f$ sends the positive vertical orientation of $H^i_j$ with respect to $R_i$ to the positive vertical orientation of $V^k_l$ with respect to $R_k$ if and only if $g$ sends the positive vertical orientation of $\bH^i_j$ with respect to $h(R_i)$ to the positive vertical orientation of $\bV^k_l$ with respect to $h(R_k)$. It follows that $\epsilon(i,j) = \epsilon'(i,j)$. This concludes the proof.
\end{proof}

One direction of Theorem~\ref{Theo: Total invariant} is a consequence of the following corollary of Proposition~\ref{Prop: Conjugated partitions same types}.

\begin{lemm}\label{Lemm: conjugates then same partition}
	If $f$ and $g$ are topologically conjugate via an orientation-preserving homeomorphism, then they admit geometric Markov partitions of the same geometric type.
\end{lemm}

\subsection{The conjugacy preserve the orientation}

We have proved one direction of our main theorem now we must to proceed to prove the other direction we must to this toward a few lemmas.

\begin{prop}\label{Prop: Same tipe then conjugated}
	If a pair $f$ and $g$ of pseudo-Anosov homeomorphisms have geometric Markov partitions with the same geometric type then they are topologically conjugated.
\end{prop}

\begin{proof}
	Let $T := T(f, \cR_f) = T(g, \cR_g)$ be the geometric type of the Markov partitions of $f$ and $g$. If $T$ does not have a binary incidence matrix, then by Lemma~\ref{Prop: Refinamiento binario}, we can consider their binary refinements to obtain a symbolically presentable geometric type. Therefore, we assume from the beginning that $T$ is symbolically presentable.
	
	The quotient spaces of $\Sigma_{A(T)}$ by $\sim_f$ and $\sim_g$ are equal to $\Sigma_T$, as proved in Proposition~\ref{Prop: cociente T}; that is, $\Sigma_g = \Sigma_T = \Sigma_f$. Moreover, the quotient shift $\sigma_T$ is topologically conjugate to $f$ via $[\pi_f] : \Sigma_T \rightarrow S_f$, and to $g$ via $[\pi_g] : \Sigma_T \rightarrow S_g$. Therefore, $f$ is topologically conjugate to $g$ by the homeomorphism $h := [\pi_g] \circ [\pi_f]^{-1} : S_f \rightarrow S_g$, as claimed.
\end{proof}

Now must to prove our the homeomorphism $	[\pi_g] \circ [\pi_f]^{-1}$ preserve the orientation.

\begin{lemm}\label{Lemm: Image Hi is Hi}
	Let $H^i_j$ and $\bH^i_j$ be the respective horizontal sub-rectangles of $\cR_f$ and $\cR_g$. Then, if $h := [\pi_g] \circ [\pi_f]^{-1}$, we have:
	$$
	h(H^i_j) = \bH^i_j.
	$$
	Similarly, for the vertical sub-rectangles $V^k_l$ and $\bV^k_l$, we have:
	$$
	h(V^k_l) = \bV^k_l.
	$$
\end{lemm}

\begin{proof}
	Consider the set 
	$$
	R(i,j) = \{\bw \in \Sigma_{A(T)} : w_0 = i \text{ and } \rho(i,j) = (w_1, l_0) \}.
	$$
Since the incidence matrix of $T$ is binary, this set y satisfy that:  $\pi_f(R(i,j)) = H^i_j$ and $\pi_g(R(i,j)) = \bH^i_j$. 	Let $R(i,j)_T$ denote the equivalence classes in $\Sigma_T$ that contain an elements of $R(i,j)$. In this manner:
	$$
	[\pi_g] \circ [\pi_f]^{-1}(H^i_j) = [\pi_g](R(i,j)_T) = \bH^i_j.
	$$
	an our proof is over for horizontal sub-rectangles. The proof for vertical sub-rectangles is totally symmetric.
\end{proof}

We must to prove a more general statement that will be useful to prove that  $h := [\pi_g] \circ [\pi_f]^{-1}$ is orientation preserving.

\begin{lemm}\label{Lemm. refinamiento h>2}
	Let $T = \big(n, \{h_i, v_i\}, \rho, \epsilon \big)$ be a symbolically presentable geometric type, and let $(f, \cR_f)$ and $(g, \cR_g)$ be pairs that realize $T$. Let $h := [\pi_g] \circ [\pi_f]^{-1}$ be the conjugacy between the homeomorphisms.
	
	Consider the horizontal sub-rectangles of $R_i$ and $\bR_i$ given by the closures of the connected components of the intersections:
	$$
	f^{-n}(\overset{\circ}{R_j}) \cap \overset{\circ}{R_i} \quad \text{and} \quad g^{-n}(\overset{\circ}{\bR_j}) \cap \overset{\circ}{\bR_i}.
	$$
	Label them with respect to the vertical direction of $R_i$ and $\bR_i$, respectively, as $\{H_j\}_{j=1}^{J}$ and $\{\bH_j\}_{j=1}^{J'}$.
	
	Similarly, define the vertical sub-rectangles of $R_i$ and $\bR_i$ as the closures of the connected components of the intersections:
	$$
	f^{n}(\overset{\circ}{R_j}) \cap \overset{\circ}{R_i} \quad \text{and} \quad g^{n}(\overset{\circ}{\bR_j}) \cap \overset{\circ}{\bR_i}.
	$$
	Label them with respect to the horizontal direction of $R_i$ and $\bR_i$, respectively, as $\{V_k\}_{k=1}^{L}$ and $\{\bV_k\}_{k=1}^{L'}$.
	
	In this setting, we have $L = L' \geq 2$, $J= J' \geq 2$  and:
	$$
	h(H_j) = \bH_j \quad \text{and} \quad h(V_k) = \bV_k.
	$$
\end{lemm}

\begin{proof}
Since $A(T)$ is the incidence matrix of the Markov partition of a pseudo-Anosov homeomorphism, it is a mixing matrix. The coefficient $a^{(n)}_{ij}$ of $A^n$ (where $n$ is the number of rectangles in the partition) counts the number of times $f^n(\overset{\circ}{R_i})$ intersects $\overset{\circ}{R_j}$, and it is at least $1$. Also, since $\sum_{i=1}^n a_{ij} = L$, we conclude that $L \geq 2$ (unless the Markov partition consists of a single rectangle with a single horizontal sub-rectangle, which is not possible due to the uniform expansion along unstable leaves). Hence, $L \geq 2$.

Let $\bw$ and $\bv$ be elements in $\Sigma_{A(T)}$. We define the relation $\sim_n$ by:
\begin{itemize}
	\item For all $m \in \{0, \dots, n\}$, we have $w_m = v_m$, and
	\item $w_0 = v_0 = i$.
\end{itemize}

This defines an equivalence relation on the closed set $\{\bw \in \Sigma_{A(T)} : w_0 = i\}$. Let $[i, w_1, w_2, \dots, w_n]$ denote an equivalence class under this relation. Then $\pi_f([i, w_1, w_2, \dots, w_n])$ is the unique sub-rectangle $H_j$ given by:
$$
H_j = \bigcap_{j=0}^{n} f^j(\overset{\circ}{R_{w_j}}).
$$

Since the incidence matrix is binary, if $n$ is even, we can rewrite the expression as:
$$
H_j = \bigcap_{j \in 2\mathbb{Z},\ j \leq n} f^{-j} \left( \overset{\circ}{R_{w_j}} \cap f^{-1}(\overset{\circ}{R_{w_{j+1}}}) \right) = \bigcap_{j \in 2\mathbb{Z},\ j \leq n} f^{-j}(H^{w_j}_{j'}).
$$

Here, $H^{w_j}_{j'}$ is uniquely determined since the incidence matrix is fixed. As $\cR_g$ has the same geometric type as $\cR_f$, it follows that the corresponding $\bH_j$ must be written as:
$$
\bH_j = \bigcap_{j \in 2\mathbb{Z},\ j \leq n} g^{-j} \left( \overset{\circ}{\bR_{w_j}} \cap g^{-1}(\overset{\circ}{\bR_{w_{j+1}}}) \right) = \bigcap_{j \in 2\mathbb{Z},\ j \leq n} g^{-j}(\bH^{w_j}_{j'}).
$$

Therefore,
$$
H_j = \pi_{(f, \cR_f)}([i, w_1, w_2, \dots, w_n]) \quad \text{and} \quad \bH_j = \pi_{(g, \cR_g)}([i, w_1, w_2, \dots, w_n]),
$$
so
$$
\bH_j = \pi_{(g, \cR_g)} \circ \pi_{(f, \cR_f)}^{-1}(H_j),
$$
as claimed.

If $n$ is not even, we can simply write
$$
H_j = \bigcap_{j=0}^{n} f^j(\overset{\circ}{R_{w_j}}) = R_i \cap \bigcap_{j=1}^{n} f^j(\overset{\circ}{R_{w_j}}),
$$
and repeat the previous argument.

The situation for vertical sub-rectangles is treated similarly.

\end{proof}

\begin{lemm}\label{Lemm: Orintacion cada rectangulo}
	The homeomorphism $h:=	[\pi_g] \circ [\pi_f]^{-1}$, when restricted to each rectangle $R_i$, preserves the  orientation of its vertical and horizontal foliations. In particular, $h$ preserves the orientation when is restricted to $R_i$.
\end{lemm}

\begin{proof}
	Let $R_i$ be a rectangle in the Markov partition $\cR_f$ of $f$, and let $\bR_i$ be the corresponding rectangle in the Markov partition $\cR_g$ of $g$. 
	
	Consider the horizontal sub-rectangles $\{H_j\}$ and $\{\bH_j\}$ constructed in Lemma~\ref{Lemm. refinamiento h>2}, and similarly the vertical sub-rectangles $\{V_l\}$ and $\{\bV_l\}$. By the same lemma, the map $h$ preserves this labeling: 
	$$
	h(H_j) = \bH_j \quad \text{and} \quad h(V_l) = \bV_l.
	$$
	
	Now, consider a positively parametrized unstable (vertical) curve $J$ with starting point $x \in \overset{\circ}{H_1}$ and endpoint $y \in \overset{\circ}{H^i_{L}}$. Then $h(J)$ is a curve with starting point $h(x) \in \overset{\circ}{\bH_1}$ and endpoint $h(y) \in \overset{\circ}{\bH_{L}}$. These are different rectangles, one above the other, so $h(J)$ is positively oriented with respect to the vertical direction of $\bR_i$.
	
	Similarly, consider a positively parametrized stable (horizontal) curve $I$ with starting point $x \in \overset{\circ}{V_1}$ and endpoint $y \in \overset{\circ}{V_{J}}$. Then $h(I)$ is a curve with starting point $h(x) \in \overset{\circ}{\bV_1}$ and endpoint $h(y) \in \overset{\circ}{\bV_{J}}$. These are different rectangles, one to the left of the other, so $h(I)$ is positively oriented with respect to the horizontal direction of $\bR_i$.
	
	This implies that $h$ preserves the orientation of both the vertical and horizontal foliations of $R_i$. As these foliations are coherent with the orientation of $S'$, the homeomorphism $h$ preserves the orientation when restricted to $R_i$.
\end{proof}

\begin{lemm}\label{Lemm: h preserving orientation }
The homeomorphism $h:=	[\pi_g] \circ [\pi_f]^{-1}:S \to S'$ is orientation preserving.
\end{lemm}

\begin{proof}
If the rectangles $R_i$ and $R_j$ intersect at a stable boundary point $x$, we can assume their horizontal direction are the same and maybe change the vertical orientation of one of them, in order to keep a orientation coherent with the surface.  Since $h$ preserves the orientation of the pair of  foliations in  $R_i$ and $R_j$, it also preserves the horizontal orientations in the union of these two adjacent rectangles. 
We can continued this analysis until cover all the other rectangles in the partitions.
\end{proof}

\begin{theo*}[\ref{Theo: Total invariant}]
	A pair of pseudo-Anosov homeomorphisms admits geometric Markov partitions with the same geometric type if and only if they are topologically conjugate via an orientation-preserving homeomorphism.
\end{theo*}

\begin{proof}
Is a direct consequence of \ref{Prop: cociente T} and Lemma \ref{Lemm: h preserving orientation }
\end{proof}

We finish with beautiful definition.

\begin{defi}\label{Defi: symbolic model}
	Let $T \in \cG\cT(\textbf{p-A})^{sp}$ be a symbolically presentable geometric type. The \emph{symbolic model} of the geometric type $T$ is the subshift of finite type $(\Sigma_T, \sigma_T)$. Similarly, if $(f, \cR)$ is a realization of $T$, then $(\Sigma_T, \sigma_T)$ is a symbolic model of $f$.
\end{defi}

\subsection*{Acknowledgements}

This work is part of the author's PhD thesis at the \emph{Université de Bourgogne}, supervised by Christian Bonatti, whose guidance and support are deeply appreciated. The research was funded by the \emph{Becas al Extranjero Convenios GOBIERNO FRANCES 2019-1} program (CONACYT). The initial version of this preprint was prepared at the Institute of Mathematics, UNAM, Oaxaca, with support from Lara Bossinger via her \emph{DGAPA-PAPIIT} grant \emph{IA100724}.

\end{document}